\documentclass[a4paper, reqno, 12pt]{amsart}

\usepackage{a4wide}
\usepackage[english]{babel}
\usepackage[normalem]{ ulem }
\usepackage{soul}
\usepackage{graphicx}

\usepackage[T1]{fontenc}
\usepackage[utf8]{inputenc}
\usepackage{amsmath,amsfonts,amssymb,amsopn,amscd,amsthm, mathrsfs}

\usepackage{amsmath}
\usepackage{amstext}
\usepackage{amsfonts}
\usepackage{amsthm}
\usepackage{amssymb}
\usepackage{ulem}
\usepackage{bbm}
\usepackage[colorlinks]{hyperref}
\usepackage{enumerate}

\usepackage{graphicx,color}
\usepackage{xcolor}
\usepackage{mdwlist}
\usepackage{tcolorbox}
\usepackage{fancybox}

\definecolor{cyan(process)}{rgb}{0.0, 0.72, 0.92}
\definecolor{columbiablue}{rgb}{0.61, 0.87, 1.0}
\usepackage{wrapfig}
\usepackage[]{pstricks}
\usepackage{fancyhdr}
\definecolor{sandstone}{HTML}{786D5F}
\definecolor{beaublue}{rgb}{0.74, 0.83, 0.9}
\definecolor{cherryblossompink}{rgb}{1.0, 0.72, 0.77}

\usepackage{tikz}
\usetikzlibrary{backgrounds}
\usetikzlibrary{patterns,fadings}
\usetikzlibrary{arrows,decorations.pathmorphing}
\usetikzlibrary{calc}
\definecolor{light-gray}{gray}{0.95}
\usepackage{graphicx}
\usepackage{epsfig}
\usetikzlibrary{shapes}
\usetikzlibrary{decorations}

\usetikzlibrary{decorations.pathmorphing}
\usetikzlibrary{decorations.pathreplacing}
\usetikzlibrary{decorations.shapes}
\usetikzlibrary{decorations.text}
\usetikzlibrary{decorations.markings}
\usetikzlibrary{decorations.fractals}
\usetikzlibrary{decorations.footprints}

\theoremstyle{plain}

\newtheorem{definition}{Definition}[section]
\newtheorem{theorem}[definition]{Theorem}
\newtheorem{lemma}[definition]{Lemma}
\newtheorem{corollary}[definition]{Corollary}
\newtheorem{prop}[definition]{Proposition}

\theoremstyle{definition}
\newtheorem{remark}[definition]{Remark}

\newcommand{\floor}[1]{\left\lfloor #1 \right\rfloor}
\newcommand{\parent}[1]{\left( #1 \right)}
\newcommand{\ee}[1]{\mathbb{E}_{\nu_{\beta,\lambda}}\left[ #1 \right]}

\newcommand{\R}{\mathbb{R}}
\newcommand{\Z}{\mathbb{Z}}
\newcommand{\N}{\mathbb{N}}

\newcommand{\Sdir}{\mathcal{S}_{\text{Dir}}}

\usepackage{todonotes}

\title[SBE on the line with a Dirichlet BC at the origin]{Derivation of stochastic Burgers on the line with  a Dirichlet boundary condition at the origin}

\author[Bernardin]{\textsc{C\'edric Bernardin}} 
\address{Faculty of Mathematics, National Research University Higher School of Economics -- 6 Usacheva, 119048 Moscow, Russia}
\email{{\tt sedric.bernardin@gmail.com}}

\author[Djurdjevac]{\textsc{Ana Djurdjevac}} 
\address{Institut für Mathematik, Freie Universität Berlin, Arnimallee 6, 14195 Berlin, Germany}
\email{{\tt adjurdjevac@zedat.fu-berlin.de}}

\author[Gon\c{c}alves]{\textsc{Patr\'icia Gon\c{c}alves}} 
\address{Center for Mathematical Analysis, Geometry and Dynamical Systems, Instituto Superior
Técnico, Universidade de Lisboa, Av. Rovisco Pais, 1049-001 Lisboa, Portugal}
\email{{\tt pgoncalves@tecnico.ulisboa.pt}}

\author[Schnee]{\textsc{Leander Schnee}} 
\address{Institut für Mathematik, Freie Universität Berlin, Arnimallee 6, 14195 Berlin, Germany}
\email{{\tt leander.schnee@fu-berlin.de}}

\date{}

\begin{document}
\begin{abstract}
We analyze the \emph{equilibrium fluctuations} of a Hamiltonian chain of oscillators on \(\mathbb{Z}\) with an exponential potential, perturbed by a conservative, symmetric noise. Under the canonical \emph{diffusive scaling} \(t \mapsto t n^2\) and an interaction strength tuned by \(n^{-1/2}\), the fluctuation field is known to converge to the \emph{energy solution} of the stochastic Burgers equation (SBE) on the torus~\cite{ABGS22}. We introduce a \emph{coupled moving heat bath} of strength \(n^{-\delta}\) acting on the particle system. We prove that for \(\delta \leq 1\) (the \emph{strong-coupling regime}), the equilibrium fluctuation field converges to the \emph{energy solution of the SBE with a Dirichlet boundary condition at zero}. We provide two distinct analytical characterizations of these boundary solutions, corresponding to different spaces of test functions. Conversely, for \(\delta > 1\) (the \emph{weak-coupling regime}), the heat bath becomes irrelevant in the scaling limit: the fluctuations converge to the standard SBE on the full line without any boundary condition, reproducing the full-line result of~\cite{GJ14}. Our analysis thus reveals a sharp \emph{critical scaling} in the coupling strength \(\delta\), which dictates the emergence---or absence---of a macroscopic boundary condition from the microscopic perturbation.
\end{abstract}

\medskip

\keywords{}
\renewcommand{\subjclassname}{%
  \textup{2010} Mathematics Subject Classification}
\subjclass[2010]{Primary 60K35; Secondary 60H15, 60F05}
\maketitle

\section{Introduction}
The Kardar--Parisi--Zhang (KPZ) equation---an infinite-dimensional stochastic partial differential equation (SPDE)---occupies a central role in mathematical physics. Introduced almost forty years ago in \cite{KPZ86}, it describes fluctuations of randomly growing interfaces in one-dimensional stochastic dynamics near a stationary state. Since then, it has stimulated intense research activity in both the physics and mathematics communities. The weak KPZ universality conjecture \cite{C12, QS15} posits that fluctuations in a broad class of one-dimensional microscopic interface growth models are governed macroscopically by the solution $\mathcal H$ of the KPZ equation:\begin{align*}
    d\mathcal H_t = \mathfrak{a}\Delta\mathcal H_tdt + \mathfrak{b} (\nabla \mathcal H_t)^2dt + \sqrt{2\mathfrak{c}}  d\mathcal W_t
\end{align*} where $\mathfrak a, \mathfrak c >0$, $\mathfrak b \in\mathbb{R}$, and $\mathcal W$ is a space-time white-noise. Equivalently, taking the gradient $\mathcal Y = \nabla \mathcal H$ yields the stochastic Burgers equation (SBE) for $\mathcal Y$
\begin{align*}
    d\mathcal{Y}_t = \mathfrak{a}\Delta\mathcal{Y}_tdt + \mathfrak{b} \nabla(\mathcal{Y}_t^2)dt + \sqrt{2\mathfrak{c}} \nabla d\mathcal W_t.
\end{align*}
SBE describes fluctuations of conserved quantities in one-dimensional interacting particle systems\footnote{For systems with multiple conserved quantities, the SBE generalizes to coupled/uncoupled SBEs \cite{BFS21, BF22,H24} or other stochastic SPDEs \cite{GH23, CGMO25}.}.

\bigskip

Significant progress regarding the weak KPZ universality conjecture has emerged in recent years:\\

\begin{itemize}
    \item[1.] Hairer's regularity structures \cite{H14} provide a framework for analyzing singular SPDEs, including the KPZ and SBE equations. A parallel approach, based on rough path theory and paracontrolled distributions, offers another general methodology to study singular SPDEs \cite{GP17}.\\
    \item[2.] The concept of energy solution for the SBE \cite{GJ14} enabled its derivation as a scaling limit for a large class of interacting particle systems \cite{GJS15}. Uniqueness of energy solution was resolved by Gubinelli and Perkowski \cite{GP17}, establishing a robust method to derive the SBE as a scaling limit. This approach, however, remains restricted to systems in their stationary states. More recently, \cite{GPP24} proposed a new concept of solutions of singular SPDE's with general boundary conditions (including Dirichlet on $(0,\infty)$). As we will see later, this new notion of solutions coincides in fact with the SBE energy solution we derive from the particle system. 
\end{itemize}

\bigskip

Let us also mention that stochastic integrability techniques applied to very specific models have generated profound results for the weak KPZ conjecture and beyond, e.g. the strong KPZ conjecture. As the KPZ equation is universal, these findings are conjectured to extend to broader microscopic systems. Nevertheless, deriving analogous results for generic models remains challenging due to the limited applicability of stochastic integrability methods.

\bigskip

Recent years have seen growing interest in studying and deriving the SBE, or other singular SPDEs, with various boundary conditions \cite{CS18, GH18}. In particular, \cite{GPS20} shows  that the stationary density fluctuations of the weakly asymmetric exclusion process are macroscopically governed by the energy solution of the SBE with homogeneous Dirichlet boundary conditions on the interval $[0,1]$, and establishes the uniqueness of this solution.

\bigskip 
In this paper, we focus on a particular interacting particle system $(\xi(t))_{t \ge 0}:=\{\xi_x (t) \; ; \; x\in\mathbb Z, \; t\ge 0\} \in (0,\infty)^{\mathbb Z}$ and study the equilibrium fluctuations of one of its conserved quantities, called \textit{energy}. The stationary Gibbs states of the model are given by explicit product probability measures $\nu_{\beta, \lambda}$, where $\beta>0$ and $\lambda>-1$, taking the form 
\begin{align} 
\label{def:nu}
    \nu_{\beta,\lambda} : = \prod_{x\in\mathbb{Z}} Z^{-1} (\beta, \lambda)\;  {\mathbf 1}_{\xi_x>0} \exp (-W_{\beta,\lambda}(\xi_x)) d\xi_x \ .
\end{align}
Here the function $W_{\beta, \lambda}:(0,\infty)\to (0,\infty)$ is given by $W_{\beta,\lambda}(u):= \beta u -\lambda \log u $ and $ Z(\beta, \lambda)$ is the partition function $Z(\beta, \lambda) = \int_0^{\infty} \exp (-W_{\beta, \lambda} (u) ) du$. The system is investigated in its stationary state, i.e., the law of $\xi(0)$ is given by $\nu_{\beta, \lambda}$. 

\bigskip 

The model is defined as follows. Fix $\beta>0$ and $\lambda>-1$ and let $n\in\N$ be the scaling parameter, then the system of interest will be described in terms of its time dependent infinitesimal generator 
\begin{align} 
\label{def:L_nt}
    L_{n,t} := \alpha n^{-\kappa}A + \gamma S + {n^{-\delta}}B_{n,t}
\end{align} 
where the parameters $\alpha\in\R$ and $\gamma> 0$ and the scaling exponents satisfy  $\kappa\geq 1/2$ and $\delta\in\R$. The generator $L_{n,t}$ consists of three parts, a deterministic conservative flow $A$ related to a chain of  non-linear coupled oscillators, a conservative noise $S$ and a moving heat bath $B_{n,t}$ which acts on only one element of the chain.

\bigskip

The operator $A$ acts on differentiable local {\footnote{The function $f$ is local if it depends on finitely many coordinates of $\xi\in(0,\infty)^{\mathbb Z}$.}} functions $f:(0,\infty)^{\Z}\to \R$ as
\begin{align*}
    \forall \xi\in (0,\infty)^{\mathbb Z}, \quad (Af)(\xi) :=\ \sum_{x\in\mathbb{Z}} \xi_x( \xi_{x+1}- \xi_{x-1})  \  (  \partial_{\xi_x} f )(\xi) \ .
\end{align*}
The conservative noise operator $S$ acts on local functions $f:\xi \in (0,\infty)^{\Z}\to f(\xi) \in \R$ as
\begin{align*}
    \forall \xi\in (0,\infty)^{\mathbb Z}, \quad (Sf)(\xi) := \sum_{x\in{\mathbb Z}} \left[ f(\xi^{x,x+1}) -f(\xi)\right]
\end{align*}
where $\xi^{x,x+1}$ is the configuration obtained from $\xi$ by exchanging the values $\xi_x$ and $\xi_{x+1}$ between the sites $x$ and $x+1$.

Finally the moving heat bath operator acts on differentiable local functions $f:(0,\infty)^{\Z}\to \R$ as a local Langevin dynamics with potential $W_{\beta,\lambda}$, namely: 
\begin{align*}
    (B_{n,t}f)(\xi) := -W_{\beta,\lambda}'(\xi_{x_n(t)}) (\partial_{\xi_{x_n(t)}(t)}f)(\xi) + (\partial_{\xi_{x_n(t)}(t)}^2f)(\xi)
\end{align*}
where $x_n(t):= -\lfloor c_nt\rfloor \in \mathbb Z$ is a time dependent moving index of the element affected by the heat bath. This index $x_n(t)$ is the point in $\Z$ which is the floor of the shift at a constant velocity $c_n\in\R$, which will be determined later (see \eqref{eq:c_n}).

\bigskip
Without the heat bath, this model is a particular case of the models introduced in \cite{BS12}. It has exactly two conserved quantities: the volume $\sum_{x\in\mathbb Z} \log \xi_x$ and the energy $\sum_{x\in\mathbb Z} \xi_{x}$ (see \cite{BS12} for a precise statement of the notion of conserved quantity). It has been proved in \cite{BG14} that when $\kappa=0$, the model is super-diffusive and it is expected that in the scaling limit the fluctuations of the energy are given by the derivative of the so-called KPZ fixed point, constructed in \cite{MQR21}, while the fluctuations of the volume are given by the solution of some infinite dimensional Ornstein-Uhlenbeck equation (OUE), hence Gaussian \cite{SS15}. The model has been investigated rigorously when $\kappa\ge 1/2$ and \cite{ABGS22} provides a detailed analysis of the fluctuations of the two conserved quantities. In particular for $\kappa=1/2$, it is proved there that the fluctuations of the energy are governed by the SBE, while they are driven by the OUE when $\kappa>1/2$.    

\bigskip

In this paper we add a moving thermal heat bath described by the operator $B_{n,t}$, to the previous dynamics, in order to incorporate a boundary condition to the SBE. The intensity of the bath is regulated by $n^{-\delta}$. We will prove that, depending on the value of $\delta \in \mathbb R$ and $\kappa\ge 1/2$ the scaling limit of the fluctuation field of $\xi$ provides a SBE or an OUE on the infinite domain $\mathbb R$ with various boundary conditions at $0$. We also observe that we did not present any result for the other conserved quantity, namely the \textit{volume} since we do not expect any technical difficulty in showing that its stationary fluctuations evolve according to an OUE, and depending on the regime of $\delta$ the process might be supplement with Dirichlet boundary conditions.

\bigskip 

This paper makes two primary contributions. First, it establishes---to our knowledge---the initial rigorous derivation of stationary stochastic boundary-driven equations with Dirichlet boundary conditions in infinite volume as a scaling limit of an interacting particle system. Second, like in \cite{CG25} we employ the new ``energy solution'' framework introduced in \cite{GPP24} as a core technical tool for obtaining a complete scaling limit derivation.

\bigskip 

The choice of interacting particle system merits clarification. While alternatives like weakly asymmetric exclusion processes (with static or dynamic reservoirs) could yield similar results, our selected model possesses \textit{multiple conserved quantities} (energy and volume). This feature not only distinguishes it from single-conservation-law systems but also aligns with our long-term objective: to derive fluctuation results for both conserved quantities in the presence of thermal baths, extending prior work \cite{ABGS22} developed without reservoirs.

\bigskip

\textbf{Outline of the paper.} In Section \ref{sec:modelandmainresults} we present the model under investigation, how it is derived and the main convergence result proved in this paper. Section \ref{sec:existenceandtightness} proves tightness of the fluctuation field and thus proves the existence of the limiting points which are characterised to be energy solutions to the corresponding SPDEs. Finally, in Section \ref{sec:uniqueness} we show the applicability of the uniqueness result of \cite{GPP24} to finish the proof of convergence. In the appendix, we prove the needed technical tools to treat our time-inhomogeneous model.

\bigskip
\textbf{Notation.} 
Given two real valued functions $f$ and $g$ we will write hereinafter $f(u) \lesssim g(u)$ if there exists a constant $C$ independent of $u$ such that $f(u) \le C g(u)$ for every $u$; moreover, we will write $f(u) = \mathcal{O} (g(u) )$ if the condition $|f (u) | \lesssim |g(u) |$ is satisfied. \\
For an open set $D\subseteq \R$ and a Banach space $E$ we denote by $C(D, E)$ the set of continuous functions $H:D\to E$, simplifying the notation when $E =\R$ to $C(D):= C(D,\R)$. We denote by $C^{\infty}(D)$ the set of  functions $H:D\to\R$ that for all $n\in\N$ they are $n$-times continuously differentiable.\\
For a measurable space $(\Omega , \Sigma, \nu)$ we  write $L^2(\nu)$ to denote the set of measurable functions $f:\Omega \to \R$ that are square integrable.
This means that the $L^2(\nu)$-norm denoted $\|f\|_{\nu}:=\left( \int f^2 d\nu\right)^{\frac{1}{2}}$ is finite. When $\Omega \subseteq \R^d$ we always choose $\nu$ to be the Lebesgue-measure on $\R^d$. We also denote the space of square integrable functions $f:D\to\R$ as $L^2(D)$ (instead of $L^2(\nu)$). For the dual coupling between two function $f,g\in L^2(\nu)$ we will write $\langle f,g\rangle_{\nu}:= \int fg \ d\nu$ or also $\langle f,g\rangle_{L^2}$ when it clear in which space and measure we refer to. 
Finally, for a topological vector space $\mathcal C$ we define $\mathcal C'$ to be its topological dual.

\section{The model and main results} \label{sec:modelandmainresults}

\subsection{Description of the model}

In this section, we define rigorously the particle system mentioned in the introduction.

Let $\{ {N}_{x,x+1} \; \mid \; x \in \mathbb Z\}$ be a collection of standard independent Poisson processes and $(B(t))_{t \ge 0}$ a standard Brownian motion. We consider the infinite-dimensional stochastic dynamics defined by 
\begin{equation}
\label{eq:app-stoch-eq}
\begin{split}
\forall x \in \mathbb Z, \quad {d \xi_x} &=  \alpha n^{-\kappa} \xi_x (\xi_{x+1} -\xi_{x-1}) dt + \gamma \nabla \left[ (\xi_{x} -\xi_{x-1}) dN_{x-1,x} (t)\right]\\
&+ n^{-\delta} {\mathbf 1}_{x=x_n (t)} \left[ -W^{\prime}_{\beta, \lambda} (\xi_x) dt + 2 dB (t) \right]  \ .
\end{split}
\end{equation}
The formal generator associated to this dynamics is given by $L_{n,t}$ introduced in \eqref{def:L_nt}. By adapting the proof presented in Appendix A of \cite{BG14}, we can prove the existence of the infinite volume dynamics $(\xi(t))_{t\ge 0}$ in the following sense.
\bigskip

Let $\mathcal D(\mathbb R_+, \mathbb R)$ be the Skorokhod space equipped with the Skorokhod topology and let ${\mathcal  D}= [\mathcal D(\mathbb R_+, \mathbb R)]^{\mathbb Z}$ equipped with the product topology and the associated Borel field ${\mathcal B}$. The smallest $\sigma$-algebra on which all projections restricted to the time interval $[0,t]$ are measurable will be denoted by ${\mathcal B}_t$. Finally, suppose that we are given a probability measure ${\mathbf P}$ on ${\mathcal  B}$ such that the Poisson processes $\{ N_{x,x+1} \; ; \; x \in \mathbb Z\}$ and the Brownian motion $B$ are realized.

\begin{prop} \label{prop:markovwellposed}
There exists a measurable subset $\Omega \subset (0, \infty)^{\mathbb Z}$ such that for any $\beta>0, \lambda>-1$, $\nu_{\beta, \lambda} (\Omega)=1$, where $\nu_{\beta, \lambda}$ was defined in \eqref{def:nu}. For any initial condition $\zeta_0 \in \Omega$, there exists a unique ${\mathcal B}_t$-adapted time-inhomogeneous Markov process $(\xi(t))_{t \ge 0}$  living in $\mathcal D $ such that  $\xi (0)=\zeta_0$, for any time $t\ge 0$, $\xi(t) \in \Omega$ and $(\xi(t))_{t \ge 0}$ satisfies the integral form of \eqref{eq:app-stoch-eq}. Moreover, the probability measures $\{ \nu_{\beta, \lambda} \; ; \; \beta>0, \lambda>-1\}$ are all invariant for the process.  Its generator is a closable extension of ${L_{n,t}}$ in $ L^2 (\nu_{\beta, \lambda})$, as defined on the set of functions $f: (t, \xi) \in [0, \infty) \times \Omega \to f(t, \xi) \in \mathbb R$ which are time differentiable with compact support, and local, compactly supported, and continuously differentiable in $\xi$.
\end{prop}

The fact that $\nu_{\beta, \lambda}$ is invariant for the operator $\alpha n^{-\kappa} A+\gamma S$ has already been proved in \cite{BG14} and the fact that it is invariant for $B_{n,t}$ follows from the fact that this operator can be written as
\begin{equation*}
(B_{n,t} f )(\xi) = e^{W_{\beta,\lambda} (\xi_{x_n(t)})} \partial_{\xi_{x_n(t)}} \left( e^{-W_{\beta,\lambda} (\xi_{x_n(t)})}\ \partial_{\xi_{x_n(t)}} f  \ (\xi_{x_n(t)}) \right) \ ,
\end{equation*} 
so that it is symmetric in $L^2(\nu_{\beta,\lambda})$.
In the following we will use $\mathbb P_{\beta,\lambda}$ and $\mathbb E_{\nu_{\beta,\lambda}}$ to denote the probability measure and its expectation on the paths in $\mathcal D$ with $(\xi(t))_{t\geq 0}$ starting from $\nu_{\beta,\lambda}$.

\subsection{Space of test functions}
\label{subsec:testfunctionsspace}
We now define several spaces of test functions in which the fluctuation field will act. These spaces will play a crucial role in the definition of the SPDEs characterising the scaling limit of the equilibrium fluctuations and encapsulate in fact boundary conditions of the fluctuations at $0$. 

To that end, we start by defining the following family of norms. Given $k\in\N_0$ and a function $H\in C^{\infty}(\R)$ we denote the $k$-th norm{\footnote{We take the supremum over $\R\setminus\{0\}$ for the later case where derivatives at 0 do not exist. By continuity this coincides with the supremum over $\R$ for $H\in C^{\infty}(\R)$.}}
of $H$ by
\begin{align*}
    \|H\|_{\infty,k} := \sup_{0\leq i,j\leq k} \;  \sup_{u\in\R \backslash\{ 0\}} \;  \left\{  (1+|u|)^i|H^{(j)}(u)| \right\} \ ,
\end{align*}
where above and in what follows, for $j\in\mathbb N$ and a function $H$, the notation $H^{(j)}(u)$ is for the $j$-th derivative of $H$ at $u$. 
We first introduce the Schwartz's space $\mathcal{S}$ defined by 
\begin{align*}
    \mathcal{S} = \left\{H\in C^{\infty}(\R)\; ; \; \forall k\in\N_0, \ \|H\|_{\infty,k}<\infty \right\} \ .
\end{align*}
We then denote $ \mathcal{S}_0 \subset  \mathcal{S}$ the set
\begin{align*}
    \mathcal{S}_0 := \{H\in\mathcal{S}\; ; \; \forall k \in \mathbb N_0, \ H^{(k)}(0) = 0 \} \ ,
\end{align*}
and, finally define
\begin{equation}
    \Sdir  := \{H\in \mathcal C(\R) \; ; \;  (H|_{\R_+},H|_{\R_-})\in\mathcal{S}_{\text{Dir},+}  \times \mathcal{S}_{\text{Dir},-}  \}
\end{equation}
the space of functions decaying like Schwartz functions but only required to fulfill boundary condition at $0$ for even derivatives, with 
\begin{equation*}
\begin{split}
    \mathcal{S}_{\text{Dir},\pm}  := \big\{H\in \mathcal C(\R_\pm) \; ; \; &\exists\tilde{H}\in\mathcal S\quad \textrm{s.t.} \quad  \tilde{H}|_{\R_\pm} = H   \\
    &  \quad \text{ and } \quad  \tilde H^{(2k)}(0) = 0 \quad  \text{for all} \quad  k\in\N_0 \big\}\ .
\end{split}
\end{equation*}

As usual, Schwartz's space $\mathcal S $ is equipped with the topology induced by the norms $\left\{ \Vert \cdot \Vert_{\infty, k} \;  \mid \;  k\in \mathbb N_0 \right\}$. It is straightforward to see that $\mathcal S_0 $ is a closed subset of $\mathcal S $ and we equip the later with the induced topology. On the other hand, $\Sdir $ is not a subset of $\mathcal S $ and we have to precise its topology. The topology of $\Sdir $ is the topology generated by the family of norms $\left\{\|\cdot\|_{\infty,k}  \; ; \; k \in \mathbb N_0 \right\}$.

\bigskip

Consider a function $H\in\Sdir $. Since all the derivatives of $H$ are continuous and bounded on $(0,+\infty)$ (resp. $(-\infty, 0)$), they are also uniformly continuous and thus can be continuously extended to $0$. For any $j \in \mathbb N_0$, we denote by $H^{(j)}(0^+)$ (resp. $H^{(j)}(0^-)$) these extensions of $H^{(j)}\vert_{(0,+\infty)}$ ( resp. $H^{(j)}\vert_{(-\infty,0)}$). Observe that $ H^{(j)}(0^+)=H^{(j)}(0^-) = 0$ if $j \in \mathbb N_0$ is even but that it can happen  that $ H^{(j)}(0^+)\ne H^{(j)}(0^-)$ if $j \in \mathbb N_0$ is odd.

For convenience, we introduce the  following notation depending on the scaling parameter $\delta\in\R$:
\begin{align*}
    \mathcal S_{\delta}:=\left\{ \begin{array}{ll}
       \mathcal S  & \text{ for } \delta\in (1,\infty) \\
        \Sdir   & \text{ for } \delta\in (-1,1] \\
        \mathcal S_0   & \text{ for } \delta\in (-\infty,-1]
    \end{array}
    \right..
\end{align*}

\begin{remark} 
\label{rem:nuclear}
For all $\delta\in\R$ we have that $\mathcal S_{\delta}$ is a nuclear Fréchet space and this property will be essential to prove our main result.
For the usual Schwartz space $\mathcal S$ this is a well known fact, which can be found in \cite[Chapter 3, Section 7]{S71} along with many characterisations of nuclear Fréchet spaces.
For $\mathcal{S}_0$, the fact that it is a closed subspace of $\mathcal S$ immediatly gives us that it also is a nuclear Fréchet space.\\
For $\Sdir$ it is useful to homeomorphically identify it with elements in $\mathcal S\times\mathcal S$ endowed with the product topology of $\mathcal S$, by mapping $ H\mapsto (H_+,H_-) \mapsto (\tilde H_+, \tilde H_-)$ where $(\tilde H_+, \tilde H_-)\in\mathcal S\times\mathcal S$ can be chosen as the unique even extension of $(H_+,H_-)\in\mathcal S_{\text{Dir},+}\times\mathcal S_{\text{Dir},-}$. Now because this identifies $\Sdir$ to a closed subspace in $\mathcal S\times\mathcal S$, which is a nuclear Fréchet space as a product as a product of $\mathcal S$, we deduce that $\Sdir$ must also be a nuclear Fréchet space.
Note, that these spaces are treated in a similar way in \cite{CGJ23}.
\end{remark}

\subsection{Fluctuation field}

In this section, we introduce the fluctuation field associated to the energy.

For a fixed $\beta>0$ and $\lambda>-1$ we denote the expectation and variance of any coordinate $\xi_x$, $x\in\Z$ of the stationary process by 
\begin{align}\label{eq:means}
    \rho:=\mathbb{E}_{\nu_{\beta,\lambda}}[\xi_x(0)] = \frac{\lambda+1}{\beta}\ \text{ and }\  \sigma^2 := \mathbb{V}_{\nu_{\beta,\lambda}}[\xi_x(0)] =\frac{\lambda +1 }{\beta ^2}.
\end{align}
To simplify we will also use the notation $\bar\xi_x = \xi_x-\rho$. Note that because for any $s,t\geq 0: \xi_x(t)\sim\xi_x(s)$, we might occasionally drop the time index when it is clear that we compute probabilities in $\R^{\Z}$.

\begin{definition}[Fluctuation Field]
\label{def:notation}
        For given $n\in\N$, $\beta>0, \lambda>-1$ and a function $H:\R\to\R$ with $\sup_{u\in\R}(1+u^2)H(u)^2 <\infty$,
        we define the fluctuation field in the diffusive time scale $tn^2$, as the linear functional acting on $H$ as
    \begin{align*}
        \mathcal{Y}^n_t(H) := \frac{1}{\sqrt{n}}\sum_{x\in\mathbb{Z}} T_{c_n t} H\parent{\tfrac{x}{n}} \bar{\xi}_x(tn^2)
    \end{align*}
    where $T$ is the shift operator such that \begin{equation}
    \label{eq:semigroup}
    T_{c_n t}H(u)=H(u+c_n t) 
    \end{equation}
    and 
    \begin{equation}
        \label{eq:c_n}
    c_n :=  2\alpha\rho \, n^{2-\kappa}
    \end{equation}
    is the velocity of the  moving frame. Ahead it will be clear the reason for this particular choice of the velocity of the moving frame. Finally we also define the following family of semi-norms in $L^2(\R)$ for $n\in\N:$
    \begin{align} \label{eq:discreteL2}
        \|H\|_{2,n}^2:= \frac{1}{n}\sum_{x\in\mathbb{Z}} \;  H^2\big(\tfrac{x}{n}\big).
    \end{align}
\end{definition}

\begin{remark} \label{rem:fluctuationL2bound} 
Let us note that:
\begin{itemize}
\item[1.] If $\sup_{u\in\R}(1+u^2)H(u)^2 <\infty$ then the random variable $\mathcal{Y}^n_t(H)$ makes sense as a random variable in $ L^2 (\nu_{\beta, \lambda})$ because from the independence of $(\bar{\xi}_x)_{x \in \mathbb Z}$ under $\nu_{\beta, \lambda}$, we have 
\begin{equation}
\label{eq:bound_HY}
    {\mathbb E}_{\nu_{\beta, \lambda}} \left[ \left( \mathcal{Y}^n_t(H)\right)^2\right] = \sigma^2 \|T_{c_nt} H\|_{2,n}^2 \leq \sigma^2\left(1+ \frac{\pi^2}{3}\right) \sup_{u\in\R}(1+u^2)H(u)^2,
\end{equation}
where the last bound will be shown in Lemma \ref{lem:L^2-normdiscrete}. 
\item[2.] Equation \eqref{eq:bound_HY} and Lemma \ref{lem:L^2-normdiscrete} prove that there exists a universal constant $C_0$ such that for any $H\in {\mathcal S}_\delta$, 
\begin{equation*}
   {\mathbb E}_{\nu_{\beta, \lambda}} \left[ \cfrac{\left( \mathcal{Y}^n_t(H)\right)^2}{\Vert H \Vert_{\infty, 2}^2}\right] \le C_0 \ . 
\end{equation*}
Hence, for any time $t\ge 0$, there exists an almost surely bounded random variable $C$ such that $\left\vert \mathcal{Y}^n_t(H)\right\vert \le C \Vert H \Vert_{\infty, 2}$, so that $\mathcal Y_t^n$ belongs to $\mathcal S^\prime_\delta$, the topological dual of $\mathcal S_\delta$. Moreover, since  the trajectories of the process $\xi$ are c\`adl\`ag, in fact that, we have, for any horizon time $T>0$, that   $(\mathcal Y_t^n)_{t \in [0,T]} \in \mathcal D ([0,T], {\mathcal S}_\delta^\prime)$. 
\end{itemize}
\end{remark}

\subsection{Stochastic Partial Differential Equations}\label{sec:spde}

 In this section we want to define rigorously the notion of stationary energy solution to the SBE (and OUE) with free and Dirichlet 
boundary conditions at $0$ on the real line, accordingly we will use the domains $D= \R$ or $D=\R\setminus\{0\}$. The SPDE will describe the limit of $(\mathcal  Y^n)_{n\in\N}$. 
The SBE (OUE if $\mathfrak b =0$) is formally written as
\begin{align} 
\tag{SBE}
\label{eq:SBE}
    \partial_t\mathcal{Y}_t = \mathfrak{a}\Delta\mathcal{Y}_t + \mathfrak{b} \nabla(\mathcal{Y}_t^2) + \sqrt{2\mathfrak{c}} \nabla {\mathcal W}_t
\end{align}
where $\mathfrak a, \mathfrak c >0$, $\mathfrak b \in\mathbb{R}$, and $\mathcal W$ is a space-time white-noise over $L^2(D)$. \\
We recall that a space-time white-noise over $L^2(D)$ on the probability space $(\Omega,\mathcal F,\mathbb P)$ is a centred Gaussian process indexed by the elements of $L^2 ([0,T] \; ; \;  L^2 (D))$ and satisfying 
\begin{equation*}
    \forall H,G \in L^2 ([0,T] \; ; \;  L^2 (D)), \quad \mathbb E ( \mathcal W (H) \mathcal W (G)) = \int_0^T \langle H_t , G_t \rangle_{L^2} \; dt \ .
\end{equation*}
Furthermore we will consider the probability measure $\mu$ for which, for any $t\ge 0$, $\mathcal Y_t$ will be a stationary space white-noise with covariance $\mathfrak c/\mathfrak a$, i.e. a Gaussian process indexed by functions in $L^2(D)$, with covariance 
\begin{align} \label{eq:spacewn}
    \mathbb{E}_{\mu} [\mathcal Y_t(G)\mathcal Y_t(H)] = \mathfrak c/\mathfrak a \langle G,H\rangle_{L^2} 
\end{align}
for $G,H\in L^2(D)$. In other words $\mathcal Y_t$ can be seen as a orthogonal isomorphism from $L^2(D) \to L^2(\mu)$.

\medskip

A stationary energy solution of the \eqref{eq:SBE} will be defined as a distribution valued process solving a martingale problem and an energy estimate. Incorporating the boundary condition at site $0$ is a subtle problem and requires to consider the distribution process as acting on the sets of test functions  $\mathcal S_\delta$ introduced in Subsection \ref{subsec:testfunctionsspace}.

\begin{definition}[Energy solutions] \label{def:SBEenergysolutions} 
    Let $\mathcal{C}\subseteq L^2(\R)$ be a topological vector space of test functions. A stochastic process $(\mathcal{Y}_t)_{t\in[0,T]}$ in $\mathcal D([0,T],\mathcal C')$ is a stationary energy solution of SBE($\mathcal{C}$) with parameters $\mathfrak a, \mathfrak b$ and $\mathfrak c$ if the following properties hold
\begin{itemize}
    \item[i)] for every $t\in [0,T]$: $\mathcal{Y}_t$ is a $\mathcal{C}'$-valued {\footnote{Take the definition as in \eqref{eq:spacewn}, but restrict the test functions to elements of $\mathcal C$.}} white-noise with covariance $\mathfrak c/\mathfrak a$;
    \item[ii)] there exists a constant $K>0$ such that for any $H \in \mathcal{C}$ and $0<\delta<\varepsilon<1$, we have the energy estimate:
    \begin{align*}
        \mathbb{E}\left[ \left(\int_s^t\int_{\mathbb{R}} \left\{\left[\mathcal{Y}_r(\iota_{\varepsilon}^u)\right]^2 - \left[\mathcal{Y}_r(\iota_{\delta}^u)\right]^2\right\} \, H ' (u)\ du\ dr\right)^2 \right] \leq K \| H' \|^2_{L^2} (t-s) \; \varepsilon, 
    \end{align*}
    where $\iota^u_{\varepsilon}: v \in \R\to \iota^u_{\varepsilon} (v) \in  \mathbb{R}$ is an approximation of the identity defined by
    \begin{align}
    \label{eq:iota}
        \iota^u_{\varepsilon} (v) := \varepsilon^{-1}\mathbf{1}_{(u,u+\varepsilon]}(v);
    \end{align}
    \item[iii)] for $H\in\mathcal{C}$,
    \begin{align} 
    \label{eq:SBEmartingale}
        \mathcal{Y}_t(H) -\mathcal{Y}_0(H) -\mathfrak a\int_0^t \mathcal{Y}_s(H'') ds + \mathfrak b \mathcal{Q}_t(H)
    \end{align}
    is a continuous $\R$-valued martingale with predictable quadratic variation $2\mathfrak c t\| H'\|_{L^2}^2$, where 
    \begin{align}
    \label{eq:QSBE}
        \mathcal{Q}_t(H) := \lim_{\varepsilon\to 0}\int_0^t\int_{\mathbb{R}}\left(\mathcal{Y}_r(\iota_{\varepsilon}^u)\right)^2 H'(u)\ du\ dr;
    \end{align}
    \item[iv)]  the reversed process $(\mathcal{Y}_{T-t})_{t\in[0,T]}$ satisfies \textit{iii)} with the nonlinear term $\mathfrak b\mathcal Q_t$ replaced by $-\mathfrak b(\mathcal{Q}_T -\mathcal Q_{T-t})$.
\end{itemize}
\smallskip 

In the case $\mathfrak b=0$, and only $i)$ and $iii)$ are relevant, we call $(\mathcal{Y}_t)_{t\in[0,T]}$ an Ornstein-Uhlenbeck process. Energy solutions of this type will be denoted by OU$(\mathcal{C},\mathfrak c/\mathfrak a)$. When $\mathcal C = \Sdir$ we say that $\mathcal Y$ satisfies Dirichlet boundary conditions. 
\end{definition}

\begin{remark} 
\label{rem:sbenotation}

\phantom{a}

\begin{itemize} 
    \item[1.]In Section \ref{sec:uniqueness}, we will use the convenient notation SBE($(H_n)_{n\in\N}$) or SBE($H_1$) for some functions $(H_n)_{n\in\N}\subset L^2(\R)$. So in the case $\mathcal C\subseteq L^2(\R)$ is not a topological vector space (when $\mathcal C'$ is not well-defined), by solution $\mathcal Y$ to SBE($\mathcal C$) we mean that for all $H\in\mathcal C: \mathcal Y(H)$ is in $ C([0,T],\R)$, with continuity uniform in $\mathcal C$. Also $\mathcal Y$ satisfies \eqref{eq:spacewn} for all $G,H\in\mathcal C$ instead of $i)$ and satisfies $ii)$, $iii)$, $iv)$ the same way as in Definition \ref{def:SBEenergysolutions}.
    \item[2.] Note that if $ \mathcal C_1\subseteq \mathcal C_2$ then a solution of SBE$(\mathcal C_2)$ is also a solution of SBE($\mathcal C_1$).
\end{itemize}
\end{remark}

\begin{remark}\label{rem:continuity&whitenoise}
From Definition \ref{def:SBEenergysolutions}  a solution $\mathcal Y$ of SBE($\mathcal C$) fulfils the following properties:
\begin{itemize}
    \item[1.] If $\mathcal C$ is dense in $L^2(\R)$, then a $\mathcal{C'}$-valued white noise $\mathcal Y_t$ can be uniquely extended to a white noise on $L^2(\R)$. For any $H\in L^2(\R)$ we have a sequence $(H_n)_{n\in\N}\subset \mathcal C$ converging $H_n\to H$ in $L^2$, then by \eqref{eq:spacewn} the sequence $(\mathcal Y_t(H_n))_{n\in\N}$ is a Cauchy in $L^2(\mu)$. By slight abuse of notation we can define the extension $\mathcal Y_t(H)\in L^2(\mu)$ as the limit of $(\mathcal Y_t(H_n))_{n\in\N}$, and find that $\mathcal Y_t$ is a white noise on $L^2(\R)$.
    \item[2.] From the properties $i),ii)$ and $iii)$ it follows that $\mathcal Y_t\in C([0,T],\mathcal C')$. The proof of this fact would follow the exact same steps as the proof for \cite[Theorem 2]{GJ14}.
\end{itemize}
\end{remark}

In the case $\mathcal C = \mathcal S_0$ we will need to require further conditions to get well-posedness.
\begin{definition}
\label{def:extradirichlet}
    Let $\mathcal Y$ be a solution of $SBE(\mathcal S_0)$. Then we define the following boundary condition:
\begin{itemize}
    \item[(BC)] the following limit holds 
    \begin{align*}
        \lim_{\varepsilon \to 0} \mathbb E\left[\sup_{t\in[0,T]}\left( \int_0^t \mathcal Y_s(\iota^0_{\varepsilon}) ds \right)^2\right] =
        0 = \lim_{\varepsilon \to 0} \mathbb E\left[\sup_{t\in[0,T]}\left( \int_0^t \mathcal Y_s(\iota^0_{\varepsilon}(-\ \cdot)) ds \right)^2\right] .
    \end{align*}
\end{itemize}
If $\mathcal Y$ satisfies this property we also say that it satisfies a Dirichlet boundary condition.
\end{definition}

\subsection{Main results}
The main result of this paper is the convergence of the fluctuation field $\mathcal Y^n$ as $n$ goes to infinity. The limit $\mathcal Y$ will be an energy solution of \eqref{eq:SBE} where the corresponding space of test functions $\mathcal S_{\delta}$  will depend on the scaling parameter $\delta$ of the heat bath. 
For $\delta <1$ the effect of the heat bath imposes a boundary condition on the limiting SPDE either in the form of the test function $\Sdir$ or more explicitly in the form of (BC).

\begin{theorem} 
\label{thm:mainresult}
    For every $n\in\N$, let $(\xi(t))_{t \ge 0}:=\{\xi_x (t) \; ; \; x\in\mathbb Z, \; t\ge 0\}$ be the Markov process evolving according to the infinitesimal generator $L_{n,t} = \alpha n^{-\kappa} A+\gamma S+n^{-\delta}B_{n,t}$ and starting from the stationary probability measure $\nu_{\beta,\lambda}$. Then,  $(\mathcal Y^n)_{n\in\mathbb N}$ converges in distribution on $\mathcal D([0,T],\mathcal S_{\delta}')$, when $\kappa = \tfrac{1}{2}$ its limit $\mathcal Y$ is a stationary energy solution of SBE($\mathcal S_{\delta}$) with parameters $\mathfrak a= \gamma, \mathfrak b = \alpha, \mathfrak c = \gamma\sigma$ where in the fast boundary case $\delta < 1$ additional the condition (BC) is fulfilled. When $\kappa>\tfrac{1}{2}$ we have $\mathfrak b =0$ i.e. the limit $\mathcal Y$ is a solution OU($\mathcal S_{\delta},\sigma$) while also fulfilling (BC) when $\delta< 1$.  
\end{theorem}

See Figure \ref{Fig:MT} for a graphical illustration of this theorem.

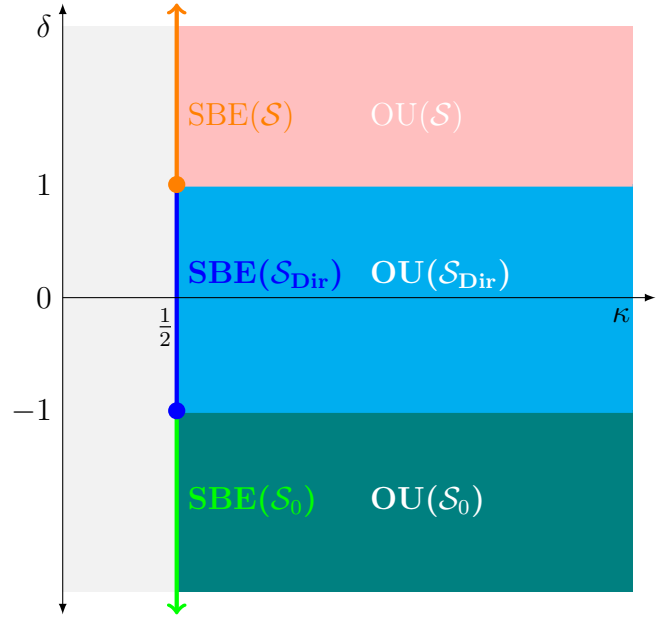
\begin{figure}[h!]
\begin{center}
\begin{tikzpicture}[scale=0.3]
\fill[light-gray] (5,25) -- (5,0) -- (0,0) -- (0,25) -- cycle;
\fill[pink] (5,25) -- (25,25) -- (25,18) -- (5,18) -- cycle;
\fill[cyan] (5,18) -- (25,18) -- (25,8) -- (5,8) -- cycle;
\fill[teal] (5,8) -- (25,8) -- (25,0) -- (5,0) -- cycle;
\draw[-,=latex,pink,ultra thick] (5,18) -- (25, 18);
\draw[-,=latex,cyan,ultra thick] (5,8) -- (25,8);
\draw[->,=latex,orange,ultra thick] (5,18) -- (5,26);
\draw[-,=latex,blue,ultra thick] (5,8) -- (5,18);
\draw[<-,=latex,green,ultra thick] (5,-1) -- (5,8);
\node[circle,fill=orange,inner sep=0.8mm] at (5,18) {};
\node[circle,fill=blue,inner sep=0.8mm] at (5,8) {};
\node[] at (5,21) [right] {\textcolor{orange}{SBE($\mathcal S$)}};
\node[] at (13,21) [right] {\textcolor{white}{OU($\mathcal S$)}};
\node[] at (5,14) [right] {\bf \textcolor{blue}{SBE($\Sdir$)}};
\node[] at (13,14) [right] {\bf \textcolor{white}{OU($\Sdir$)}};
\node[] at (5,4) [right] {\bf \textcolor{green}
{SBE($\mathcal{S}_0$)}};
\node[] at (13,4) [right] {\bf \textcolor{white}{OU($\mathcal{S}_0$)}};
\draw[->,>=latex] (0,13) -- (26,13);
\draw[<->,>=latex] (0,-1) -- (0,26);
\draw (0,25) node[left]{$\delta$};
\draw (24.5,13) node[below]{$\kappa$};
\draw (0,13) node[left]{$0$};
\draw (0,8) node[left]{$-1$};
\draw (0,18) node[left]{$1$};
\draw (4.5,13) node[below]{$\frac{1}{2}$};
\end{tikzpicture}
\end{center}
\caption{Illustration of the main theorem regarding the space of test functions and respective equations. }
\label{Fig:MT}
\end{figure}

To prove Theorem \ref{thm:mainresult} we proceed with the classical proof-scheme used for many similar scaling limit results. We will prove the tightness of $(\mathcal{Y}^n)_{n\in\N}$ in the corresponding Skorohod space using Mitoma's criterion. Then we show that the limiting points of that sequence are all energy solutions to the respective SPDE. And finally we use the uniqueness of those energy solutions to show that all limiting points are equal in distribution and thus together with the tightness we prove the convergence.

\medskip

More precisely in the last section we will show the following uniqueness result, that will give us the full convergence. 

\begin{theorem}
 \label{thm:SBEunique}
    Let $\mathcal C \in \{ \mathcal S,\Sdir,\mathcal S_0\}$, then the solution to SBE($\mathcal C$) is unique in law  if $\mathcal C\in \{ \mathcal S,\Sdir\}$ or if $\mathcal C = \mathcal S_0$ and the solution additionally satisfies the condition (BC).
\end{theorem}

\begin{proof}
    The proof for the case $\mathcal C = \mathcal S$ is the result of \cite{GP18}. For $\mathcal C = \Sdir$, this is the content of Theorem \ref{thm:B=Q} and for the last case, the one of Corollary \ref{cor:S0unique}.
\end{proof}

The application of the uniqueness result of \cite{GPP24} will be discussed in Section \ref{sec:uniqueness} through Theorem \ref{thm:B=Q} and Corollary \ref{cor:S0unique}. The distinction for the proof of uniqueness in the case that $\mathcal C = \Sdir$ or $\mathcal C = \mathcal S_0$ is crucially due to the fact that the result in \cite{GPP24} requires $\mathcal C$ to be a core of the Dirichlet-Laplacian $\Delta$, this is however not the case for $\mathcal S_0$. For the definition and references on the notion of a core see Appendix \ref{sec:sobolev}.

\begin{remark}
 For $\delta \in (-1,1)$ there are two equivalent ways to uniquely characterise the limit $\mathcal Y$ through its Dirichlet boundary condition. It can be obtained as the unique energy solution to SBE($\Sdir$) or as the unique solution to SBE($\mathcal S_0$) also satisfying (BC). For $\delta = 1$ only the former characterisation holds. 
\end{remark}

\section{Existence and Characterization of limit-points} 
\label{sec:existenceandtightness}

In this section we prove that $(\mathcal Y^n)_{n\in \N}$ has limit-points that solve the stochastic Burgers equation. We first assume that the sequence  $(\mathcal{Y}^n)_{n \in \mathbb N}$ is tight with respect to the Skorohod topology of  $\mathcal{D}([0,T],\mathcal{S}_{\delta}')$. Then from Prohorov's theorem we know that sub-sequential limits exist. Let us denote a limit point of $(\mathcal{Y}^n)_{n \in \mathbb N}$ by $\mathcal{Y}$. The proof of tightness is postponed to Subsection \ref{sec:tightness}. We now characterise the limit $\mathcal{Y}$ as an energy solution to the corresponding SBE($\mathcal S_\delta$) depending on the range of the parameter $\delta$ that rules the strength of the moving heat bath. Finally, we also show that for $\delta<-1$ the additional boundary conditions from Definition \ref{def:extradirichlet} are fulfilled.

\subsection{Characterization of limit-points}

We start by proving condition $i)$ of Definition \ref{def:SBEenergysolutions}. We follow the same steps as \cite[Chapter 11, Section 2]{KL99}.
We want to show that $\mathcal{Y}_t$ is a white noise with variance $\sigma^2$. Thus we need to show that  for every $H\in\mathcal{S}_{\delta}$ that $\mathcal Y_t(H)$ has Gaussian distribution  $N(0,\sigma^2\|H\|_{L^2}^2)$. To that end, we consider the Fourier transform of $\mathcal Y^n_t$ and since we are under the invariant state, which is a product measure, we have that for $n\in\N$ and $t\in[0,T]$ it holds
\begin{align*}
     \mathbb E_{\nu_{\beta,\lambda}}[ \exp(i \mathcal Y^n_t(H)] & =\mathbb E_{\nu_{\beta,\lambda}}[ \exp(i \mathcal Y^n_0(T_{c_n t}H)]= \prod_{x\in\Z} \mathbb E_{\nu_{\beta,\lambda}} \bigg[\exp\bigg(\frac{i}{\sqrt{n}} T_{c_n t}H(\tfrac{x}{n}) \bar\xi_x\bigg)\bigg].
\end{align*}
From Taylor expansion of the exponential, last display becomes equal to 
$$
  \prod_{x\in\Z}  \bigg(1- \frac{\sigma^2}{2n} T_{c_n t}H(\tfrac{x}{n})^2 + \mathcal O\Big(n^{-\frac{3}{2}}\Big) \bigg) .$$
  Taking the logarithm of last display and noting that  $\log (1+x) = x + \mathcal O(x^2)$, we get that 
 \begin{align*}
     \lim_{n\to\infty}\log\Big(\mathbb E_{\nu_{\beta,\lambda}}[ \exp(i \mathcal Y^n_t(H)]\Big) 
    & =\lim_{n\to \infty} -\frac{\sigma^2}{2n}\sum_{x\in\Z} T_{c_nt}H(\tfrac{x}{n})^2 + \mathcal O\big(n^{-\frac 12}\big)= -\frac{\sigma^2}{2}\|H\|^2_{L^2}.
\end{align*}
Where the last equality uses the approximation of a Riemann-integral for sufficiently smooth $H$.

\subsubsection{The martingale problem} \label{sec:martingale}

Recall that we want to show that limit points satisfy a martingale problem. To that end, we employ Dynkin's formula for time-inhomogeneous Markov processes that we state in Appendix \ref{sec:inhomogenoustools}. To that end, let $H\in\mathcal S_\delta(\mathbb R)$ and note that

\begin{equation} \label{eq:martingaledecomp}
   \mathcal M^n_t(H) := \mathcal{Y}^n_t(H) - \mathcal{Y}^n_0(H) - \int_0^t (\partial_s + n^2L_{n,s})\mathcal{Y}^n_s(H) ds,
\end{equation}
is a martingale with respect to the natural filtration of the process. We start by developing the integral term. We start with the contribution of the time derivative.  Recalling \eqref{eq:c_n}, a simple computation shows that  
\begin{align*}
    \partial_s\mathcal{Y}^n_s(H) = c_n \mathcal{Y}^n_s(H')=\frac{2\rho n^2 \alpha}{n^{3/2+\kappa}}\sum_{x\in\mathbb{Z}} T_{c_n s} H'(\tfrac{x}{n}) {\bar\xi_x}(sn^2).
\end{align*}
Now we compute the term related to the action of the generator ${L} = \alpha n^{-\kappa}A +\gamma S + n^{-\delta}B_{n,t}$. To make the presentation simple, we  start by computing the contribution of the Hamiltonian dynamics. Since $A(\xi_x)= \xi_x\xi_{x+1} -\xi_{x-1}\xi_x$, by a summation by parts and by properly centring the variables, we obtain 
\begin{align*}
A \mathcal{Y}_s^n(H) &= 
     -\frac{1}{n^{3/2}} \sum_{x\in\mathbb{Z}} \nabla_n^+(T_{c_n s}H)(\tfrac{x} {n}) {\bar\xi_x} (sn^2){\bar\xi_{x+1}} (sn^2)+ \frac{\rho}{n^{5/2}} \sum_{x\in\mathbb{Z}} \Delta_n (T_{c_n s}H)(\tfrac{x}{n}) {\bar\xi_x}(sn^2) \\   &- \frac{2\rho}{n^{3/2}} \sum_{x\in\mathbb{Z}} \nabla_n^+(T_{c_ ns}H)(\tfrac{x}{n}) \bar{\xi_x}(sn^2),
\end{align*}
where $$\nabla_n^+H(\tfrac xn)=n\big(H(\tfrac{x+1}{n}) - H(\tfrac xn)\big)$$ denotes the discrete gradient of $H$. We also remark, that with our notation the discrete gradient commutes with the translation, i.e. $\nabla^+_nT_{c_n s} = T_{c_ns}\nabla^+_n$.

On the other hand, for the exchange dynamics, by performing a summation by parts we get
\begin{equation*}
    \gamma S\mathcal{Y}^n_s(f) = \frac{\gamma}{\sqrt{n}} \sum_{x\in\mathbb{Z}} T_{c_n s}H(\tfrac{x}{n}) ({\bar \xi_{x+1}}(sn^2)+{\bar\xi_{x-1}}(sn^2)-2{\bar\xi_x}(sn^2)) = \frac{\gamma}{n^2} \mathcal{Y}^n_s(\Delta_nH) ,
\end{equation*}
where $$\Delta_nH(\tfrac xn) = n^2\big(H(\tfrac {x+1}{n})+H(\tfrac{x-1}{n})- 2H(\tfrac xn))$$ is the discrete Laplacian of the test function $H$.
Adding up all the terms we obtain
\begin{equation} \label{eq:dynkinoperator}
    \begin{split}
        (\partial_s + n^2L_{n,s})&\mathcal{Y}^n_s(H)= \gamma\mathcal{Y}^n_s(\Delta_nH) -\alpha  n^{1/2-\kappa} \sum_{x\in\mathbb{Z}} \nabla^+_n (T_{c_n s}H)(\tfrac{x}{n}) {\bar\xi_{x+1}(sn^2)}{\bar\xi_x}(sn^2)\\
         + \alpha \rho n^{-1/2-\kappa}&\sum_{x\in\mathbb{Z}} \Delta_n (T_{c_n s}H)(\tfrac{x}{n}){\bar\xi_x}(sn^2) +  2\alpha \rho n^{1/2-\kappa}\sum_{x\in\mathbb{Z}}\left(T_{c_n s}H' - \nabla^+_n(T_{c_n s}H)\right) (\tfrac{x}{n})\bar{\xi_x}(sn^2)\\
          + n^{3/2-\delta} H&\left(\frac{c_ns - \lfloor c_ns \rfloor}{n}\right) \left(\frac{\lambda}{\xi_{-\floor{c_ns}}(sn^2)} -\beta\right).
    \end{split}
\end{equation}
To proceed we introduce a technical lemma which will be useful to bound \eqref{eq:dynkinoperator}.
\begin{lemma} 
\label{lem:L^2-normdiscrete}
   Recall the notation \eqref{eq:discreteL2} introduced in Definition \ref{def:notation}. Let $H:\R\to\R$ be such that  $K:= \sup_{u \in \mathbb R }  \left\{ (1+u^2) H^2(u) \right\}  <\infty$, then for all $t\in\R$ we have 
 \begin{equation}
 \label{eq:Leander2}
    \begin{split}
        \left\Vert T_{t} H\right\Vert_{2,n}^2 & \leq  (1+\pi^2/3) K  \ .
        \end{split}
    \end{equation}
\end{lemma}

\begin{proof}
To prove \eqref{eq:Leander2}, we observe that by shift invariance of the integers we have
\begin{align*} \label{eq:main2nnormestimate}
        \|T_tH\|_{2,n}^2 & = \frac{1}{n} \sum_{i = 0}^{n-1} \sum_{x\in\mathbb{Z} } H\left( x + \frac{i}{n} +t\right)^2 = \frac{1}{n} \sum_{i = 0}^{n-1} \sum_{x\in\Z} H\left( x+ \frac{i}{n}+t-\floor{\frac{i}{n} + t} \right)^2\\
    &\leq \frac{1}{n} \sum_{i = 0}^{n-1} \sup_{s\in[0,1)}\sum_{x\in\Z} H\left( x+ s\right)^2 = \sup_{s\in [0,1)} \sum_{x\in\mathbb{Z}} H(x+s)^2 \ .
\end{align*}
Then using that $\sum_{x\in\N}\frac{1}{x^2} = \tfrac{\pi^2}{6}$ it follows
\begin{align*}
    \sup_{s\in[0,1)}\sum_{x\in\Z} H(x +s)^2 &=  \sup_{s\in[0,1)} \sum_{x\in\Z} (1+(x+s)^2)H(x+s)^2 \frac{1}{1+(x+s)^2} \\
    & \leq K  \sum_{x\in\Z} \sup_{s\in[0,1)} \frac{1}{1+(x+s)^2}  \leq  K\left(1+2\sum_{x\in\N}\frac{1}{x^2}\right)  =K \left(1+\frac{\pi^2}{3}\right).
\end{align*}
\end{proof}
Observe that by the fact of having introduced the field on a moving frame with the specific velocity given in \eqref{eq:c_n},  some of the degree one terms (i.e. polynomials in $\xi_x$) appearing in the previous display are shown to  vanish, as $n\to\infty$, w.r.t. the $L^2(\mathbb P_{\nu_{\beta,\lambda}})$-norm, more precisely
\begin{equation} \label{eq:discretegradapprox}
    \begin{split}
        \mathbb E_{\nu_{\beta,\lambda}}&\Bigg[ \left(\int_0^t 2\alpha b^2\rho n^{1/2-\kappa} \sum_{x\in\mathbb{Z}}\left(T_{c_n s}H' - \nabla^+_n(T_{c_n s}H)\right)(\tfrac{x}{n})\bar{\xi_x}(n^2 s)\ ds \right)^2 \Bigg] \\
         \lesssim \  & t \int_0^t  n^{1-2\kappa} \sum_{\Z \setminus \{-\floor{c_ns}\}} \Big(\left(T_{c_n s}H' - \nabla^+_n(T_{c_n s}H)\right) (\tfrac{x}{n})\Big)^2  ds  \\ 
         &+ \ee{\left(\int_0^t n^{\frac{1}{2} - \kappa} \left(H'\left(\frac{c_ns-\floor{c_ns} }{n}\right) -\nabla^+_n H\left(\frac{c_ns-\floor{c_ns} }{n}\right)\right)\bar \xi_{\floor{c_ns}}(n^2s)\ ds\right)^2}\ ,
    \end{split}
\end{equation}
where in the last estimate we split the term with $x =-\floor{c_nt}$ with the convex inequality of the square and for the first term we used the Cauchy-Schwarz's inequality, Fubini's theorem,  together with the fact that the invariant measure is product.
To bound the  second line on \eqref{eq:discretegradapprox}, we use Lemma \ref{lem:L^2-normdiscrete} and the continuous differentiability of $H$ outside  $\R\setminus\{0 \}$ to obtain 
\begin{align*}
    t \int_0^t  n^{1-2\kappa}&  \sum_{\Z \setminus \{-\floor{c_ns}\}}   \Big(\left(T_{c_n s}H' -   \nabla^+_n(T_{c_n s}H)\right) (\tfrac{x}{n})\Big)^2  ds \\ 
    &= tn^{2-2\kappa}\int_0^t\|(T_{c_n s}H' - \nabla^+_n(T_{c_n s}H))T_{c_ns}\mathbbm{1}_{\R\setminus [-n^{-1},0]}\|_{2,n}^2 ds\\
    &\lesssim n^{2-2\kappa} \sup_{u\in\R}(1+|u|^2)((H'(u) - \nabla^+_nH(u))^2\mathbbm{1}_{\R\setminus[-n^{-1},0]}(u))^2\\
    &\lesssim n^{2-2\kappa} \sup_{u\in\R\setminus[-n^{-1},0]}(1+|u|^2) \frac{1}{n^2} \sup_{v\in[u,u+n^{-1}]}H''(v)^2 \lesssim n^{-2\kappa}.
\end{align*}
In the case when $H\in\Sdir$ and is not continuously differentiable at $0$, we use the replacement Lemma \ref{lem:boxreplacement} to show that that the third line of \eqref{eq:discretegradapprox} can be replaced by the average over a box of size $n$ with vanishing variance:
\begin{align*}
    & \ee{\left(\int_0^t n^{\frac{1}{2} - \kappa} \left(H'\left(\frac{c_ns-\floor{c_ns} }{n}\right) -\nabla^+_n H\left(\frac{c_ns-\floor{c_ns} }{n}\right)\right)\bar \xi_{\floor{c_ns}}(n^2s)\ ds\right)^2} \\
    & \lesssim n^{1-2\kappa} \Bigg(\ee{\left(\int_0^t \frac{1}{n} \sum_{k=1}^n \bar \xi_{\floor{c_n s}+k}(n^2s)\ ds\right)^2} \\
    &+ \ee{\left(\int_0^t \frac{1}{n}\sum_{k=1}^n (\bar\xi_{\floor{c_n s}}(n^2s)-\bar\xi_{\floor{c_n s}+ k}(n^2s)\ ds\right)^2}\Bigg) \lesssim n^{1-2\kappa} \left(\frac{1}{n} +\frac{n}{n^2}\right) =2n^{-2\kappa}.
\end{align*}
Observe that above we used crucially  the specific choice of the velocity of the moving frame in order to show that  \eqref{eq:discretegradapprox} is of order $\mathcal O(n^{-2\kappa})$.
We also note that the same reasoning allows us to show that 
\begin{equation*}
    \begin{split}
\mathbb E_{\nu_{\beta,\lambda}}\Bigg[ \left(\int_0^t \alpha b^2\rho n^{-1/2-\kappa} \sum_{x\in\mathbb{Z}} \Delta_n (T_{c_ns}H)(\tfrac{x}{n}){\bar\xi_x}(n^2s) \ ds \right)^2 \Bigg] 
        \lesssim \   t\int_0^t n^{-2\kappa}\|\Delta_n (T_{c_ns}H) \|_{2,n}^2 ds 
    \end{split}
\end{equation*}
vanishes at rate $\mathcal{O}\left(  n^{-2\kappa} \right) $, since from Lemma \ref{lem:L^2-normdiscrete} we have
\begin{equation} \label{eq:discreteL2discontinuity}
\begin{split}
    &\|\Delta_n (T_{c_ns}H) \|_{2,n}^2\\ &\lesssim \frac{1}{n}\sum_{x\in\Z\cap[-\floor{c_ns}-1,-\floor{c_ns}+1]} \Delta_n (T_{c_ns}H)(\tfrac{x}{n})^2 +\sup_{u\in\R} (1+u^2) \Delta_n H(u)^2\mathbbm{1}_{\R\setminus[-n^{-1},n^{-1}]}(u) \\
    & \lesssim \frac{1}{n} + \sup_{x\in\R\setminus[-n^{-1},n^{-1}]} (1+ u^2)\sup_{v\in[u-n^{-1},u+n^{-1}]} H''(v) = \mathcal{O}(1).
\end{split}
\end{equation}

We will now study the remaining terms in (\ref{eq:dynkinoperator}). We start with the contribution of the  heat bath. The previous strategy of employing the Cauchy-Schwarz inequality will not be enough and in order to control these additive functionals we will use the  Kipnis-Varadhan inequality as stated in  Corollary \ref{cor:h-1heatbath}. It is at this point that we have to consider different spaces of test functions in order to obtain the desired SPDEs.

\begin{itemize}
\item If $\delta>1$ and for $H\in\mathcal S $, from  Corollary \ref{cor:h-1heatbath} we have that
\begin{equation}
\begin{split}
&\mathbb E_{\nu_{\beta,\lambda}}\Bigg[ \left(\int_0^t n^{3/2-\delta} H\parent{\frac{c_ns-\floor{c_ns}}{n}}  \parent{\frac{\lambda}{\xi_{-\floor{c_n s}}(sn^2)} -\beta} \ ds \right)^2\Bigg] \\ 
&  \leq n^{3-2\delta}\|H\|_{\infty} \ee{ \parent{ \int_0^t \frac{\lambda}{\xi_{-\floor{c_n s}}(sn^2)} -\beta \; ds }^2 } \lesssim t^2 n^{1 -\delta}
\end{split}
\end{equation}
which vanishes when $\delta>1$.
\item For $\delta\in(-1,1]$ we will need the additional assumption that $H(0) = 0$, thus $H\in\Sdir$ suffices.
Now we can apply a first order Taylor approximation and  Corollary \ref{cor:h-1heatbath} to obtain
\begin{equation}
\begin{split}
&\ee{ \parent{\int_0^t n^{3/2-\delta} H\parent{\frac{c_ns-\floor{c_ns}}{n}}  \parent{\frac{\lambda}{\xi_{-\floor{c_n s}}(n^2s)} -\beta} \ ds }^2} \\ 
&  \leq n^{1-2\delta}\|H'\|_{\infty}^2 \ee{ \parent{ \int_0^t \frac{\lambda}{\xi_{-\floor{c_n s}}(sn^2)} -\beta \; ds }^2 } \lesssim t^2 n^{-\delta-1}
\end{split}
\end{equation}
which vanishes when $\delta>-1$.
\item For $\delta\in(-\infty,-1]$ we need to take $ H\in\mathcal{S}_{0},$ so that for any $d\in\mathbb{N}$ and  $j\in\{0,1,\cdots,d\}$ it holds $H^{(j)}(0) = 0$. From a Taylor expansion of order $d+1$ and  Corollary \ref{cor:h-1heatbath}, we obtain
\begin{equation}
\begin{split}
&\ee{ \parent{\int_0^t n^{3/2-\delta} H\parent{\frac{c_ns-\floor{c_ns}}{n}}  \parent{\frac{\lambda}{\xi_{-\floor{c_n s}}(sn^2)} -\beta} \ ds }^2} \\ 
&  \leq n^{3-2\delta-2(d+1)}\|H^{(d+1)}\|_{\infty} \ee{ \parent{ \int_0^t \frac{\lambda}{\xi_{-\floor{c_n s}}(sn^2)} -\beta \; ds }^2 } \lesssim t^2 n^{ -\delta-1-2d}.
\end{split}
\end{equation}
This vanishes  if $\delta>-1-2d$, however by taking $H\in\mathcal{S}_0$ we can take $d\in\N$ arbitrarily large, so that the rate can be made negative for all $\delta\in\R$.
\end{itemize}
  \medskip 
  
We observe  that as the scaling of the heat bath changes (i.e. depending on the regime of  $\delta$) we were forced to take test functions in the space $\mathcal S_\delta$ to make this term vanish as the estimates we have at hands are not sufficient to accommodate all the values of $\delta$. The drawback of this approach is that we use  a space of test functions for which  we may loose the uniqueness in law of the solution of the martingale problem. For that reason we need to extract the information about the boundary behavior directly from the microscopic system and this is the role of condition (BC).

\medskip

Now let us move towards the term coming from the exchange dynamics of $S$. From the Cauchy-Schwarz inequality together with the fact that $\mathcal Y_t^n$ is a linear continuous functional such that $\ee{\parent{\mathcal Y_t^n(H}^2}\lesssim \|T_{c_n t}H\|_{2,n}^2$
we obtain that 
\begin{align*}
   & \ee{ \sup_{t\in[0,T]} \left(\int_0^t\mathcal{Y}^n_t(\Delta_n H) - \mathcal{Y}^n_t(\Delta H) ds\right)^2} \lesssim T\int_0^T\left\|\Delta_n (T_{c_n s} H) -\Delta (T_{c_n s} H)\right\|_{2,n}^2 ds
\end{align*}
and from Lemma \ref{lem:L^2-normdiscrete} we get the convergence to zero by separating the discontinuity of $H'$ like in \eqref{eq:discreteL2discontinuity} and using the fact that 
\begin{align*}
    |\Delta_nH(u) - H''(u)|  \lesssim n^{-1} \sup_{v\in[u-n^{-1},u+n^{-1}]} |H'''(v)|
\end{align*}
for $u\in\R\setminus[-n^{-1},n^{-1}]$. 

Then for the converging subsequence we have convergence in law 
\begin{align*}
    \int_0^t \mathcal Y^n_s(\Delta_nH) ds \xrightarrow[n\to\infty]{} \int_0^t \mathcal Y_s(\Delta H) ds.
\end{align*}

Finally, it remains to deal with the rightmost term on the first line of \eqref{eq:dynkinoperator}.
This is the term whose limit depends on the range of the parameter $\kappa$. To characterize the limit we will need the following result known as the \emph{second-order Boltzmann-Gibbs principle}.
\begin{theorem}[Second-order Boltzmann-Gibbs principle] \label{thm:2ndboltzgibbs}
Let $\psi: \R_+\times \R\to \R$ be such that 
\begin{align*}
    \int_0^t \|\psi(s,\cdot)\|_{2,n}^2 ds < \infty.
\end{align*}
Then for all $t\in[0,T]$, $n\in\N$ and any $\varepsilon>0$ it holds that:
\begin{align*}
   & \ee{\left(\int_0^t \sum_{x\in\Z} \psi\left(s,\tfrac{x}{n}\right)\left(\bar{\xi}_x(sn^2)\bar{\xi}_{x+1}(sn^2) -\left(\overrightarrow{\xi}^{\varepsilon n}_x(sn^2)\right)^2 + \frac{\sigma^2}{\varepsilon n}\right) \;ds\right)^2 }\\&\quad \lesssim \int_0^t \|\psi(s,\cdot)\|_{2,n}^2\;  ds\left(\varepsilon +\frac{t}{\varepsilon^2 n} \right),
\end{align*}
where $\overrightarrow{\xi}_x^{\varepsilon n}$ is the empirical average on the box of size $\lfloor \varepsilon n \rfloor$ at the right side of $x$ given by  
\begin{align*}
    \overrightarrow{\xi}_x^{\varepsilon n} := \frac{1}{\floor{\varepsilon n}}\sum_{y = x+1}^{x+\floor{\varepsilon n}} \bar{\xi}_y.
\end{align*}
\end{theorem}

We do not present the proof of last result since it is exactly the same as in \cite{GJS17}. We note that in that proof only the exchange dynamics is used to control the term that we are estimating and we can simply neglect the contribution of the moving heat bath. We leave these details to the reader. Applying the  last result,  we observe that when $\kappa>1/2$ the  $L^2(\mathbb P_{\nu_{\beta,\lambda}})$-norm of the rightmost term on the first line of \eqref{eq:dynkinoperator} vanishes as $n\to\infty$. On the other hand, for $\kappa=1/2$, we are able to replace each variable $\xi_x$ by the average $ \overrightarrow{\xi}_x^{\varepsilon n} $, and this allows   rewriting that term as 
\begin{align}\label{eq:quadraticterm}
    \int_0^t\alpha b^2 \sum_{x\in\mathbb{Z}} \nabla^+_n \left(T_{c_n s} H\left(\tfrac{x}{n}\right) \right)\left(\overrightarrow{\xi}^{\varepsilon n}_x(sn^2)\right)^2  ds 
\end{align} 
plus terms that vanish as $n\to+\infty$ and $\epsilon\to0$.
Recall \eqref{eq:iota} and note that 
\begin{align*}
    \frac{1}{\sqrt{n}}\mathcal{Y}^n_s(\iota_{\varepsilon}^{{z}/n}) = \frac{1}{\varepsilon n} \sum_{y = \floor{z-c_ns}+1}^{\floor{z -c_ns+\varepsilon n}} \bar\xi_{y}(sn^2).
\end{align*}
Thus, by making a change of variables,  we can rewrite \eqref{eq:quadraticterm} as
\begin{align*}
   &  \alpha b^2\int_0^t \frac{1}{n}\sum_{z\in\mathbb{Z}} \nabla^+_n H\left(\tfrac{z}{n}\right) \left(\mathcal{Y}^n_t
   \Big(\iota_{\varepsilon}^{z/n}\Big)\right)^2  ds
\end{align*}
plus terms whose  $L^2(\mathbb P_{\nu_{\beta,\lambda}})$-norm  vanishes as $n\to+\infty$ and then $\varepsilon\to0$.
Assuming the convergence of $\mathcal{Y}^n\to\mathcal{Y}$ in distribution on $\mathcal D([0,T],\mathcal{S}_{\delta}')$ now clearly leads to 
\begin{align*}
    &\alpha b^2\int_0^t \int_{\R} H'\left(u\right) \left(\mathcal{Y}_t(\iota_{\varepsilon}^u)\right)^2  du\  ds + \mathcal{O}(\varepsilon).
\end{align*}

This means that by taking $\varepsilon\to 0$ we get in the limit $\alpha  b^2\mathcal{Q}_t(H)$ of Definition \ref{def:SBEenergysolutions}. We note that  the above limit exists since the previous term forms a Cauchy  sequence in $L^2$,  and by completeness. Thus for limiting points we obtain all the elements of the martingale problem formulation of the SBE.\\
For the proof of the bound for $\mathcal Q_t(H)$ of $ii)$ in Definition \ref{def:SBEenergysolutions}, we can use the same arguments as in Section 5.3 of \cite{GJ14} applying the second order Boltzmann-Gibbs principle. We also leave  this details to the reader as they are standard now.\\
We also note that when we consider the time reversed process $(\mathcal Y_{T-t})_{t\in[0,T]}$ appearing in  condition $iv)$ of Definition \ref{def:SBEenergysolutions} we know that it is generated by $(L_t^{\ast})_{t\in[0,T]}$, which is the $L^2(\nu_{\beta,\lambda})$-adjoint of $(L_t)_{t\in[0,T]}$.
We have seen in Section \ref{sec:modelandmainresults} that $S$ and $B_{n,t}$ are symmetric operators and $A$ is antisymmetric, thus $L^{\ast}_t = -\alpha n^{-\kappa}A + \gamma S+ n^{\delta}B_{n,t}$. Thus, from the same procedure as we outlined above,  it is easy to prove that $(\mathcal Y_{T-t})_{t\in[0,T]}$ satisfies the martingale problem stated in condition $iii)$ with the prefactor $\mathfrak b=-\alpha b^2$. \\
In the simpler case $\kappa>\frac{1}{2}$ we retain a prefactor $n^{\kappa-1/2}$ in front of the previous computations for the asymmetric term. This implies that the limit of the nonlinear term is $0$ and that we fall back into the definition of the Ornstein-Uhlenbeck equation, i.e. \eqref{eq:SBE} with $\mathfrak b=0$.\\

Finally,  it remains showing that the  martingale appearing in Dynkin's formula \eqref{eq:martingaledecomp} converges, but this is the next and final step that we present below. From this, we conclude that 
$(\mathcal{Y}^n)_{n\in\mathbb N}$ has a limit point $\mathcal{Y}$ that is an energy solution as we stated in Theorem \ref{thm:mainresult}. 

\subsubsection{Convergence of the martingale to a continuous Gaussian process} \label{sec:martconvergence}
We need to show that the martingale $M^n_t(H)$ defined in \eqref{eq:martingaledecomp} converges to a mean zero Gaussian process in $\mathcal{D}([0,T],\R)$ with covariance given by $t\mapsto 2\gamma \sigma^2 \|H'\|_{L^2}^2t$. 
To prove this we will use Proposition 4.4 of \cite{GJMM24} which originally can be found in \cite{JS03} as Theorem VIII.3.11.

\begin{prop} \label{prop:martconvergence}
    Let $((M^n_t)_{t\in[0,T]})_{n\in\N}$ be a sequence of $L^2$-martingales with càdlàg trajectories and let $\Phi:[0,T]\to[0,\infty)$ be a continuous increasing function. Assume that
    \begin{itemize}
        \item[i)] the quadratic variation $\langle M^n\rangle_t\to \Phi_t$ in probability as $n\to \infty$ for all $t\in[0,T]$;
        \item[ii)] the size of the biggest jump of $M^n$ vanishes in probability, in other words
        \begin{align*}
            \sup_{0\leq t\leq T} |M^n_t - M^n_{t^-}| \xrightarrow[n\to\infty]{\mathbb P} 0.
        \end{align*}
    \end{itemize}
    Under these conditions, we conclude that $M^n\to M$ in law as $n\to\infty$ in the $\mathcal D([0,T],\R)$ where $(M_t)_{t\in[0,T]}$ is a continuous martingale of predictable quadratic variation $\langle M\rangle_t = \Phi_t$ for every $t\in[0,T]$.
\end{prop}
To prove $i)$, we compute the quadratic variation of $M^n_t(H)$, which  is given by  \eqref{eq:dynkin2}:
\begin{align*}
    \langle M^n(H) \rangle_t  = n^2\int_0^t L_{n,s}(\mathcal{Y}^n_s(H))^2 - 2\mathcal{Y}^n_s(H)L_{n,s}\mathcal{Y}^n_s(H) ds.
\end{align*}
This can be computed for each part of the generator separately. The computations for  $A$ and  $S$ have already be done in \cite{ABGS22}. Here we just need to compute the contribution coming from the dynamics of the moving thermal bath. Putting all together we get:
\begin{equation}\label{eq:quadvar}\begin{split}
    \langle M^n(H)\rangle_t &=\gamma \int_0^t\frac{1} {n}\sum_{x\in\Z} \left(\nabla^+_nT_{c_ns}H\left(\frac{x}{n}\right)\right)^2 \left(\bar{\xi}_{x+1}(n^2s) -\bar{\xi}_x(n^2s)\right)^2\\&+ n^{1-\delta} H\left(\frac{c_ns-\floor{c_ns}}{n}\right)^2 ds .
\end{split}\end{equation}
For the appropriate test function $H\in\mathcal S_{\delta}$ we have seen in the previous section that the summand due to the heat bath vanishes at a positive rate. 
Because $\xi(n^2s) \sim \nu_{\beta,\lambda}$ it is clear that $\ee{(\bar{\xi}_{x+1}(n^2 s) -\bar{\xi}_x(n^2 s))^2} = 2\sigma^2$. Therefore as an approximation of a Riemann-integral we obtain
\begin{align} \label{eq:quadvarlimit}
    \lim_{n\to\infty} \ee{\langle M^n(H)\rangle_t} = 2\gamma \sigma^2 t\|H'\|_{L^2}^2.
\end{align}
In the second step, to get the convergence in probability, we will in fact prove that 
\begin{align*}
    \lim_{n\to\infty} \ee{ (\langle M^n(H)\rangle_t - \ee{\langle M^n(H)\rangle_t})^2} = 0
\end{align*} from where the convergence in probability follows.  By Cauchy-Schwarz inequality, the stationarity of $\xi$ and the fact that the heat bath contribution vanishes with positive rate we can write 
\begin{align*}
    &\mathbb E\big[ (\langle M^n(H)\rangle_t - \ee{\langle M^n(H)\rangle_t})^2 \big ] \nonumber\\ 
    &\lesssim t\int_0^t\ee{\left( \gamma \frac{1}{n} \sum_{x\in\Z} \left(\nabla^+_n T_{c_ns}H\left(\frac{x}{n}\right)\right)^2 \left(\bar{\xi}_{x+1}(n^2 s) -\bar{\xi}_x(n^2 s)\right)^2 - 2 \sigma ^2 \  \right)^2} ds \nonumber \\ 
    & = \frac{\gamma t}{n^2} \sum_{x\in\Z} \sum_{y\in\Z} \int_0^t \left(\nabla^+_n T_{c_ns}H\left(\frac{x}{n}\right)\right)^2\left(\nabla^+_n T_{c_n s}H\left(\frac{y}{n}\right)\right)^2 ds \ \ee{Z_{x+1,x}(\bar\xi) Z_{y+1,y}(\bar\xi)},
\end{align*}
where the random variable $Z_{x,y}(\xi) = (\xi_{x} -\xi_y)^2-2\sigma^2$ and they are  centred and for disjoint indices also uncorrelated. Furthermore when correlated we have $\ee{Z_{x+1,x}^2} = 2\ee{\bar \xi_x^4} + 5\sigma^4$ and  $\ee{Z_{x+1,x}(\bar\xi) Z_{x+2,x+1}(\bar\xi)} = \ee{\bar\xi^4_{x}} + 2\sigma^4$, from where it follows that they are uniformly bounded. Thus, the last display is equal to 
\begin{align*}
    &\frac{\gamma t}{n^2} \sum_{x\in \Z}\sum_{y\in\{x-1,x,x+1\}} \int_0^t \left(\nabla^+_nT_{c_ns}H\left(\frac{x}{n}\right)\right)^2\left(\nabla^+_nT_{c_n s}H\left(\frac{y}{n}\right)\right)^2 ds \ee{Z_{x+1,x}(\bar\xi) Z_{y+1,y}(\bar\xi)} \\
    &\lesssim\frac{1}{n} \int_0^t \left\| \sum_{i\in\{-1,0,1\}} \left(\nabla^+_n T_{c_ns}H\left(\frac{\cdot}{n}\right)\right)^2\left(\nabla^+_nT_{c_ns}H\left(\frac{\cdot+i}{n}\right)\right)^2\right\|^2_{2,n} ds \lesssim \frac{1}{n}\xrightarrow[n\to\infty]{} 0,
\end{align*}
where the last bound is obtained by Lemma \ref{lem:L^2-normdiscrete} and an estimate similar to \eqref{eq:discreteL2discontinuity}. This estimate combined with \eqref{eq:quadvarlimit} gives that $\langle M^n\rangle_t\to 2\gamma \sigma \| H'\|_{L^2}^2t$ in probability as $n\to\infty$ for all $t\in[0,T]$.

It remains to prove the condition $ii)$ of Proposition \ref{prop:martconvergence}. Note that we can not replicate the argument of \cite{ABGS22}, which used crucially the conservation of energy,  because in our case the Markov process is evolving on $\Z$ and not on the  finite torus. 
We need then to study the expectation of the maximal jump of $M^n(H)$, which is equal to the maximal jump of the fluctuation field $\mathcal Y^n_t(H)$ because they only differ by an integral in time which, as we will see in the proof of tightness, is almost surely continuous.
To study the jumps of $\mathcal{Y}^n_t$ we define $t\mapsto y_t$ to be the process that maps $t$ to the index $x$ where the last exchange of $\xi_x$ and $\xi_{x+1}$ occurred (due to the dynamic of $S$). This is well-defined because there are only countably many jumps with Exp($\gamma$)  distribution. Then,
\begin{align*}
    \mathbb E_{\nu_{\beta,\lambda}} \bigg[\sup_{t\in[0,T]}\ &|(\mathcal{Y}^n_t(H)-\mathcal{Y}^n_{t^-}(H)|^2\bigg]\\ &= \frac{1}{n}\ee{\sup_{t\in[0,T]} \left(\left( T_{c_nt}H\left(\frac{y_t+1}{n}\right)-T_{c_nt}H\left(\frac{y_t}{n}\right)\right)\left(\xi_{y_t+1}(n^2t)-\xi_{y_t}(n^2t)\right)\right)^2}\\
    & = \frac{1}{n^3}\ee{\sup_{t\in[0,T]} \left( \nabla_n^+T_{c_nt}H\left(\frac{y_t}{n}\right)\left(\xi_{y_t+1} (n^2t)-\xi_{y_t}(n^2t)\right))\right)^2}.
\end{align*}
We then use the convexity of the square $\left(\xi_{y_t+1} -\xi_{y_t}\right)^2\leq 2\xi_{y_t + 1}^2 + 2\xi_{y_t}^2$ which gives two terms that can be bounded in the same way: using the fact that all $\xi_x$ are positive we have
\begin{align*}
   \frac{1}{n^3}\mathbb E\Bigg[ & \sup_{t\in[0,T]} \bigg( \nabla_n^+T_{c_nt}H\left(\frac{y_t}{n}\right)\xi_{y_t}(n^2t)\bigg)^2 \Bigg]  \leq \frac{1}{n^3}\ee{\sup_{t\in[0,T]} \left( \sum_{x\in\Z} \bigg|\nabla_n^+T_{c_nt}H\left(\frac{x}{n}\right)\bigg|\xi_{x}(n^2t)\right)^2} \\ 
   & = \frac{1}{n^3} \ee{\sup_{t\in[0,T]} \left( \sum_{x\in\Z} \bigg|\nabla_n^+T_{c_nt}H\left(\frac{x}{n}\right)\bigg|\bar{\xi}_{x}(n^2t) + \bigg|\nabla_n^+T_{c_nt}H\left(\frac{x}{n}\right)\bigg|\rho\right)^2}\\
   & \lesssim \frac{1}{n^3} \ee{\sup_{t\in[0,T]} \left( \sum_{x\in\Z} \bigg|\nabla_n^+T_{c_nt}H\left(\frac{x}{n}\right)\bigg|\bar{\xi}_{x}(n^2t) \right)^2} + \frac{\rho^2}{n}\left\|\bigg|\nabla_n^+T_{c_nt}H\left(\frac{\cdot}{n}\right)\bigg|^{\frac{1}{2}} \right\|^4_{2,n}.
\end{align*}
Now, rightmost term in last line is of order $\mathcal{O}(n^{-1})$ due to Lemma \ref{lem:L^2-normdiscrete}, and therefore vanishes as $n\to+\infty$, we focus on the leftmost term. We use the martingale decomposition for $\mathcal Y^n_\cdot (G_n)$, with $G_n(x):= |\nabla^+_n H(x)|$, and applying Doob's martingale inequality to the martingale, we obtain 
\begin{align*}
    \frac{1}{n^3} \mathbb E &\Bigg[\sup_{t\in[0,T]} \left( \sum_{x\in\Z} \bigg|\nabla_n^+T_{c_nt}H\left(\frac{x}{n}\right)\bigg| \bar{\xi}_{x}(n^2t) \right)^2 \Bigg] = \frac{1}{n^2} \ee{\sup_{t\in[0,T]} \mathcal{Y}^n_t(G_n)^2} \\
    & \leq \frac{1}{n^2} \ee{\langle M^n(G_n)\rangle_T} + \frac{1}{n^2} \ee{\sup_{t\in[0,T]} \left(\mathcal{Y}^n_0(G_n) + \int_0^t (\partial_s + n^2L_{n,t})\mathcal{Y}^n_s(G_n) ds\right)^2}.
\end{align*}
In the next section, while proving tightness, we will show that the initial fluctuation field and the integral of the generators $A$ and $S$ applied to the fluctuation field are bounded in expectation. 
Therefore, last terms vanish at rate $O(n^{-2})$. 
The contributions of the heat bath generator $B_{n,s}$ have to be analised  separately because the function $G_n$ defined above does not fulfill the required  boundary condition at $0$ even when $H\in\Sdir$. In fact, using the Corollary \ref{cor:h-1heatbath} we get
\begin{align*}
    &\frac{1}{n^2}\ee{\sup_{t\in[0,T]} \left(\int_0^t B_{n,s}\mathcal{Y}^n_s(G_n)ds\right)^2} \\
    &\leq \frac{1}{n^2}\ee{\sup_{t\in[0,T]} \left(\int_0^t n^{2-\frac{1}{2}-\delta} G_n \left(\frac{c_ns - \floor{c_ns}}{n}\right)W'_{\beta,\lambda}(\xi_{x_n(s))}(n^2t))  ds\right)^2}\\
    & \lesssim T n^{1-2\delta}\|G_n\|_{\infty} \int_0^T \|W_{\beta,\lambda}(\xi_{\floor{c_ns}})\|^2_{-1,n,s}ds  \lesssim n^{-(\delta + 1)}
\end{align*}
which vanishes for $\delta\in(-1,\infty)$. For $\delta\leq -1$ we take $H\in\mathcal S_{\delta } = \mathcal{S}_0$, and from a Taylor approximation at $0$ of order $d\in\N$ for $H\left(\frac{c_ns - \floor{c_ns}+1}{n}\right)$ and $H\left(\frac{c_ns - \floor{c_ns}}{n}\right)$, we get 
\begin{align*}
    G_n\left(\frac{c_ns - \floor{c_ns}}{n}\right) &= n\left( H\left(\frac{c_ns - \floor{c_ns}+1}{n}\right) -H\left(\frac{c_ns - \floor{c_ns}}{n}\right)\right) \lesssim n^{1-d}.
\end{align*}
Thus the term containing $B_{n,s}$ vanishes a arbitrarily high rate when $\delta\in (-\infty,-1]$ and $H\in\mathcal{S}_0$.
In the same way we can treat the term $\ee{\langle M^n(G_n)\rangle_T}$ where we have calculated the expression in \eqref{eq:quadvar}.  The first part given by the generator $S$ converges and is bounded, and hence vanishes as $\mathcal{O}(n^2)$. The term due to $B_{n,t}$ is of order $\mathcal{O}(n^{1-\delta})$ for $\delta \in (-1,\infty]$, for $\delta\in (-\infty,-1)$ the test function $H\in\mathcal S_0$ gives us a vanishing term at arbitrarily fast rate again.

Thus, from $ii)$ follows.
Therefore we can apply Proposition \ref{prop:martconvergence} so that $(M^n_t(H))_{t\in[0,T]}$ converges in law on $\mathcal{D}([0,T],\R)$ to a continuous mean zero martingale $(M_t(H))_{t\in[0,T]}$ with predictable quadratic variation $2\gamma \sigma^2 \|\nabla H\|_{L^2}^2 t$.

\subsubsection{Additional boundary conditions when $\delta<1$.}
\label{subsubsec:BC}
We want to show that in the case that $\delta<1$ any limiting point $\mathcal Y$ of the fluctuation field satisfies the explicit boundary condition (BC).

\begin{lemma}
    In the case that $\delta < 1$ limiting fluctuations $\mathcal Y$ satisfy the boundary condition {(BC)}.
\end{lemma}
\begin{proof}
    The proof is completely analogous to that of Lemma 5.2 of \cite{BGJS22}, relying essentially on two replacement lemmas given in Appendix \ref{sec:replacementlemmas}. We summarise the proof as follows: one first shows that
    \begin{align} \label{eq:iotalowersemicont}
        \mathbb E\left[\sup_{t\in[0,T]}\left( \int_0^T \mathcal Y_s(\iota^0_{\varepsilon})\ ds\right)^2\right] \lesssim \liminf_{n\to\infty} \ee{\sup_{t\in[0,T]}\left( \int_0^T \mathcal Y_s^n(\iota^0_{\varepsilon})\ ds\right)^2},
    \end{align}
    and then writes 
    \begin{align*}
        &\ee{\sup_{t\in[0,T]} \left(\int_0^t \mathcal{Y}^n_s(\iota_{\varepsilon}^0)\ ds \right)^2}= \ee{\sup_{t\in[0,T]} \left(\int_0^t \frac{1}{\varepsilon\sqrt{n}}\sum_{x=\lceil c_nt\rceil}^{\floor{c_nt+\varepsilon n}} \bar{\xi}_x(n^2s) ds\right)^2}\\
        & \lesssim \ee{\sup_{t\in[0,T]} \left( \int_0^t \frac{1}{\varepsilon \sqrt{n}} \sum_{x=\lceil c_nt\rceil}^{\lceil c_nt\rceil +\floor{\varepsilon n}}\bar \xi_x(n^2s) -\bar{\xi}_{\floor{c_ns}}(n^2s) ds\right)^2 } \\
        & + \ee{\sup_{t\in[0,T]}\left(\int_0^t \sqrt{n}\bar{\xi}_{\floor{c_ns}}(n^2s) ds\right)^2} +\mathcal O\left(\frac{1}{\varepsilon \sqrt{n}}\right).
    \end{align*}
    Then by Lemma \ref{lem:boxreplacement}, choosing $l = \floor{\varepsilon n}$, we get a bound of order $\mathcal O(\varepsilon)$ for the first term on the right hand side, and using Lemma \ref{lem:boundaryreplacement} we obtain a bound of order $\mathcal O(n^{\delta-1})$ on the boundary term. Thus we showed that \eqref{eq:iotalowersemicont} is bounded by a vanishing quantity for $\delta<1$ as $n\to\infty$ and then $\varepsilon\to 0$.
\end{proof}

\subsection{Tightness}
\label{sec:tightness}
To prove our main convergence theorem, showing the tightness of $\mathcal{Y}^n$ in $\mathcal{D}([0,T],\mathcal{S}_{\delta}')$ is an essential step, as it will assure that every subsequence of $\mathcal{Y}^n$ has itself a convergent subsequence (of which we have already shown that its limit is an energy solution of (SBE)($\mathcal S_{\delta}$)).
To show the tightness in the Skorohod topology of $\mathcal{D}([0,T],\mathcal{S}_{\delta}')$, we will use two powerful results, often used to break down the proof into smaller simpler steps.
\begin{theorem}[Mitoma's Criterion]
    Let $E$ be a nuclear Fréchet space. Then a sequence of stochastic processes $((X^n_t)_{t\in[0,T]})_{n\in\N}$ in $\mathcal{D}([0,T],E')$ is tight with respect to the Skorohod topology if, and only if, for every $H\in E'$ the sequence of real-valued processes  $((X^n_t(H))_{t\in[0,T]})_{n\in\N}$ is tight in the Skorohod topology of $\mathcal{D}([0,T],\R)$.
\end{theorem}
\begin{proof}
    The proof can be found in \cite{M83}.
\end{proof}

\begin{theorem}[Aldous' Criterion]
    A sequence of real valued processes $((X^n_t)_{t\in[0,T]})_{n\in\N}$ defined on some probability space $(\Omega,\mathcal{F},\mathbb{P})$ is tight with respect to the Skorohod topology of $\mathcal{D}([0,T],\R)$ if
    \begin{itemize}
        \item[(i)] $\lim_{A\to \infty} \limsup_{n\to\infty} \mathbb{P}(\sup_{t\in[0,T]} |X^n_t|>A) = 0$,
        \item[(ii)] for any $\varepsilon>0$: $\lim_{r\to 0} \limsup_{n\to\infty} \sup_{0<s<r} \sup_{\tau\in\mathcal{T}_T} \mathbb{P}(|X^n_{s +\tau} -X^n_{\tau}| > \varepsilon) = 0$
    \end{itemize}
    where $\mathcal{T}_T$ is the space of stopping times bounded by $T$.
\end{theorem}
\begin{proof}
    The proof can be found in \cite[Proposition 4.1.6]{KL99}.
\end{proof}
Now by Remark \ref{rem:nuclear} we can use Mitoma's criterion, and are only required to prove tightness for every $H\in\mathcal S_{\delta}$ of $\mathcal{Y}^n(H)$ in $\mathcal{D}([0,T],\R)$, which is a much smaller space.
This will be done by proving tightness for each element of the martingale decomposition individually with Aldous' criterion. 
From Section \ref{sec:martingale}, we know that $\mathcal{Y}^n_t(H)$ is equal to 
\begin{align}\label{eq:dynkintight}
     \mathcal{Y}_0(H) + M^n_t(H) + \int_0^t \gamma\mathcal{Y}^n_t(\Delta_nH) + \alpha b^2 n^{1/2-\kappa}\sum_{x\in\Z} \nabla^+_n T_{c_ns}H\left(\frac{x}{n}\right)\bar\xi_{x+1}(n^2s)\bar\xi_x(n^2s) ds 
\end{align}
plus a term that vanishes in probability at some positive rate.

\medskip
\textbf{Tightness of the initial condition:} this is a trivial consequence of the convergence result that we proved before using characteristic functions. Indeed, since the starting distribution of the oscillator chain is the product invariant measure $\nu_{\beta,\lambda}$, we have by Lemma \ref{lem:L^2-normdiscrete}
\begin{align*}
    \ee{|\mathcal{Y}^n_0(H)|^2} = \frac{1}{n} \sum_{x\in\Z} T_{c_ns}H\left(\frac{x}{n}\right)^2  \ee{\bar{\xi}_x^2} = \mathcal O (1).
\end{align*}
Together with Tchebychev's inequality this implies $(i)$. Condition $(ii)$ is immediate in this case.

\medskip
\textbf{Tightness of the martingale:} This follows from our results we proved in Section \ref{sec:martconvergence} that $(M^n(H))_{n\in\mathbb N}$ fully converges in $\mathcal{D}([0,T],\R)$.

\medskip
\textbf{Tightness of the symmetric component:} Here we use the Cauchy-Schwarz inequality and the stationarity of $\nu_{\beta,\lambda}$ to obtain
\begin{align*} \label{eq:tightsymmetric}
    \ee{\sup_{t\in[0,T]}\left(\int_0^t \mathcal{Y}^n_s(\Delta_n H) ds\right)} &\leq T \int_0^T \ee{\left(\frac{1}{\sqrt{n}}\sum_{x\in\Z} \Delta_n T_{c_ns}H\left(\frac{x}{n}\right) \bar{\xi}_x(n^2s)  \right)^2 }ds \nonumber \\
    & = \sigma^2T\int_0^T \left\| \Delta_n T_{c_ns}H\left(\frac{x}{n}\right)\right\|^2_{2,n} ds.
\end{align*}
The convergence of the discrete Laplacian and Lemma \ref{lem:L^2-normdiscrete} imply that the last term is bounded by $\sigma^2T^2$ times a constant that depends on $H$.
This bound, with the use of Tchebychev's inequality, implies property $(i)$.
To get property $(ii)$ we remark that one can estimate an integral from $\tau$ to $s+\tau$, without taking the supremum over all $t\in[0,T]$,  in the same way to obtain the bound $\sigma^2 s^2 \|\Delta H\|_{2,n}^2$ plus a vanishing term. Again Tchebychev's inequality shows that $(ii)$ holds.
\medskip

\textbf{Tightness of the asymmetric term:}
For $\kappa>1/2$ this term just converges to zero.
For $\kappa = 1/2$ it is given by the last term in \eqref{eq:dynkintight} which using the triangle inequality can be bounded by the two following terms 
\begin{equation} 
\label{eq:boltzgibbstight}
    \ee{ \sup_{t\in[0,T]}\left(\alpha b^2  \int_0^t \sum_{x\in\Z} \nabla_n^+T_{c_ns}H\left(\frac{x}{n}\right)\left(\bar{\xi}_x(sn^2)\bar{\xi}_{x+1}(sn^2) -\left(\overrightarrow{\xi}^{\varepsilon n}_x(sn^2)\right)^2 + \frac{\sigma^2}{\varepsilon n}\right) ds \right)^2}
\end{equation}
and 
\begin{equation} 
\label{eq:averagedboxtight}
    \ee{ \sup_{t\in[0,T]} \left( \alpha b^2 \int_0^t \sum_{x\in\Z} \nabla_n^+T_{c_ns}H\left(\frac{x}{n}\right)\left( \left(\overrightarrow{\xi}^{\varepsilon n}_x(sn^2)\right)^2 - \frac{\sigma^2}{\varepsilon n} \right)ds\right)^2}.
\end{equation}
The term in (\ref{eq:boltzgibbstight}) can be bounded, with the use of the second order Boltzmann-Gibbs principle of Theorem \ref{thm:2ndboltzgibbs}, by a constant times $\int_0^T\|\nabla^+_n T_{c_ns}H\|_{2,n}^2 ds(\varepsilon + T/(\varepsilon^2n))$, which in turn is bounded by $T\|H'\|_{L^2}^2(\varepsilon + T/(\varepsilon^2n))$ plus a vanishing approximation term.
Then, using the Cauchy-Schwarz inequality, we can bound the term \eqref{eq:averagedboxtight} by 
\begin{align*}
    \alpha b^2 t \int_0^t \sum_{x,y\in\Z}\nabla^+_n T_{c_ns}H\bigg(\frac{x}{n}\bigg)\nabla^+_nT_{c_ns}H\bigg(\frac{y}{n}\bigg) \ee{\bigg(\Big(\overrightarrow{\xi}^{\varepsilon n}_x\Big)^2-\frac{\sigma^2}{\varepsilon n }\bigg) \bigg(\Big(\overrightarrow{\xi}^{\varepsilon n}_y\Big)^2-\frac{\sigma^2}{\varepsilon n}\bigg)}ds. 
\end{align*}
Now, using  the fact that, for $|x-y| > \varepsilon n$, the two previous averages  do not overlap and are uncorrelated and the variables involved are mean zero, it allows us to bound the double sum by a factor $\varepsilon n$ and a single sum. Moreover the convex inequality $2xy\leq x^2+y^2$ enables us to estimate the previous term by a constant times
\begin{align*}
    t \varepsilon n \int_0^t \sum_{x\in\Z}\nabla^+_n T_{c_ns}H\bigg(\frac{x}{n}\bigg) \ee{\bigg(\Big(\overrightarrow{\xi}^{\varepsilon n}_x\Big)^2-\frac{\sigma^2}{\varepsilon n }\bigg)^2}ds.
\end{align*}
Since the variables $\xi_x$ have finite fourth moments w.r.t. the invariant measure, we can bound the variance above by $O((\varepsilon n)^{-2})$, which allows bounding last display by a constant times
\begin{align} \label{eq:tightnonlinear}
    \frac{T}{\varepsilon} \int_0^T \left\| \nabla_n^+ T_{c_ns}H\left(\frac{x}{n}\right)\right\|^2_{2,n}ds.
\end{align}
With Lemma \ref{lem:L^2-normdiscrete} this can be bounded by a constant dependent on $H$ times $T^2/\varepsilon$.
In summary we can bound the contribution of the asymmetric generator by some constant times
\begin{align*}
    T\|H'\|_{L^2}^2 \left( \varepsilon + \frac{T}{\varepsilon^2n}  +\frac{T}{\varepsilon}\right).
\end{align*}
This together with  Tchebychev's inequality proves $(i)$. Noting the linear dependence in $T$ due to the integration domain, we get the property $(ii)$ by an analogous argument.

This concludes the section showing the tightness of $\mathcal{Y}^n_t$ in $\mathcal{D}([0,T],\mathcal{S}_{\delta}')$ and identified the limit points as energy solutions of the corresponding SPDE.

\section{Uniqueness  in law}
 \label{sec:uniqueness}

In this section we prove Theorem \ref{thm:SBEunique}, namely that stationary energy solutions, in the sense of Definition \ref{def:SBEenergysolutions}, are unique for $\mathcal C \in \{\mathcal S, \Sdir\}$ and for $\mathcal C=\mathcal S_0$ if the boundary condition (BC) is  also fullfiled. This, pat{together with the proof of tightness,}  completes the proof of the convergence. \\
 
Uniqueness of stationary solutions to SBE($\mathcal S$) has been proved in \cite{GP18}: this is exactly the case where $\mathcal S_{\delta} = \mathcal S$ and thus covers the regime $\delta >1$. 

Uniqueness of stationary solutions to SBE($\Sdir$) will be proved here and it is therefore a new result. It corresponds to SBE on the line with a Dirichlet boundary condition at the origin, i.e. with domain $D:= (-\infty,0)\cup (0,\infty)$. The case with Dirichlet boundary conditions on $[0,1]$ was already treated in \cite{GPS20}.  The proof essentially relies on the general uniqueness recent result of \cite{GPP24}, where Dirichlet boundary conditions are considered, but on the domain $D=(0,\infty)$. 

A crucial requirement in the proof of \cite{GPP24} is that the space of test functions $\mathcal C$ involved is a core of the domain of the Dirichlet-Laplacian. Since it is not the case for $\mathcal S_0$ (see Lemma \ref{lem:Sdircore}), we treat solutions to SBE($\mathcal S_0$) separately with the additional boundary condition (BC) in order to restore uniqueness.

\subsection{Uniqueness for SBE($\Sdir$)} 
\label{subsec:Sdir}
The main task of this section is to show that our notion of energy solutions coincides with that of \cite{GPP24} and that our approximation of the non-linear term $\nabla(\mathcal Y^2)$ fulfils the necessary conditions for the uniqueness result therein.

We start by recalling the notion of energy solution from \cite{GPP24} and the corresponding uniqueness result, which differs from ours. There, the equation under investigation is formally written in the form ($\mathfrak a, \mathfrak b, \mathfrak c$ constants)
\begin{align} 
\label{eq:SBEGPP24}
    \partial_t \mathcal{Y}_t = -\mathfrak aA(\mathcal{Y}_t) + \mathfrak b:\!B(\mathcal{Y}_t,\mathcal{Y}_t)\!: + \sqrt{2\mathfrak cA}\mathcal W_t \ .
\end{align}
where the operator $A$ is defined in the setting of a Gelfand triple \cite{LR15}  $\mathcal V \hookrightarrow \mathcal H=\mathcal H'\hookrightarrow \mathcal V'$, $B$ is an ill-defined non-linear term and $\mathcal W$ is a standard space-time white noise on $\mathcal H$. The precise definition of energy solutions of \cite{GPP24} to \eqref{eq:SBEGPP24} and the corresponding uniqueness result is proved there under the conditions \textbf{(A)} and \textbf{(B)} involving the operators $A, B$ above and a dense subset $\mathcal C \subset \mathcal V$. Those conditions are rather lengthy to explain, see Section 1.1 of \cite{GPP24}.

\medskip

In the present context, $\mathcal H:= L^2(D)$ where $D=\R\setminus \{0\}$, $\mathcal V := \mathcal H^1_0(D)$ with $\|\cdot\|_{\dot{\mathcal V}} = \|\cdot\|_{\mathcal V}-\|\cdot\|_{\mathcal H}$ and the topological dual $\mathcal V':= \mathcal H^{-1}_0(D)$. The operator $A := \Delta:\mathcal V\to \mathcal V'$ is the Dirichlet-Laplacian and $B:\mathcal V^2\to \mathcal V'$ is the Burgers nonlinearity given by $(G,H) \mapsto\partial_u (GH)$. 
We recall that in this setting for $G\in \mathcal H$ and $H\in\mathcal V$ we have the inner product $\langle G,H\rangle_{\mathcal V',\mathcal V} = \langle G , H\rangle_{\mathcal H}$. With the choice $\mathcal C=\Sdir$, the conditions  \textbf{(A)} and \textbf{(B)} can be shown to hold. In fact this is almost done in \cite[Example 4.1]{GPP24} where the domain considered is not $D$ but $(0,\infty)$. To adapt to our case the proof of conditions  \textbf{(A)} and \textbf{(B)}, the only important difference is the choice of the approximation operator $\check \rho_{\varepsilon}$ of the non-linear term, defined  in \eqref{eq:checkrho}. Then, to be able to follow the same steps as in \cite{GPP24}, we need to recover the regularity estimates found in \cite[Lemma A.4]{GPP24} and this is done in Lemma \ref{lem:rhoregularity}. \\

We use the notation $\colon B(\mathcal{Y},\mathcal{Y}) \colon $ like in \cite{GPP24} to point out the fact that the way to make sense of the nonlinear term $\partial_u (\mathcal Y^2)$  is, a priori, different from  \eqref{eq:QSBE} of Definition \ref{def:SBEenergysolutions} which is being used for the same purpose.\\

For any $\varepsilon>0$, the bounded operator $\check\rho_\varepsilon: L^2(D) \to \mathcal H^1_0(D)$ is defined, on  $G \in L^2 (D)$, by 
\begin{equation}
\label{eq:checkrho}
    \check\rho_\varepsilon(G)(u):= \chi_{\varepsilon} (u) (\check\iota_{\varepsilon} \ast G)(u)= \varepsilon^{-1} \ \inf \Big( 1, \big\vert \tfrac{u}{\varepsilon}\big\vert \Big) \int_u^{u+\varepsilon} G(v) dv
\end{equation}
where{\footnote{The function $\iota_\varepsilon^0$ is defined in \eqref{eq:iota}.}} $\check\iota_{\varepsilon}(u):=  \iota_{\varepsilon}^0(-u)$ and     
\begin{align}
\label{eq:defchiepxilon}
    \chi_{\varepsilon}: u \in \R\to \chi_{\varepsilon} (u):= {\mathbf 1}_{\R\setminus (-\varepsilon,\varepsilon)}(u) + |u|{\mathbf 1}_{(-\varepsilon,\varepsilon)}(u) = \inf \Big ( 1, \big\vert \tfrac{u}{\varepsilon} \big\vert \Big) \in [0,1] \ .
\end{align}

Note that from \cite[Lemma 2.4]{GPP24} we can define the nonlinearity appearing in the SPDE in the following way. 
\begin{prop} 
\label{prop:Burgersapprox}
Let $\rho:\mathcal H\to \mathcal V$ be a bounded operator and let $G\in \mathcal V$. Let $(e_k)_k$ be a complete orthonormal system (CONS) of $\mathcal H$. Then when $\mathcal Y_s$ is a space white-noise on $\mathcal H$ with covariance $\mathfrak c/ \mathfrak a$, whose law is denoted by $\mu$, the following limit exists in $L^2(\mu)$ and does not depend on the specific orthonormal basis: 
    \begin{align}
    \label{eq:defBtilderho}
        \langle :\!\tilde{B}^{\rho}(\mathcal{Y}_s,\mathcal{Y}_s)\!:, G\rangle_{\mathcal V',\mathcal V} :=  \lim_{m\to\infty} \sum_{k,l\leq m} (\mathcal{Y}_s(e_k)\mathcal{Y}_s(e_l) - \delta_{kl})\langle B(\rho(e_k),\rho(e_l)),G\rangle_{\mathcal V',\mathcal V},
    \end{align}
    with the bound
    \begin{align} \label{eq:Btildebound}
        \| \langle :\!\tilde{B}^{\rho}(\mathcal{Y}_s,\mathcal{Y}_s)\!:, G\rangle_{\mathcal V',\mathcal V} \|_{L^2}  \lesssim \|\rho\|_{L(\mathcal H,\mathcal V)}\|\rho\|_{L(\mathcal H,\mathcal H)} \|G\|_{\dot{\mathcal V}}.
    \end{align}
\end{prop}
Now, the new definition of stationary energy solutions taken from \cite[Lemma 2.19]{GPP24} can be defined as follows. 
\begin{definition} 
\label{def:otherSBEenergysol}
We say that the process $\mathcal Y = (\mathcal{Y}_t)_{t\in[0,T]}$ in $\Sdir'$ is a stationary energy solution of \eqref{eq:SBEGPP24} (in the sense of \cite{GPP24}) if it satisfies the condition $i)$ and $iv)$ of Definition \ref{def:SBEenergysolutions} as well as the two modified conditions $ii')$ and $iii')$ (replacing $ii)$ and $iii)$ of Definition \ref{def:SBEenergysolutions})
     \begin{itemize}
         \item[ii')] For all $G\in \Sdir$, the process $t \in [0,T] \mapsto \mathcal{Y}_t(G) \in \mathbb R $ is continuous.
         \item[iii')] For $G\in\Sdir$, the following limit  
    \begin{align}
    \label{eq:BSBE}
        \tilde{\mathcal Q}_t(G) = \lim_{\varepsilon\to 0} \int_0^t \langle :\!\tilde{B}^{\check \rho_{\varepsilon}}(\mathcal{Y}_s,\mathcal{Y}_s)\!:, G\rangle_{\mathcal V',\mathcal V} \, ds
    \end{align}
    exists  in probability, uniformly in time $t \in [0,T]$, and it is of vanishing quadratic variation. Moreover, the process 
    \begin{align*}
        M_t(G) :=  \mathcal{Y}_t(G) - \mathcal{Y}_0(G) + \mathfrak a\int_0^t\mathcal{Y}_s (AG) ds - \mathfrak b \tilde{\mathcal{Q}}_t(G), \quad t\in[0,T]
    \end{align*} 
    is a continuous martingale with quadratic variation $2t\mathfrak c\langle AG,G\rangle_{\mathcal H}$.
     \end{itemize}
\end{definition}

From the discussion above and \cite[Theorem 2.4]{GPP24} we have thus the following 
\begin{prop}
\label{prop:uniqueSBE}
The stationary energy solutions $\mathcal Y$ of (\ref{eq:SBEGPP24}) in the sense of Definition \ref{def:otherSBEenergysol} are unique in law and define a Markov process with invariant measure the law $\mu$ of the white noise on $\mathcal H$ with covariance $\mathfrak c / \mathfrak a$.
\end{prop}
We know state the main theorem of this section.

\begin{theorem} \label{thm:B=Q}
    Let $\mathcal Y =(\mathcal{Y}_t)_{t\in[ 0,T]}$ be a stationary energy solution as given in Definition \ref{def:SBEenergysolutions}  with $\mathcal C=\Sdir$. Then $(\mathcal Q_t)_{t\in [0,T]}$ and $(\tilde{\mathcal Q}_t)_{t\in[0,T]}$ are indistinguishable on $C([0,T],\Sdir')$ (i.e. almost surely equal in law on the path space). 
    In particular, $\mathcal Y$ is also an energy solution in the sense of Definition \ref{def:otherSBEenergysol}, and thus unique in law.
\end{theorem}

\begin{proof}
    We start by proving $\mathcal Q$ and $\tilde{\mathcal Q}$ are modifications of each other by using the following intermediate approximation sequence $(\rho_\varepsilon)_{\varepsilon>0}$ defined by     \begin{equation}
     \label{eq:def_rhovarepsilon}
    \forall \varepsilon>0, \quad \forall u \in \R, \quad  \rho_{\varepsilon}^u := \chi_{\varepsilon}(u)\ \iota^u_{\varepsilon},
    \end{equation}
    where we recall that $\chi_\varepsilon$ is defined by \eqref{eq:defchiepxilon}.  In fact we will show that for all $t\in[0,T]$ and $G\in \Sdir$ that a.s.
    \begin{align} \label{eq:B=Q}
        \tilde{\mathcal Q}_t(G) &:= \lim_{\varepsilon\to 0} \int_0^t \langle :\!\tilde{B}^{\check\rho_\varepsilon}(\mathcal{Y}_s,\mathcal{Y}_s)\!:, G\rangle_{\mathcal V',\mathcal V} ds  \\
        & = \lim_{\varepsilon \to 0}\int_0^t \big(\mathcal{Y}_s(\rho_{\varepsilon}^u)^2-\|\rho^u_{\varepsilon}\|^2_{L^2}\big)\partial_uG(u) \ du\ ds= \mathcal{Q}_t(G) \ . 
    \end{align}
     Fix $s\in[0,T]$ and $G\in\Sdir$. For the first equality in \eqref{eq:B=Q} let $(e_k)_{k\in\N}$ be a CONS of $L^2(D)$ in $\Sdir$, then the $L^2(\mu)$ limit expansion of $\mathcal Y_s$ is given by 
    \begin{align*}
        \mathcal{Y}_s(G) = \lim_{m\to\infty} \sum_{k\leq m} \mathcal{Y}_s(e_k) \langle e_k ,G\rangle_{L^2}.
    \end{align*}
    Recall the definition of $B$ given at the beginning of Subsection \ref{subsec:Sdir}   and \eqref{eq:defBtilderho}. By \eqref{eq:checkrho}, $\check\rho_{\varepsilon}(e_k)(u) =\chi_{\varepsilon}(u) (\check\iota_{\varepsilon} \ast e_k)(u) = \langle\rho^u_{\varepsilon},e_k\rangle_{L^2}$, we can exchange the limit and integrate by parts:
    \begin{align*}
        \int_{D} \mathcal{Y}_s(\rho^u_{\varepsilon})^2 \partial_u G(u) du &= \int_{D}  \lim_{m\to\infty} \sum_{k,l\leq m} \mathcal{Y}_s(e_k)\mathcal{Y}_s(e_l) \check\rho_{\varepsilon}(e_k)(u)\check\rho_{\varepsilon}(e_l)(u)\partial_u G(u) du \\
        & = \lim_{m\to\infty} \sum_{k,l\leq m}   \mathcal{Y}_s(e_k)\mathcal{Y}_s(e_l) \langle -\partial_u (\check\rho_{\varepsilon}(e_k)\check\rho_{\varepsilon}(e_l)),G \rangle_{L^2} \ .
    \end{align*}
For the remaining part in \eqref{eq:defBtilderho}, we observe that
\begin{align*}
    \lim_{m\to\infty} \sum_{k,l\leq m}  & \delta_{kl} \langle -\partial_u (\check\rho_{\varepsilon}(e_k)\check\rho_{\varepsilon}(e_l)),G \rangle_{L^2}du  = \lim_{m\to\infty} \sum_{k\leq m} \langle \check\rho_{\varepsilon}(e_k)^2, \partial_u G\rangle_{L^2}\\
    & = \int_{D}\lim_{m\to\infty} \sum_{k\leq m} \langle \rho^u_{\varepsilon},e_k\rangle_{L^2}^2  \partial_u G(u) du  = \int_{D}\|\rho^u_{\varepsilon} \|_{L^2}^2 \partial_u G(u) du .
\end{align*}
This shows the second equality in \eqref{eq:B=Q}. The third one is proved in Theorem \ref{thm:nonlinapprox}. Hence, \eqref{eq:B=Q} is proved. 
From  \eqref{eq:Btildebound} we know that the integral of $\tilde{\mathcal Q}$ is continuous and, from Theorem \ref{thm:nonlinapprox}, the same holds for the other two terms in \eqref{eq:B=Q}. Therefore $\tilde{\mathcal Q}$ and $\mathcal Q$ are indistinguishable.

Now, from \eqref{eq:B=Q} we get that $\mathcal Y=(\mathcal Y_t)_{t\in[0,T]}$ is an energy solution of SBE($\Sdir$) in the sense of Definition \ref{def:SBEenergysolutions}, and therefore, it also satisfies $ii')$ and the martingale problem of $iii')$ and, as a consequence, it is an energy solution in the sense of Definition \ref{def:otherSBEenergysol}. By Proposition \ref{prop:uniqueSBE} we get uniqueness of $\mathcal Y$.
\end{proof}

\subsubsection{Continuous approximation of the Burgers nonlinearity} \label{sec:burgersapprox}
One important step to prove Theorem \ref{thm:B=Q} is  to show that the nonlinearity of (SBE) is a unique process in the space of distributions independently of the mollifier approximating it. Before stating it we will define new space of test functions containing the first derivatives $H'$ for $H\in\Sdir$: 
\begin{align*}
    \mathcal S_{\text{Neu}}:= \{ \psi \; : \;  D\to\R \; \mid \;  &\exists \psi_-,\psi_+\in\mathcal S\text{ with }\forall k\in\N_0: \\ & \psi_-^{(2k +1)}(0) = 0=\psi_+^{(2k +1)}(0)\text{ and } \psi =\mathbbm{1}_{(-\infty,0)}\psi_- +\mathbbm{1}_{(0,\infty)}\psi_+ \}.\\
\end{align*}
\begin{theorem}[Continuous approximation of Burgers nonlinearity] 
\label{thm:nonlinapprox}
Let $\mathcal Y$ be a solution of SBE($\Sdir$).
Then there exists a unique process $\int_0^{\cdot}\mathcal{Y}^{\diamond 2}_s ds \in \mathcal \mathcal C(\R_+, \mathcal{S}_{Neu}')$ such that for all $\psi \in \mathcal{S}_{\text{Neu}}$ and all $\rho \in L^1(D^2) \cap L^{\infty}(D^2)$ the following bound holds for any $T>0$
\begin{equation}
\label{eq:CABN1}
\begin{split}
    &\mathbb E\left[\sup_{t\in[0,T]}\left|\left( \int_0^t \mathcal{Y}^{\diamond 2}_s ds\right)(\psi) - \int_0^t\left(\int_{D} \left(\mathcal{Y}_s\left(\rho(u,\cdot)\right)^2- \left\|\rho(u,\cdot)\right\|^2_{L^2}\right) \psi (u) du\right) ds\right|^2\right] \\
    & \lesssim T \left(\left(\int_D u^2 \psi(u)^4 du\right)^{\frac{1}{2}} + \|\psi\|_{L^2}\right)^2  R(\rho) S(\rho) 
\end{split}
\end{equation}
where we set
\begin{equation}
\label{def:R-S}
\begin{split}
    R(\rho) & := \sup_{u\in D} \||u-\cdot|\rho(u,\cdot)\|_{L^1}^{\frac{1}{2}} +\|\langle \rho(u,\cdot), 1\rangle_{L^2} - 1\|_{L^4_u}, \\ 
    S(\rho) & :=\Big(\sup_{u\in D}\|\rho(u,\cdot)\|_{L^1} + \|\rho(u,\cdot)\|_{L^1}^2\Big) \cdot \sup_{v\in D}\|\rho( \cdot,v)\|_{L^1}^{\frac{1}{2}}.
\end{split}
\end{equation}
Furthermore, recalling the definition of $\rho_\varepsilon^u$ given by \eqref{eq:def_rhovarepsilon},  for all $G\in \Sdir$ 
\begin{equation}
\label{eq:Qdiamond}
\begin{split}
     \mathcal Q_t(G) &:= \lim_{\varepsilon\to 0}\int_0^t\int_D\mathcal{Y}_s(\iota_{\varepsilon}^u)^2 G'(u) \ du ds = \lim_{\varepsilon\to 0} \int_0^t\int_D \big(\mathcal Y_s(\rho^u_{\varepsilon})^2-\|\rho_{\varepsilon}^u\|_{L^2}^2\big)G'(u) \ duds\\
     &= \left(\int_0^t \mathcal{Y}^{\diamond 2}_s ds\right)(G'). 
\end{split}
\end{equation}
\end{theorem}

A very similar result has been shown first in \cite[Proposition 3.15]{GP18} on the infinite line $\mathbb R$ and then in \cite[Lemma 6.10]{GPS20} in $[0,1]$. We will adapt here the proof of \cite[Lemma 6.10]{GPS20} to our setting where $D = \R\setminus \{0\}$ is unbounded.
The proof requires some preliminary theory on functionals of Gaussians using Wiener Chaos theory, which we recall in Appendix \ref{app:gaussian}.

\begin{proof}
We start by defining 
    \begin{equation*}
        \tilde{\mathcal{Y}}^{\diamond 2}_s(\rho(u,\cdot)) = \mathcal{Y}_s(\rho(u,\cdot))^2 -\|\rho(u,\cdot)\|_{L^2}^2.  
    \end{equation*}
Then, for all $\psi\in \mathcal S_{\text{Neu}}$, we will construct the process $\big(\int_0^{\cdot}\mathcal{Y}^{\diamond 2}_s ds\big) (\psi)$ as the limit of the Cauchy sequence of processes in $C([0,T],\R)$:
    \begin{equation*}
       \left(  \int_0^{\cdot}\int_D\tilde{\mathcal{Y}}^{\diamond 2}_s(\rho_{\varepsilon}(u,\cdot))\psi(u) du ds \right)_{\varepsilon>0} 
    \end{equation*}
for the sequence $(\rho_\varepsilon)_{\varepsilon>0}$ given by  $\rho_\varepsilon (u,v)=\rho_\varepsilon^u (v)$, where $\rho_\varepsilon^\cdot(\cdot)$ has been defined in \eqref{eq:def_rhovarepsilon}. 
However, for the moment  $(\rho_{\varepsilon})_{\varepsilon>0}$ is an arbitrary sequence of elements in $L^1(D^2) \cap L^{\infty}(D^2)$.\\

 Recall that for any $s\ge 0$, by stationarity, $\mathcal Y_s$ is a space white noise with covariance $\mathfrak c / \mathfrak a$, which for simplicity we assume to be equal to one. The term $\tilde{\mathcal{Y}}_s^{\diamond 2}$  can be rewritten in terms of the second Wiener-It\^o integral, and, using the stochastic Fubini's theorem,  we have
    \begin{align*}
        \int_{D} \big(\mathcal{Y}_s (\rho(u,\cdot))^2 -\|\rho(u,\cdot)\|^2_{L^2}\big) \psi(u) du &= \int_{D} W_2\big(\rho(u,\cdot)^{\otimes 2}\big)(\mathcal Y_s)\psi(u) du \\
        & = W_2\bigg(\int_{D} \rho(u,\cdot)^{\otimes 2}\psi(u) du \bigg)(\mathcal Y_s).
    \end{align*}
 Let us fix two positive real numbers $\varepsilon$ and $\delta$. We apply the well-known variance estimate, from Proposition \ref{prop:KVineq} together with the linearity of $W_2$ and Lemma \ref{lem:sobolevchaos}, to get 
    \begin{align}
    \label{eq:dualityburgers}
        &\mathbb E\left[\sup_{t\in[0,T]} \left| \int_0^t\int_{D}\tilde{\mathcal{Y}}^{\diamond 2}_s(\rho_{\varepsilon}(u,\cdot))\psi(u) du ds - \int_0^t\int_{D}\tilde{\mathcal{Y}}^{\diamond 2}_s(\rho_\delta (u,\cdot))\psi(u) du ds \right|^2\right] \nonumber\\
        &\leq T \left\|\int_{D}W_2\left( \rho_{\varepsilon}(u,\cdot)^{\otimes 2} \right)(\mathcal{Y})\psi(u) ds - \int_{D}W_2\left( \rho_\delta(u,\cdot)^{\otimes 2} \right)(\mathcal{Y}) \psi(u) ds\right\|_{-1,\mu}^2 \nonumber \\
        &\leq T \left\|W_2\left( \int_{D} \left(\rho_{\varepsilon}(u,\cdot)^{\otimes 2}- \rho_\delta(u,\cdot)^{\otimes 2}\right) \psi(u) ds\right)(\mathcal{Y} ) \right\|_{-1,\mu}^2 \nonumber \\
        &\leq T \sup_{\substack{f\in \Sdir(D^2)\\  f \text{  symmetric }}}\frac{\left\langle \int_{D} \left(\rho_{\varepsilon}(u,\cdot)^{\otimes 2}- \rho_\delta (u,\cdot)^{\otimes 2} \right) \psi(u) ds,f\right\rangle_{L^2} }{2d!\|\nabla f\|_{L^2}}
    \end{align}    
where $\Sdir (D^2)$ is defined in \eqref{eq:sdird}.  To estimate the last term we rewrite the scalar product
    \begin{equation}
      \begin{split}
        &\left\langle \int_{D} (\rho_{\varepsilon}(u,\cdot)^{\otimes 2} - \rho_\delta(u,\cdot)^{\otimes 2})\psi(u)\  du \ , f \right\rangle_{L^2} \\
        & = \int_{D^3} \left( \rho_{\varepsilon}(u,v_1)\rho_{\varepsilon}(u, v_2) - \delta_0(u-v_1)\delta_0(u-v_2)\right) \psi(u)f(v_1,v_2) \ du dv_1 dv_2 \\
        &+ \int_{D^3}   \left( \delta_0(u-v_1)\delta_0(u-v_2)- \rho_\delta (u,v_1)\rho_\delta (u, v_2)\right) \psi(u)f(v_1,v_2) \ du dv_1 dv_2.
      \end{split} 
    \end{equation}
 Now we only treat the term on the second line of last display  because the third one can be handled in the same way by changing $\varepsilon$ into $\delta$. We decompose
    \begin{align*}
        \rho_{\varepsilon}(u,v_1)\rho_{\varepsilon}(u,v_2)& - \delta_0(u-v_1)\delta_0(u-v_2)\\
        &= (\rho_{\varepsilon}(u,v_1) -\delta_0(u,v_1))\rho_{\varepsilon}(u,v_2) + \delta_0(u-v_1)(\rho_{\varepsilon}(u,v_2) - \delta_0(u-v_2))
    \end{align*}
    and again only consider the first term, which is bounded like the second by symmetry. We then add and subtract $f(u,v_2)$, apply the triangle inequality and use the Fundamental Theorem of Calculus to obtain
    \begin{align}
        &\left| \int_{D^3} (\rho_{\varepsilon}(u,v_1) -\delta_0(u-v_1))\rho_{\varepsilon}(u,v_2)\psi(u) f(v_1,v_2) du dv_1 dv_2 \right| \nonumber\\
        & \leq \left| \int_{D^3} \rho_{\varepsilon}(u,v_1)\rho_{\varepsilon}(u,v_2)\psi(u) (f(v_1,v_2) - f(u,v_2)) dudv_1dv_2 \right| \nonumber \\
        & + \left| \int_{D^2} \int_{D} \rho_{\varepsilon}(u,v_1) -\delta_0(u-v_1) \, dv_1 \rho_{\varepsilon}(u,v_2)\psi(u) (f(v_1,v_2) - f(u,v_2)) dudv_2 \right| \nonumber \\
        & = \left| \int_{D^3} \rho_{\varepsilon}(u,v_1)\rho_{\varepsilon}(u,v_2)\psi(u) \int_u^{v_1} \partial_w f(w,v_2) dw dudv_1dv_2 \right| \label{eq:firstterm37} \\
        & + \left| \int_{D^2} \int_{D} \rho_{\varepsilon}(u,v_1) -\delta_0(u-v_1) \, dv_1 \rho_{\varepsilon}(u,v_2)\psi(u)  f(u,v_2) dudv_2 \right|  \label{eq:secondterm38}.
    \end{align}
 Up to here the proof followed the same steps as in \cite{GPS20}. But, now, we use new bounds because our domain $D$ is unbounded.

    To treat \eqref{eq:firstterm37}, we use the Cauchy-Schwarz's inequality, then we apply twice the Jensen's inequality (as in 
    $(\int f g )^2 \leq \int f^2\|g\|_1 g$ for $g\geq0$ with $g\in L^1$), and  we conclude using $L^{\infty}$ estimates. Thus we get
    \begin{align} 
        &\left| \int_{D^3} \rho_{\varepsilon}(u,v_1)\rho_{\varepsilon}(u,v_2)\psi(u) \int_u^{v_1} \partial_w f(w,v_2) dw dudv_1dv_2 \right| \nonumber \\
        & \underset{C.-S.}{\leq}\left(\int_{D^2} \left( \int_{D^2} \rho_{\varepsilon}(u,v_1)\rho_{\varepsilon}(u,v_2)\psi(u) \mathbbm{1}_{[u,v_1]}(w) dudv_1\right)^2 dw dv_2 \right)^\frac{1}{2} \left(\int_{D^2} \partial_w f(w,v_2)^2 dw dv_2 \right)^{\frac{1}{2}} \nonumber\\
        &\underset{\text{Jensen}}{\leq} \left(\int_{D^3}  \left(\int_{D}  \mathbbm{1}_{[u,v_1]}(w) \rho_{\varepsilon}(u,v_1)dv_1\right)^2\psi(u)^2 \|\rho_{\varepsilon}(\cdot,v_2)\|_{L^1}\rho_{\varepsilon}(u,v_2)du dw  dv_2\right)^\frac{1}{2} \| \nabla f\|_{L^2(D^2)} \nonumber \\
        & \underset{Jensen}{\leq} \left(\int_{D^4}   \mathbbm{1}_{[u,v_1]}(w)^2 \rho_{\varepsilon}(u,v_1) \|\rho_{\varepsilon}(u,\cdot)\|_{L^1}  \|\rho_{\varepsilon}( \cdot,v_2)\|_{L^1}\psi(u)^2\rho_{\varepsilon}(u,v_2) dw dv_1 dv_2du   \right)^\frac{1}{2} \| \nabla f\|_{L^2(D^2)} \nonumber\\
        & \underset{L^{\infty}}{\leq} \left(\sup_{u\in D} \||u-\cdot|\rho_{\varepsilon}(u,\cdot)\|_{L^1} \|\rho_{\varepsilon}(u,\cdot)\|_{L^1} \sup_{v\in D}\|\rho_{\varepsilon}( \cdot,v)\|_{L^1} \int_{D^2}  \psi(u)^2\rho_{\varepsilon}(u,v_2)    dv_2du   \right)^\frac{1}{2} \| \nabla f\|_{L^2(D^2)} \nonumber\\
        & \leq \sup_{u\in D} \||u-\cdot|\rho_{\varepsilon}(u,\cdot)\|_{L^1}^{\frac{1}{2}} \sup_{u\in D} \|\rho_{\varepsilon}(u,\cdot)\|_{L^1} \sup_{v\in D}\|\rho_{\varepsilon}( \cdot,v)\|_{L^1}^{\frac{1}{2}}\|\psi\|_{L^2}  \| \nabla f\|_{L^2(D^2)}.  \label{eq:bigrhoestimate1}
    \end{align}
For  \eqref{eq:secondterm38}, we use the boundary condition $f(0,v_2) = 0$ to use the Fundamental Theorem of Calculus. Then we use Jensen's inequality, $L^{\infty}$ bounds and  the Cauchy-Schwarz's inequality to obtain
    \begin{align} \label{eq:bigrhoestimate2}
        &\int_{D^2} \int_D (\rho_{\varepsilon}(u,v_1)-\delta_0(u-v_1)) \ dv_1 \ \rho_{\varepsilon}(u,v_2) \psi(u) f(u,v_2)dv_2 du  \nonumber\\
        & = \int_{D^2}  (\langle\rho_{\varepsilon}(u,\cdot),1\rangle_{L^2} - 1) \rho_{\varepsilon}(u,v_2) \psi(u) \int_D \mathbbm{1}_{[0,u]}(w) \partial_w f(w,v_2)dw dv_2 du \nonumber\\
        & \underset{C.-S.}{\leq} \left(\int_{D^2}\left( \int_{D} (\langle\rho_{\varepsilon}(u,\cdot),1\rangle_{L^2} - 1) \rho_{\varepsilon}(u,v_2) \psi(u)  \mathbbm{1}_{[0,u]}(w)  du\right)^2dw dv_2 \right)^{\frac12} \|\nabla f\|_{L^2(D^2)}  \nonumber\\  
        & \underset{\text{Jensen}}{\leq} \left(\int_{D^3} (\langle \rho_{\varepsilon}(u,\cdot),1\rangle_{L^2} -1)^2 \rho_{\varepsilon}(u,v_2)\|\rho_{\varepsilon}(\cdot,v_2)\|_{L^1(D)}\psi(u)^2 \mathbbm 1_{[0,u]}(w)^2 dwdudv_2 \right)^{\frac12} \|\nabla f\|_{L^2(D^2)}  \nonumber\\
        & \underset{L^{\infty}}{\leq} \left(\sup_{v\in D} \|\rho_{\varepsilon}(\cdot,v)\|_{L^1}\int_{D^2} (\langle \rho_{\varepsilon}(u,\cdot),1\rangle_{L^2} -1)^2 \rho_{\varepsilon}(u,v_2) \psi(u)^2 u \ dv_2du \right)^{\frac12} \|\nabla f\|_{L^2(D^2)}  \nonumber\\      
        & \underset{L^{\infty}}{\leq} \left(\sup_{v\in D} \|\rho_{\varepsilon}(\cdot,v)\|_{L^1}\sup_{u\in D} \|\rho_{\varepsilon}(u,\cdot)\|_{L^1} \int_{D}(\langle \rho_{\varepsilon}(u,\cdot),1\rangle_{L^2} -1)^2 \psi(u)^2 u \ du \right)^{\frac12} \|\nabla f\|_{L^2(D^2)}  \nonumber\\
        & \underset{C.-S.}{\leq} \left(\sup_{v\in D} \|\rho_{\varepsilon}(\cdot,v)\|_{L^1}^{\frac12}\sup_{u\in D} \|\rho_{\varepsilon}(u,\cdot)\|_{L^1}^{\frac12}\|\langle \rho_{\varepsilon}(u,\cdot),1\rangle_{L^2} -1\|_{L_u^4} \left(\int_D u^2\psi(u)^4 du\right)^{\frac{1}{4}}\right) \|\nabla f\|_{L^2(D^2)} .
    \end{align}
    
    Recall the definition of $R,S$ given by \eqref{def:R-S}. Putting \eqref{eq:bigrhoestimate1} and \eqref{eq:bigrhoestimate2} together, enables us to bound the fraction in the supremum in \eqref{eq:dualityburgers} uniformly in $f$, so that we get
    \begin{equation}
    \label{eq:linftyl2normapprox}
    \begin{split}
        &\mathbb{E}\Bigg[\sup_{t\in[0,T]} \left| \int_0^t\int_{D}\tilde{\mathcal{Y}}^{\diamond 2}_s(\rho_{\varepsilon}(u,\cdot))\psi(u) du ds - \int_0^t\int_{D}\tilde{\mathcal{Y}}^{\diamond 2}_s(\rho_{\delta}(u,\cdot))\psi(u) du ds \right|^p \Bigg] \\
        & \lesssim T \Big(\left(\int_D u^2\psi(u)^4 du\right)^{\frac{1}{4}} + \|\psi\|_{L^2}^2\Big) \big( R(\rho_{\varepsilon}) S(\rho_{\varepsilon}) +R(\rho_{\delta}) S(\rho_{\delta})\big).
    \end{split}
    \end{equation}
 To obtain a converging sequence it now suffices to choose a sequence $(\rho_{\varepsilon})_{\varepsilon>0}$ such that 
    \begin{align} \label{eq:approxbound}
        R(\rho_{\varepsilon})S(\rho_{\varepsilon})\to 0 \ \text{ as }\ \varepsilon \to 0.
    \end{align}

Indeed for such $(\rho_{\varepsilon})_{\varepsilon>0}$ we obtain a Cauchy sequence in $C([0,T],\R)$ for all test functions $\psi$, with limit $\big( \int_0^\cdot \mathcal Y^{\diamond 2}ds\big)(\psi)$. This together with Mitoma's criterion gives us then the convergence to $\int_0^\cdot \mathcal Y^{\diamond 2}ds $ in distribution in $C([0,T],\mathcal S_{\text{Neu}})$. We also then deduce that \eqref{eq:CABN1} holds. Observe that $\int_0^\cdot \mathcal Y^{\diamond 2}ds$ depends a priori of the sequence $(\rho_{\varepsilon})_{\varepsilon>0}$ chosen.

We now show that $(\iota^{\cdot}_{\varepsilon}(\cdot))_{\varepsilon>0}$ and $(\rho^{\cdot}_{\varepsilon}(\cdot))_{\varepsilon>0}$ satisfy \eqref{eq:approxbound}. Indeed $\|\iota_{\varepsilon}^u (\cdot)\|_{L^1} = 1 = \|\iota_{\varepsilon}^{\cdot}(v)\|_{L^1}$ and $\chi_{\varepsilon}\leq 1$, thus we have $S(\iota_{\varepsilon}) = 1$ and $ S(\rho_{\varepsilon}) = 1$. Furthermore we have $\langle \iota_{\varepsilon}^u(\cdot), 1\rangle_{L^2} = 1$ for all $u\in D$, and additionally,  
    \begin{align*}
        \|\langle \rho_{\varepsilon}^u,1\rangle_{L^2} -1\|_{L^4}^4 = \int_D |\chi_{\varepsilon}(u)-1|^4 du \leq 2 \varepsilon.
    \end{align*} 
Finally for any $u\in D$
    \begin{align*}
        \||u-\cdot|\times \iota_{\varepsilon}^u(\cdot)\|_{L^1} + \||u-\cdot|\times \rho_{\varepsilon}^u(\cdot)\|_{L^1}  \leq  \varepsilon( \|\iota_{\varepsilon}^u(\cdot)\|_{L^1} + \|\rho_{\varepsilon}^u\|_{L^1}) \leq \varepsilon\xrightarrow[n\to\infty]{} 0 .
    \end{align*}
    Thus we have proven that $R(\iota_{\varepsilon})\to 0$ and $R(\rho_{\varepsilon})\to 0$ as $\varepsilon\to 0$, which completes our proof. 
    
 \bigskip 
 
 We thus define the process $\int_0^\cdot \mathcal Y^{\diamond 2}ds$ by using the sequence $(\rho_\varepsilon)_{\varepsilon>0}  = (\rho_\varepsilon^\cdot (\cdot))_{\varepsilon>0}$. Then, for any $G \in \Sdir$, we apply \eqref{eq:CABN1} with $\psi=G'$, $\rho = \iota_\varepsilon^\cdot (\cdot)$ and deduce, since  $R(\iota_{\varepsilon}^\cdot (\cdot))S(\iota_{\varepsilon}^\cdot(\cdot)))$ goes to $0$ as $\varepsilon$ goes to $0$, the statement \eqref{eq:Qdiamond}. This completes the proof of the theorem. 
\end{proof}

\subsection{Uniqueness for SBE($\mathcal S_0$) with boundary conditions (BC)} 
\label{sec:uniqunessS0} 

To prove uniqueness for solutions of SBE($\mathcal S_0$) that also satisfy additional boundary conditions  given in (BC) on $D=\R\setminus \{0\}$ we will closely follow the ideas of \cite{BGJS22} which were used to handle the Ornstein-Uhlenbeck process on the domain $(0,1)$. The goal can be summarised in the two following results. 

\begin{theorem} 
\label{thm:sdirextension}
Assume that $\mathcal Y$ is a solution of SBE($\mathcal S_0$) satisfying (BC). Then $\mathcal Y$ can be uniquely extended into a process $\mathcal Y$ in $C([0,T],\Sdir')$ solving SBE($\Sdir$). 
\end{theorem}

Last theorem enables us to make use of the uniqueness result proved in the previous section.

\begin{corollary}
 \label{cor:S0unique}
Let $\mathcal{Y}$ and $\tilde{\mathcal{Y}}$ be two solutions to SBE($\mathcal S_0$) that additionally satisfy (BC). Then $\mathcal{Y}$ and $\tilde{\mathcal{Y}}$ have the same law.
\end{corollary}
\begin{proof}
 By Theorem \ref{thm:sdirextension}, $\mathcal Y$ has a uniquely defined extension that is solution to SBE($\Sdir$). Because the solutions to SBE($\Sdir$) are unique there can not be any $\tilde{\mathcal{Y}}$ different from $\mathcal Y$ satisfying the assumptions of the theorem. 
\end{proof}

Now that we have seen how to obtain uniqueness, we have to construct the extension of $\mathcal Y$ solving SBE($\Sdir$). We start by defining the Banach-space which will be used for that construction.

\begin{definition}
    Let $\mathfrak H$ be the Hilbert space of real valued progressively measurable processes $(X_t)_{t\in[0,T]}$ in $L^2$ equipped with the norm $\|X\|_{\mathfrak H} = \left(\mathbb E \left[ \int_0^T X_s^2 ds\right]\right)^{\frac{1}{2}}$.\\
    Denote $\mathfrak B$ as the Banach space of real valued processes $(X_t)_ {t\in[0,T]}$ with continuous trajectories equipped with the norm $\|X\|^2_{\mathfrak B} = \mathbb E\left[\sup_{t\in[0,T]} |X_t|^2 \right]$. Note that $\|X\|_{\mathfrak H} \leq \sqrt{T}\|X\|_{\mathfrak B}$ gives us the continuous embedding $\mathfrak B \hookrightarrow \mathfrak H$.
\end{definition}
Then we start with the simple but more general extension.

\begin{lemma}\label{lem:L2extension}
    For any $H\in L^2(D)$ we can uniquely define $(\mathcal Y_t(H))_{t\in[0,T]}$ as a process in $\mathfrak H$,   allowing this abuse of notation using $\mathcal Y$ for our limit process and its extension. The extension coincides with our SBE solution as soon as $H\in\mathcal S_0$. Furthermore, for every $t\in[0,T]$ the extension $\mathcal Y$ is a space white-noise. 
\end{lemma}
\begin{proof}
    The proof uses $i)$ of Definition \ref{def:SBEenergysolutions} and the argument of Remark \ref{rem:continuity&whitenoise}. 
\end{proof}
To extend $\mathcal Y$ to a solution of SBE($\Sdir$) we will use the following intermediate space of test functions that we define as follows: 
    \begin{align*}
        \tilde{\mathcal S}:= \big\{H\in \mathcal H^2(D) \mid H(0^+) = H(0^-) = H'(0^+)  = H'(0^-) = 0 \big\}.
    \end{align*}

\begin{lemma} \label{lem:H2approx}
  Let $H\in\tilde{\mathcal{S}}$. There exists a sequence of functions $(H_n)_{n\in\N}$ in $\mathcal S_0$ such that $H_n\to H$ in $\mathcal H^2(D)$, i.e. $\lim_{n\to\infty} H_n^{(k)} = H^{(k)}$ in $L^2(D)$ for $k=0,1,2$. 
\end{lemma}
\begin{proof}
For  $H\in \tilde{\mathcal S}$, the sequence $H_n = H\Phi_{n^{-1}}\in \mathcal S_{0}$ as for, where 
\begin{align*}
    \Phi_{n}(u):= \left\{ \begin{array}{ll}
    \tilde \phi(u) & u\in[0,n^{-1}], \\
    \tilde{\phi}_n(-u)  & u\in[-n^{-1},n^{-1}], \\
    1 & u\in\R\setminus [-n^{-1},n^{-1}],
    \end{array} \right.
\end{align*}
and  $\tilde{\phi}_{n}(u):= \int_0^{nu}a(v)\ dv$ where $a(\cdot)$ is given in  Definition \ref{def:psialphabeta}. The proof  is analogous to that of Lemma A.4 of \cite{BGJS22} and for that reason it is omitted.
\end{proof}

\begin{lemma} \label{lem:tildeSextension}
    The continuous extension of $\mathcal Y$ is a solution to SBE($\tilde{\mathcal S}$) in the sense of Definition \ref{def:SBEenergysolutions}.
\end{lemma}

\begin{proof}
   The proof is similar to the one of  Lemma 5.3 in \cite{BGJS22}, except that  now we need to handle the nonlinear term.  
    To that end, fix any $H\in \tilde{\mathcal S}$. From  Lemma \ref{lem:H2approx}, we now that there exists  a sequence $(H_n)_{n\in\N}$ with $H_n\in\mathcal S_0$ that converges in $\mathcal H^2(D)$ to $H$.
    One easily sees that $\mathcal{Q}$ is linear with respect to the test function $H$ which means that the martingale $M$ also is, such that we get the convergence of $(M(H_n))_{n \in \mathbb N}$ in $\mathfrak{B}$ with a limit $\tilde M(H)$ which is a continuous martingale with quadratic variation $t\to \mathfrak c t\|H'\|_{L^2}$. 
    Lemma \ref{lem:L2extension} tells us that the extention of $\mathcal Y$ satisfies $i)$ and also gives us existence of the $\mathfrak H$-limit $\mathcal Y_{\cdot}(H) -\mathcal Y_0(H) - \mathfrak a\int_0^{\cdot}\mathcal Y_s(H'') ds$.
    Moreover we can show using Cauchy-Schwarz for the invariant process $\mathcal Y$ and the $L^2$ bound from Lemma \ref{lem:L2extension} that $(\int_0^{\cdot}\mathcal Y_s(H''_n) ds)_{n \in \mathbb N}$ is a Cauchy sequence in $\mathfrak B$ and that the limit must be $\int_0^{\cdot}\mathcal Y_s(H'')ds$ in  $\mathfrak B$ (and $\mathfrak H$). This implies that pathwise continuity of the Laplacian term.\\
    For the nonlinear term we note that the proof of the second order Boltzmann-Gibbs principle in Theorem \ref{thm:2ndboltzgibbs} holds for any test function $H\in L^2(\R)$ and does not require special boundary conditions because only the exchange dynamics of $S$ are relevant. This implies that for any such $H$ the bound of property $ii)$ in Definition \ref{def:SBEenergysolutions} holds and enables the definition of $\mathcal Q_{\cdot}(H)$.
    In the same way the proof from Section 6.1 of \cite{GJ14} yields the bound
    \begin{align} \label{eq:Qbound}
    \mathbb E\left[ \left(\mathcal Q_t (H-H_n)\right)^2\right] \lesssim t^{\frac{3}{2}} \|H' -H_n'\|_{L^2}.
    \end{align}
    This implies the convergence of $(\mathcal{Q}_{\cdot}(H_n))_{n \in \mathbb N}$ to $\mathcal Q_{\cdot}(H)$ in $\mathfrak H$. Moreover, we get the same bound \eqref{eq:Qbound} for the limit which, through stationary increments and Kolmogorov's continuity theorem, implies the almost sure pathwise continuity of $\mathcal Q_{\cdot}(H)$.
    Combining those convergence results we obtain that $\tilde M_{\cdot} (H)$ is martingale with continuous trajectories and quadratic variation given by $t\mapsto \mathfrak c t\|H\|_{L^2}$ with the equality in $\mathfrak H$
    \begin{align} \label{eq:stildemartingale}
        \tilde M_t(H) = \mathcal Y_t(H) -\mathcal Y_0(H) - \mathfrak a\int_0^t\mathcal Y_s(H'') ds + \mathfrak b\mathcal  Q_t(H).
    \end{align}
    Additionally we get the almost sure pathwise continuity of $\mathcal Y_{\cdot}(H)$.\\
    Finally this argument can be repeated analogously for the backwards martingale of $\mathcal Y_{T-\cdot}(H)$, yielding condition $iv)$ of Definition \ref{def:SBEenergysolutions}.
\end{proof}

\begin{remark} 
\label{rem:SBEextensions}
Here we collect some observations about our previous results. 
\begin{enumerate}
\item [1.] We showed generally that for any test function $H$ there exists an extension of $\mathcal Y$ (as in Lemma \ref{lem:L2extension}) that solves SBE($H$) (see Remark \eqref{rem:continuity&whitenoise}) if there exists a sequence $(H_n)_{n\in\N}$ converging in $\mathcal{H}^2(D)$ (i.e. $H_n^{(k)}\xrightarrow[]{L^2}H^{(k)}$ for $k=0,1,2$) and such that $\mathcal Y$ solves SBE($(H_n)_{n\in\N}$). More precisely the convergence of $(\|H_n\|_{L^2})_{n \in \mathbb N}$ controls the one of $(\mathcal Y(H_n))_{n \in \mathbb N}$, the convergence of $(\|H'_n\|_{L^2})_{n \in \mathbb N}$ controls that of the martingale sequence $(M(H_n))_{n \in \mathbb N}$ and $(\mathcal Q(H_n))_{n \in \mathbb N}$. Finally we need the convergence of $(\|H''_n\|_{L^2})_{n \in \mathbb N}$ to control the sequence $(\int_0^{\cdot}\mathcal Y_s(H_n'')ds)_{n \in \mathbb N}$. 

\item [2.]The reason that we can not directly extend $\mathcal Y$ to $\Sdir$ with the first remark by approximating $H\in \Sdir$ with $(H_k)_{k\in\N}\subset \mathcal S_0$ in $\mathcal H^2(D)$ is the same used to argue that $\mathcal S_0$ is not a core. In the proof of \ref{lem:Sdircore} we show that the closure of $\mathcal S_0$ w.r.t. the norm $\|H\|_{L^2} +\|H''\|_{L^2}$ is smaller than $\Sdir$. This is why a condition like (BC) is needed.
\end{enumerate}
\end{remark}

Below, we define some functions that will help us in approximating elements in $\Sdir$ by elements of $\mathcal S$.

\begin{definition} \label{def:psialphabeta}
    Let $a:\R\mapsto\R$ be given by 
    \begin{align*}
        a(u) = ce^{-\frac{1}{u(1-u)}}\mathbbm{1}_{(0,1)}(u), \ \text{ where } \ c := \left( \int_0^1e^{-\frac{1}{u(1-u)}} du\right)^{-1}.
    \end{align*}
    Then we define $\phi:\R\mapsto\R$ as the function $\phi(u) := \left(1-\int_0^u a(v)\ dv\right) \mathbbm{1}_{(0,1)}(u)$.
    Finally for $\beta \in (0,1)$ and $\alpha>(1-\beta)^{-1}$ we set $\psi_{\alpha,\beta}(u) := u\phi(u)$ and $\check\psi_{\alpha,\beta}(u):= \psi_{\alpha,\beta}(-u)$ for $u\in\R$. 
\end{definition}

We note that for $u\in (-\infty,0]$ we have $\psi_{\alpha,\beta}(u) = 0$, for $u\in[0,\beta]:\psi_{\alpha,\beta}(u) = u$ and for $u\in[(\beta + \frac{1}{\alpha})\wedge 1],\infty): \psi_{\alpha,\beta}(u) = 0$.

\begin{lemma}\label{lem:psiextension}
    Assume (BC). Fix $\beta\in(0,1)$ and $\alpha >(1-\beta)^{-1}$. Then $\mathcal Y$ is a solution to SBE($\{\psi_{\alpha,\beta},\check\psi_{\alpha,\beta}\}$).
\end{lemma}

We observe that it is only at this point that we need to use the condition (BC) explicitly. In all the rest of this article this condition is  not necessary.

\begin{proof}
    In analogy to Lemma 5.5 in \cite{BGJS22}, we omit the fixed indexes $\alpha,\beta$ and only show the result for $\psi$. The strategy of this proof consists in approximating $\psi$ by $(\psi_{\varepsilon})_{\varepsilon>0}$ that solve SBE ($(\psi_\varepsilon)_{\varepsilon>0})$, then controlling the $L^2$-norms of the derivative of $\psi_{\varepsilon}$'s on $D$ will yield the lemma. We define $\psi_{\varepsilon} := h_{\varepsilon}\phi$, where $h_{\varepsilon}$ is the Tanaka function
    \begin{align*}
        h_{\varepsilon}(u) := \left\{ \begin{array}{ll}
        0 & \text{for } u\in(-\infty,0] \\
        \frac{u^2}{2\varepsilon}  & \text{for } u\in[0,\varepsilon] \\
        u-\frac{\varepsilon}{2}  & \text{for } u\in[\varepsilon,{\infty})
        \end{array}\right..
    \end{align*}
    It is easy to show that $\psi_{\varepsilon}\in \tilde{\mathcal S}$ and thus Lemma \ref{lem:tildeSextension} tells us that the continuous extension of $\mathcal Y$ is a solution of SBE($(\psi_{\varepsilon})_{\varepsilon>0}$). Furthermore, $\lim_{\varepsilon\to 0}\psi_{\varepsilon}^{(k)} = \psi^{(k)}$ in $L^2(D)$ for $k\in\{0,1\}$. 
    Then in analogy to the proof of Lemma \ref{lem:tildeSextension} it remains to prove the convergence of $\int_0^{\cdot} \mathcal Y_s(\psi_{\varepsilon}'')ds$ in $\mathfrak B$ as $\varepsilon \to 0$. Observe that the weak derivatives on $D$ satisfy
    \begin{align*}
        \psi_{\varepsilon}''  = h_{\varepsilon}'' \phi + 2h_{\varepsilon}' \phi' + h_{\varepsilon}\phi'',
    \end{align*}
    and that $2h_{\varepsilon}' \phi' + h_{\varepsilon}\phi'' \to 2\phi' + u \phi''= \psi''$ as  
    $\varepsilon \to 0$ in $L^2$ (note that the fact that $\phi$ is discontinuous at $0$ is the reason why we have to consider weak derivatives on $D$). So we want the term $h_{\varepsilon}'' \phi$ to vanish. Now note that since $h_{\varepsilon}'' \phi = \iota_{\varepsilon}^0 $, then  from condition (BC)  we get
    \begin{align*}
        \lim_{\varepsilon\to 0 } \mathbb{E}\left[\sup_{t\in[0,T]} \left(\int_0^t \mathcal Y_s(h_{\varepsilon}'' \phi) \ ds\right)^2 \right] = 0.
    \end{align*}
    This means that$\int_0^{\cdot} \mathcal Y_s(\psi_{\varepsilon}'')ds \to \int_0^{\cdot} \mathcal Y_s(\psi'') ds$  in $\mathfrak B$ and thus finishes the proof. For $\check\psi$ we simply use the other equality of (BC) and $h_{\varepsilon}''(-\cdot) \phi(-\cdot) = \iota_{\varepsilon}^0(-\cdot) $. 
\end{proof}

Now we are in position to prove the extension result.

\begin{proof}[Proof of Theorem \ref{thm:sdirextension}]
  We first decompose in a suitable way functions in $\mathcal S_{\textrm{Dir}}$ in terms of functions for which we know that the martingale problem holds. To that end, for all $\beta\in(0,1)$ and $\alpha>(1-\beta)^{-1}$ we make the splitting
    \begin{align*}
        H = (H - H(0^+)\psi_{\alpha,\beta} -H(0^-)\check\psi_{\alpha,\beta}) + H(0^+)\psi_{\alpha,\beta} -H(0^-)\check\psi_{\alpha,\beta}
    \end{align*}
    where one easily checks that defining $G^{\alpha,\beta}:=(H - H(0^+)\psi_{\alpha,\beta} -H(0^-)\check\psi_{\alpha,\beta})$, then  $G^{\alpha,\beta}\in \tilde{\mathcal S}$. Therefore by Lemma \ref{lem:tildeSextension} and Lemma \ref{lem:psiextension} and the linearity of the involved martingales we obtain existence of the continuous martingale $\tilde M^{\alpha,\beta}(H) := H(0^+)\tilde M(\psi^{\alpha,\beta}) + H(0^-)\tilde M(\check \psi^{\alpha,\beta}) + \tilde M(G^{\alpha,\beta})$ satisfying
    \begin{align}\label{eq:alphabetamartingale}
        \tilde M^{\alpha,\beta}_t(H) = \mathcal Y_t(H) -\mathcal Y_0(H) - \mathfrak a\int_0^t\mathcal Y_s(H'') ds + \mathfrak b\mathcal  Q_t(H).
    \end{align} 
    It is also clear that condition  $i)$ and $ii)$ of Definition \ref{def:SBEenergysolutions} hold.
    Now it remains to prove the convergence in $\mathfrak B$ as $\alpha$ goes to infinity of $\tilde M^{\alpha,\beta}(H)$ with $\beta=\alpha^{-1}$ to the martingale $\tilde M(H)$ with quadratic variation given by $t\mapsto \mathfrak ct\|H'\|_{L^2}$. This is done just like in Corollary 5.6 of \cite{BGJS22}, namely, we employ  Doob's inequality and then we control the norms $\|(G^{\alpha,\alpha^{-1}}-H)'\|_{L^2}$. The same procedure also works for the extension of the backwards martingale to obtain condition $iv)$. This proves the result.
\end{proof}

\appendix

\section{Sobolev spaces and core of the Laplacian.} \label{sec:sobolev}
In this section we recall the definitions of Sobolev-spaces and some facts on the core of the Laplacian with Dirichlet boundary conditions on $D=\R\setminus\{0\}$.
\begin{definition} We write $\mathcal H^k(D)$ for $k=1,2$ to be the space of function $H \in L^2(D)$ with the $k$-th derivative $H^{(k)}\in L^2(D)$  given as the unique function $H^{(k)}$ satisfying
\begin{align*}
    \forall G\in C^{\infty}_c(D): \int_{D}H(u) G^{(k)}(u) \ du = (-1)^k \int_{D} H^{(k)}(u) G(u) \ du
\end{align*}
For a function $H\in\mathcal H^k(D)$ we define the  semi-norms $\|H\|_{\dot{\mathcal H}^{k}}:= \|H^{(k)}\|_{L^2}$ and the full norms on the space given by
\begin{align*}
    \|H\|_{k}^2 := \sum_{i = 1}^k\|H\|_{\dot{\mathcal H}^i}.
\end{align*}
We also define the space $\mathcal H^1_0(D)$ as the closure of $\Sdir$ 
with respect to $\|\cdot\|_k$. 
Furthermore we define the fractional Sobolev norms for $\alpha\in (0,1)$ by $\|H\|_{\mathcal H^{\alpha}}:= \|H\|_{L^2} + \|H\|_{\dot{\mathcal H}^{\alpha}}$ setting 
\begin{align*}
    \|H\|_{\dot{\mathcal H}^{\alpha}} := \int_{(0,\infty)^2} \frac{|H(v)-H(u)|^2}{|v-u|^{1+\alpha}} dudv + \int_{(-\infty,0)^2} \frac{|H(v)-H(u)|^2}{|v-u|^{1+\alpha}} dudv.
\end{align*}
\end{definition}
\begin{remark} Note that by considering the weak derivatives on $D$ we allow for discontinuities of the weak-derivatives at $0$. Therefore $\Sdir\subset \mathcal H^2(D)$, but not every function in $\Sdir$ is 2-times (weakly-) differentiable on $\R$, e.g. a function $H\in\Sdir$ such that for $u\in(-1,1): H(u) = |u|$.
\end{remark}

We state a technical lemma which is directly adapted from \cite[Lemma A.4]{GPP24}, and was used in the uniqueness result. Recall  \eqref{eq:checkrho}.

\begin{lemma}
\label{lem:rhoregularity}
    Let $0\leq \alpha \leq \beta \leq 1$. Then
    \begin{align}
    \label{eq:rhobounded} 
        \|\check\rho_\varepsilon(G)\|_{\mathcal H^{\beta}} &\lesssim \varepsilon^{-(\beta-\alpha)} \|G\|_{\mathcal H^{\alpha}},\\
        \label{eq:rhoapprox}
        \|G-\check\rho_\varepsilon(G)\|_{\mathcal H^{\alpha}} &\lesssim \varepsilon^{\beta-\alpha} \|G\|_{\mathcal H^{\beta}},
    \end{align}
    and the same bounds also hold if we replace the norms $\Vert \cdot \Vert_{\mathcal H^{\alpha}}$ and $\Vert \cdot \Vert_{\mathcal H^{\beta}}$ by the semi-norms $\Vert \cdot \Vert_{\dot{\mathcal H}^{\alpha}}$ and $\Vert \cdot \Vert_{\dot{\mathcal H}^{\beta}}$.
\end{lemma}

\begin{proof}
We start by proving inequality \eqref{eq:rhobounded}, which will be done by interpolating the edge cases (when $\alpha =\beta = 0, 0 = \alpha<\beta = 1$ and $\alpha = \beta =1$).
    For $\alpha = 0 = \beta$ we use the fact that $\check \iota_{\varepsilon}$ is the density of a probability measure on $\mathbb R$  with support in $(0,\varepsilon]$ and that $|\chi_{\varepsilon}|\leq 1$. With Jensen's inequality and Fubini's theorem we get 
    \begin{align*}
        \|\check\rho_{\varepsilon}G\|_{L^2}^2 &= \int_D \chi_{\varepsilon}(u)^2 \left( \int_D \check\iota_\varepsilon (u-v) G(v)dv\right)^2du \leq \int_D  \int_D \varepsilon^{-1}{\mathbf 1}_{(v,v+\varepsilon]}(u) G(v)^2dvdu = \| G\|_{L^2}^2.
    \end{align*}
For $\beta = 1,\alpha = 0$ and $G\in L^2(D)$ we use the fact that $\partial_u(\check\iota_{\varepsilon}\ast G)(u) = \varepsilon^{-1}(G(u+\varepsilon) - G(u))$ to get
    \begin{align*}
        \|\check \rho_{\varepsilon}(G)\|_{\dot{\mathcal H}^1_0}^2 &= \int_{D} |\partial_u\check\rho_{\varepsilon}(G)(u)|^2 du\\
        &  \lesssim \int_{D}|\partial_u \chi_{\varepsilon}(u)|^2 |\check\iota_{\varepsilon}\ast G(u)|^2 du + \int_{D} |\chi_{\varepsilon}(u)|^2 |\partial_u(\check\iota_{\varepsilon}\ast G)(u)|^2 du \\
        & \lesssim \varepsilon^{-2} \|\check\iota_{\varepsilon}\ast G\|_{L^2}^2 + \|\partial_u \check \iota_{\varepsilon}\ast G\|^2_{L^2} \lesssim \varepsilon^{-2} \|G\|_{L^2}^2.
    \end{align*}
    Finally for $\alpha = 1 = \beta$, we assume $G\in H^1_0(D)$. Then
    \begin{align*}
        \int_{D}|\partial_u\check\rho_{\varepsilon}(G)(u)|^2 du &\lesssim \int_{D}|\partial_u\chi_{\varepsilon}(u)|^2 |\check \iota_{\varepsilon}\ast G(u)|^2 du + \int_{D} |\chi_{\varepsilon}(u)|^2 |\partial_u(\check \iota_{\varepsilon}\ast G)(u)|^2 du \\
        &\lesssim \int_{-\varepsilon}^{\varepsilon} \varepsilon^{-2} \sup_{v\in[-2\varepsilon,2\varepsilon]} |G(v)|^2 du + \|\partial_u G\|^2_{L^2} \lesssim \|\partial_u G\|_{L^2}^2,
    \end{align*}
    where we used the Cauchy-Schwarz's inequality with 
    \begin{align} \label{eq:supiotaG}
       \sup_{u\in[-\varepsilon,\varepsilon]}|\check \iota_{\varepsilon}\ast G(u)| &= \sup_{u\in[-\varepsilon,\varepsilon]} \left|\int_u^{u+\varepsilon}\varepsilon^{-1} G(v)dv\right| \leq \sup_{u\in[-2\varepsilon,2\varepsilon]} |G(u)| \nonumber \\
       & = \sup_{u\in[-2\varepsilon,2\varepsilon]} \left|\int_0^{u}\partial_uG(v) dv \right|\lesssim  \varepsilon^{\frac12}\|\partial_u G\|_{L^2}.
    \end{align}
Now it is possible to interpolate between the three shown bounds like in \cite[Lemma A.4]{GPP24} to get \eqref{eq:rhobounded}.

Now we proceed similarly to prove \eqref{eq:rhoapprox} and prove the bound for the edge cases  $0\leq \alpha\leq \beta\leq 1$. For $0= \alpha  = \beta$ and for $\alpha = \beta = 1$ we can directly apply the triangle inequality and use \eqref{eq:rhobounded}. Then for $0 = \alpha, \beta=1$, for any $G\in L^2(D)$, we use the triangle inequality as:
    \begin{align}\label{eq:L2H1approxbound}
        \|\check\rho_{\varepsilon}(G) - G\|_{L^2}^2 &\lesssim \|(\chi_{\varepsilon} -1)\check \iota_{\varepsilon}\ast G\|^2_{L^2} + \|\check \iota_{\varepsilon} \ast G -G\| ^2_{L^2}.
    \end{align}
    For the first term on the right hand side of \eqref{eq:L2H1approxbound} we use an $L^{\infty}$-bound and the fact that  $\|\chi_\varepsilon -1\|^2_{L^2} \lesssim\varepsilon $, to get
    \begin{align*}
        \|(\chi_{\varepsilon} -1)\check \iota_{\varepsilon}\ast G\|^2_{L^2} \lesssim \|\chi_{\varepsilon}-1\|_{L^2}^2 \|\mathbbm{1}_{[-\varepsilon,\varepsilon]}(\check \iota_{\varepsilon}\ast G)\|^2_{L^{\infty}} \lesssim \varepsilon^2\|\partial_w G\|^2_{L^2}.
    \end{align*}
    The second term in \eqref{eq:L2H1approxbound} can be bounded as follows:
    \begin{align*}
        \|\check \iota_\varepsilon\ast G -G\|^2_{L^2}  &\leq \int_{D} \left(\int_{D} \varepsilon^{-1} \mathbbm{1}_{[0,\varepsilon]}(v-u)(G(v) - G(u)) dv \right)^2 du \\
        &\leq \int_{D} \int_{D} \varepsilon^{-1}\mathbbm 1_{[0,\varepsilon]}(v-u) \left( \int_u^v \partial_w G(w) dw\right)^2 dv du \\
        & \leq \int_{D} \int_{D} \int_{D} \varepsilon^{-1}\mathbbm 1_{[0,\varepsilon]}(v-u) \mathbbm 1_{(u,v)}(w) |\partial_w G(w)|^2 dw dv du \\
        & \leq \int_{D} \int_{D} \int_{D} \varepsilon^{-1}\mathbbm 1_{[0,\varepsilon]}(v-u) \mathbbm 1_{(u,v)}(w) |\partial_w G(w)|^2 dw dv du \lesssim \varepsilon^2\|\partial_w G\|^2_{L^2}.
    \end{align*}
    Combining all this we obtain the required estimate $\|\check\rho_{\varepsilon}(G) - G\|_{L^2}^2 \lesssim \varepsilon^2\|\partial_w G\|^2_{L^2}$.
    We then again finish the proof by doing the interpolation as in \cite{GPP24}.
\end{proof}

We then introduce the notion of core which is essential to differentiate the role of $\Sdir$ and $\mathcal S_0$ regarding the uniqueness with boundary conditions.

\begin{definition}
    Let $E$ be a Banach space and $A:\mathcal D(A) \to E$ a linear operator with dense domain $\mathcal D(A)\subseteq E$. We then define a core $\mathcal C\subseteq \mathcal D(A)$ of $A$ to be a space such that for all $f\in \mathcal D(A)$ there exists $(f_k)_{k\in\N}$ in $\mathcal C$ such that $f_k\to f$  and $Af_k \to Af$ in $E$ as $k\to\infty$ (in other words, $\mathcal{C}$ is  dense in $\mathcal D(A)$ with respect to the norm $\|H\|_E +\|AH\|_E$). 
\end{definition}

\begin{lemma} 
\label{lem:Sdircore}
    For the Dirichlet-Laplacian $\Delta: \mathcal D(\Delta) \to L^2(D)$ with $\mathcal D(\Delta) = \{H\in \mathcal H^1_0(D) \; ; \;  H'' \in L^2(D)\}$, $\Sdir$ is a core of $\mathcal{D}(\Delta)$, but $\mathcal S_0$ is not.
\end{lemma}

\begin{proof}
Let us first show that $\mathcal S_0$ is not a core by contradiction. Let us consider a function $H\in \Sdir \subseteq\mathcal D(\Delta)$ such that $H (u) = u$ for $u\in(0,1)$ (and thus $H'' (u) =0$). Since $\mathcal S_0$ is assumed to be a core, there exists a sequence $(H_k)_{k \in \mathbb N} $ of functions $H_k\in \mathcal S_0$ such that $\lim_{k \to \infty}  \big\{\Vert H_k -H \Vert_{L^2 (D)} + \Vert H_k^{''} -H^{''} \Vert_{L^2(D)} \big\} =0$. Since $\Vert \cdot \Vert_{L^2 ((0,1))} \le \Vert \cdot \Vert_{L^2 (D)}$, then $\lim_{k \to \infty}  \big\{\Vert H_k -H \Vert_{L^2 ((0,1))} + \Vert H_k^{''} -H^{''} \Vert_{L^2((0,1))} \big\} =0$. Moreover, since $H_k(0) = H_k'(0) = 0$ we have that, for $u\in(0,1)$,
    \begin{align*}
      \vert   H_k(u)  \vert = \left\vert \int_0^u\int_0^v H''_k(w) \ dw dv  \right\vert \leq \left(\int_0^u \int_0^v H''_k(w)^2\ dw dv\right)^{\frac{1}{2}} \leq \|H''_k\|_{L^2 ((0,1))} ,
    \end{align*}
    so that 
    \begin{equation*}
        0< \Vert H \Vert_{L^2((0,1))} = \lim_{k \to \infty} \Vert H_k \Vert_{L^2 ((0,1))} \le \lim_{k \to \infty} \Vert H_k^{''} -H^{''} \Vert_{L^2 ((0,1))} = 0
    \end{equation*}
(recall that $H'' = 0$ in $(0,1)$). This gives a contradiction and shows that $\mathcal S_0$ is not a core for $D(\Delta)$.  
\bigskip

Let us now show that $\Sdir$ is a core. Let us consider a function $H \in D(\Delta)$ and denote by $H_+$ (resp. $H_-$) the restriction of $H$ to $(0,+\infty)$ (resp. $(-\infty,0)$). By using a standard mollifier (i.e. a smooth, compactly supported approximation to the identity) to convolve  $H_+$ and $H_-$, we get two smooth approximation sequences $(H_{k,\pm})_{k \in \mathbb N}$ of $H_{\pm}$ in the sense that
\begin{equation*}
   \lim_{k \to \infty}  \big\{\Vert H_{k,+} -H_+ \Vert_{L^2 ((0,+\infty))} + \Vert H_{k,+}^{''} -H_{+}^{''} \Vert_{L^2((0,+\infty))} \big\} =0 
\end{equation*}
and similarly for the approximation of $H_-$. We would like now to glue $H_{k,-}$ with $H_{k,+}$ in order to have an approximating sequence of $H$ for the graph norm in $D$ but the gluing has to be done in order that the resulting function is in $\Sdir$. The basic gluing will not necessarily satisfy this condition. Instead, we modify $H_{k,\pm}$ near $0$ so that the new function $\tilde H_{k,\pm}$ satisfies the following conditions:  it is still smooth and Schwartz like away from $0$, approximates $H_\pm$ in the graph norm and  it satisfies the boundary condition defining $\Sdir$. This can be done by subtracting from $H_{k,\pm}$ a carefully chosen smooth bump function near $0$ that carries the unwanted boundary values but has negligible $L^2$-norm of its second derivative as $k$ goes to infinity. The details of this last argument are left to the reader. This concludes the proof that $\Sdir$ is a core.         

\end{proof}

\section{Tools for fluctuations of time inhomogeneous Markov Processes} \label{sec:inhomogenoustools}
In this section, we generalise standard tools  for studying fluctuations of time homogeneous Markov processes to the time inhomogeneous setting. 
The two main ingredients are the well known Dynkin's formula (which can also be found in \cite{EK86}) and a variance bound of additive functionals of Markov processes due to Sethuraman \cite{S00} (see also \cite{CLO01})
that we reformulate in a more general context. Let us recall first these two fundamental results in the context of time homogeneous Markov processes.  Let $(X_t)_{t\ge 0}$ be  a Feller Markov process with state space $E$ and generator $L$ whose domain we denote by $\mathcal D(L)$.  We assume that there exists an invariant measure $\pi$ and we denote by $L^*$ the adjoint of $L$ in $L^2(\pi)$. We write $L=A+S$ where $A=\tfrac{1}{2} (L-L^*)$ and $S=\tfrac{1}{2}(L+L^*)$. Observe that $A$ is antisymmetric in $L^2 (\pi)$ while $S$ is symmetric in $L^2 (\pi)$. Moreover, we assume that  $S$ is the generator of a Markov process with respect to $\pi$.  

\bigskip

\textbf{Dynkin's formula:} If $f: E\to\mathbb R$ is a function in  $\mathcal D(L)$, then the process $(M_t(f))_{t \ge 0}$ defined by 
\begin{equation} 
\label{eq:dynkin1}
    M_t(f) := f(X_t) - f(X_0) - \int_0^t   Lf(X_s) ds
\end{equation}
is a martingale with respect to the natural filtration of the process $(\mathcal G_t)_{t\ge 0}$ where $\mathcal G_t=\sigma(X_s\; ; \;  s\in[0,t])$. If additionally $f^2\in \mathcal D(L)$, then the predictable quadratic variation of $M_t(f)$ is equal to 
\begin{equation} \label{eq:dynkin2} 
    \langle M(f)\rangle_t= \int_0^t L f^2(X_s) - 2 f(X_s)Lf(X_s) ds.
\end{equation}

\bigskip

\textbf{Sethuraman variance estimate:} Let $\mathcal C\subseteq L^2(\pi)$ be a core of $L$ and $L^{\ast}$. 
We first define the following norm for $f\in \mathcal C$
\begin{align*}
    \|f\|_{-1,\pi} := \sup_{g\in\mathcal C}\big\{ 2\langle f,g\rangle_{L^2} - \langle g, (-S)g\rangle_{\pi}\big\},
\end{align*}
and let $\dot{\mathcal H}^{-1}_\pi$ be the completion of $\mathcal C$ w.r.t. $\|\cdot\|_{-1,\pi}$ and the quotient over the equivalence relation between $f$ and $g$ exactly when $\|f-g\|_{-1,\pi} = 0$. Then for $f\in L^2(\pi)\cap \dot{\mathcal H}^{-1}_{\pi}$, we have, for any $T>0$, that
\begin{equation} \label{eq:generalvariancest}
 E_{\pi}\left[  \sup_{t\in[0,T]} \left( \int_0^t f(X_s) ds \right)^2 \right] \lesssim T\ \| f \|_{-1,\pi}^2.
\end{equation}
This inequality was first proved in \cite[Lemma 3.9]{S00} (see also  \cite[Chapter 2]{KLO12} for a more general context).

\bigskip

\textbf{Application to the time inhomogeneous case:} We turn now to see how we can generalise the previous results to the case of time inhomogeneous Markov processes. To that end, let $X:=(X_t)_{t\ge 0}$ be a time inhomogeneous Markov process with generator $L_t$, with invariant measure $\pi$ (which is time independent) and whose adjoint in $L^2(\pi)$ is denoted by $L^*_t$. Let $\mathcal C$ be the common core of both $L_t$ and $L^*_t$ for all times $t$. We assume that the initial state $X_0$ has  law $\pi$. It is standard to extend the process $X$ into a Markovian process $(X_t)_{t \in \mathbb R}$ defined for all times $t \in \mathbb R$, having the same probability transition kernels as $X$ for non-negative times and still having $\pi$ as an invariant measure. Now we associate to $(X_t)_{t\in \mathbb R}$ a new process $(Y_t)_{t\in \mathbb R}$ with state space $\mathcal E=\mathbb R\times E$ given by $Y_t=(t,X_t)$. It is well known that $(Y_t)_{t\in \mathbb R}$ is a time homogeneous Markov process on $\mathcal E$ with generator $\tilde L=\partial_t + L_t$ and invariant measure{\footnote{This is an invariant measure but not an invariant probability measure.}} $\tilde\pi = dt\otimes \pi$. Then, we can apply the previous two results to the process $Y$.
\medskip

If $f: \mathcal E\to\mathbb R$ is a function in  $\mathcal D(\tilde L)$, then with \eqref{eq:dynkin1} and $\tilde Lf(Y_s) = (\partial_s +L_s)f(s,X_s)$ we get that the process $(M_t (f))_{t \ge 0}$ defined by
\begin{equation} \label{eq:dynkin1inhom}
    M_t(f) := f(Y_t) - f(Y_0) - \int_0^t   (\partial_s + L_s)f(s,X_s) ds
\end{equation}
is a martingale with respect to the natural filtration of the process $(\mathcal F_t)_{t\ge 0}$ where $\mathcal F_t:=\sigma(Y_s\; ; \;  s\in[0,t])$ and this coincides with $(\mathcal G_t)_{t\ge 0}$. If additionally $f^2\in \mathcal D(\tilde L)$, then the  predictable quadratic variation of $M_t(f)$ is equal to 
\begin{equation} \label{eq:dynkin2inhom}
    \langle M(f)\rangle_t= \int_0^t L_sf^2(s,X_s) - 2 f(s,X_s)L_sf(s,X_s) ds.
\end{equation}

For the variance estimate we start by defining for $f \in L^2 (\tilde\pi)$ the $\mathcal H^{-1}_{\tilde \pi}$ semi-norm 
\begin{equation}
\Vert f \Vert_{-1,\tilde\pi}^2 := \sup_{g \in L^2(\tilde\pi)} \left\{ 2 \langle f, g\rangle_{\tilde\pi} - \langle g , (-\tilde S) g \rangle_{\tilde\pi} \right\},
\end{equation}
with $\tilde S =\tfrac{1}{2} (\tilde L + \tilde L^{\ast})$. We note that
\begin{align*}
    \|f\|^2_{-1,\tilde\pi} = \int_0^{\infty} \sup_{{g\in\mathcal C}} 2\langle f(t,\cdot),g\rangle_{\pi} -\langle g,(-S_t)g\rangle_{\pi} dt =: \int_0^{\infty} \|f(t,\cdot)\|_{-1,\pi,t}^2\ dt
\end{align*}
which, as an application of \eqref{eq:generalvariancest}, yields the estimate
\begin{equation}\label{eq:inhomvariance} 
 E_{\pi}\left[ \sup_{t\in[0,T]} \left( \int_0^t f(s,X_s) ds \right)^2 \right] \lesssim T\int_0^{\infty}\|f(t,\cdot)\|^2_{-1,\pi,t}\  dt.
\end{equation}

One special case, given in the following result, will be used to bound the heat bath generator.

\begin{corollary} \label{cor:h-1heatbath}
For any $n\in\N, 0\leq t \le T$ and $\xi\in \R^{\Z}$ let $\varphi_n(t,\xi):= -W'_{\beta,\lambda}(\xi_{\floor{c_nt}}) =  \big (\tfrac{\lambda}{\xi_{\floor{c_n t}}}-\beta\big)$. Then for the Markov process $\xi$ with its invariant measure $\nu_{\beta,\lambda}$ defined in \eqref{def:nu} we have the following bound: 
\begin{align*}
\mathbb E_{\nu_{\beta,\lambda}} \Big[ \sup_{t\in[0,T]} \Big( \int_0^t \varphi_n(s,\xi(sn^2))\ ds \Big)^2 \Big]\lesssim T^2n^{\delta-2}.
\end{align*}
\end{corollary}
\begin{proof}
    We apply \eqref{eq:inhomvariance} to the setting with $\pi =\nu_{\beta,\lambda}$ and $\tilde S = n^2(\gamma S+ n^{-\delta}B_{n,t})$, then it only remains to compute the norm $\|\varphi_n(s,\cdot)\|_{-1\nu_{\beta,\lambda},s}$.
    We have that
    \begin{equation} \label{eq:phinnorm}
    \begin{split}
        \Vert \varphi_n(s,\cdot) \Vert_{-1,n,s}^2 &=\sup_{g\in \mathcal C} \left\{ 2 \langle \varphi_n(s,\cdot), g \rangle_{\nu_{\beta,\lambda}} +  n^2 \left\langle g, (\gamma \mathcal S + n^{-\delta}B_{n,s}) g \right\rangle_{\nu_{\beta,\lambda}}  \right\}\\
        &\le \sup_{g\in \mathcal C} \left\{ 2 \langle \varphi_n(s,\cdot) , g \rangle_{\nu_{\beta,\lambda}} +  n^{2-\delta} \left\langle g, B_{n,s} g \right\rangle_{\nu_{\beta,\lambda}}  \right\}\\
        &= \sup_{g\in \mathcal C} \left\{ 2 \langle \varphi_n(s,\cdot) , g \rangle_{\nu_{\beta,\lambda}} -  n^2n^{-\delta}  \left\langle \left( \partial_{\xi_{\floor{c_n s}}} g \right)^2\right\rangle_{\nu_{\beta,\lambda}}  \right\}.
        \end{split}
        \end{equation}
    By definition of $\nu_{\beta,\lambda}$ and $\varphi_n$ we can integrate by parts and use Young's inequality
    \begin{equation*}
        \langle\varphi_n(s,\cdot),g\rangle_{\nu_{\beta,\lambda}} =  \langle1,\partial_{\xi_{\floor{c_ns}}}g\rangle_{\nu_{\beta,\lambda}} \leq \frac{1}{2K}+ \frac{K}{2}\|\partial_{\xi_{\floor{c_ns}}}g\|_{L^2}^2.
    \end{equation*}
    Plugging this into \eqref{eq:phinnorm} with $K= n^{2-\delta}$ yields
    \begin{align*}
        \|\varphi_n(s,\cdot)\|^2_{-1,\nu_{\beta,\lambda},s} \leq n^{\delta-2},
    \end{align*}
    independently of $s\in[0,T]$. This finishes the proof.
    \end{proof}

\section{Gaussian Analysis and Estimates} 
\label{app:gaussian}

In this section we develop the tools necessary to estimate the variance of functionals of the space {white-noise}. Building on the variance estimate \eqref{eq:generalvariancest} we will use the Wiener-chaos expansion to get a more precise estimate of the variational norm, as in \cite[Section 3]{GP18}.\\
Recall the definition in Section \ref{sec:spde} of the space white-noise on $L^2(D)$ with $D=\R\setminus\{0\}$, and the corresponding measure $\mu$ given in \eqref{eq:spacewn}. We assume that $\mathcal Y = (\mathcal Y_t)_{t\in[0,T]}$ is an energy solution of SBE($\Sdir$) and thus for every $t\in [0,T]$, $\mathcal Y_t$ can be extended uniquely to a space white-noise mapping functions of $L^2(D)$ to $L^2(\mu)$. 

\subsection{The variance estimate}
We will restate the variance estimate of Appendix \ref{sec:inhomogenoustools} in the setting of the solutions to SBE($\Sdir$) on $L^2(\mu)$.
We define the vector space $\mathcal{F}$ of {\textit{cylinder}} functions $F:\Sdir'\to \R$ such that
\begin{align*}
   \forall \mathcal Y \in \Sdir', \quad  F(\mathcal{Y}) =  f(\mathcal{Y}(\varphi_1),...,\mathcal{Y}(\varphi_d))
\end{align*}
where $d\in\mathbb{N}$,  $f\in C^2(\R^d)$ has at most polynomial growth of all derivatives up to order $2$, and $(\varphi_i)_{i=1, \ldots,d} \in\Sdir^d$ are arbitrary. The operator $L_0$ is defined by its action on such a cylinder function $F\in\mathcal F$ as 
\begin{align*}
    L_0F(\mathcal{Y}) := &\sum_{i=1}^d \partial_if(\mathcal{Y}(\varphi_1),...,\mathcal{Y}(\varphi_d))\mathcal{Y}(\Delta\varphi_i) +\sum_{i,j=1}^d \partial_{i,j}f(\mathcal{Y}(\varphi_1),...,\mathcal{Y}(\varphi_d))\langle \nabla\varphi_i,\nabla\varphi_j\rangle.
\end{align*}
Since, by Lemma \ref{lem:Sdircore}, $\Sdir$ is a core of the Dirichlet-Laplacian, it follows that the space $\mathcal F$ is a core of $L_0$ in $L^2(\mu)$, i.e. the domain $\mathcal D(L_0)$ is the closure of $\mathcal F$ with respect to the graph norm  $\|F\|_{L^2(\mu)} + \|L_0F\|_{L^2(\mu)}$ (the argument is the same as in \cite[Corollary 3.8]{GP18}). 

We define the following variational norms.
\begin{definition} \label{def:probsobolev}
    The semi-norms $\Vert \cdot\Vert_{\pm 1, \mu}$ are defined for $F\in\mathcal F$ by
    \begin{align*}
        \|F\|^2_{1,\mu} &:= 2\mathbb E\left[F(\mathcal{Y}_0)(-L_0)F(\mathcal{Y}_0)\right] = 2\langle F,(-L_0)F\rangle_{\mu}\ , \\
        \|F\|_{-1,\mu} &:= \sup_{G\in \mathcal F} \left\{  2\mathbb E\left[ F(\mathcal{Y}_0)G(\mathcal{Y}_0)\right] -\|G\|^2_{1,\mu} \right\}^{\frac{1}{2}}  = \sup_{G\in \mathcal F} \frac{\mathbb E\left[F(\mathcal{Y}_0)G(\mathcal{Y}_0)\right]}{\|G\|_{1,\mu}}  \ .
    \end{align*}
We introduce an equivalence relation on $\mathcal F$ by declaring that $F$ and $G$ are equivalent whenever $\|F-G\|_{1,\mu}=0$. In this way, $\|\cdot\|_{1,\mu}$ acts as a norms on the associated quotient space  of $\mathcal F$. We then define $\dot{\mathcal H}^{1}_{0,\mu}$ as the closure of the resulting 
    equivalence classes of elements of $\mathcal F$ with respect to $\|\cdot\|_{1,\mu}$. Similarly, $\dot{\mathcal H}^{-1}_{\mu}$ is defined as the closure of these classes under  $\|\cdot\|_{-1,\mu}$.
\end{definition}
Similarly to \cite[Corollary 3.5]{GP18} we have:
\begin{prop}
\label{prop:KVineq}
Let $F\in L^2(\mu)\cap\dot{\mathcal{H}}^{-1}_{\mu}$. Then 
\begin{align*}
    \mathbb E\left[\sup_{t\in[0,T]} \left|\int_0^t F(\mathcal{Y}_s) ds\right|^2\right] \lesssim T\|F\|^2_{-1,\mu}.
\end{align*}
\end{prop}

\subsection{Gaussian Analysis} 
We now describe Wiener-chaos theory (see \cite[Chapter 1]{N06} for more details) which will be useful to estimate the previous $\|\cdot\|_{-1,\mu}$-norm in terms of the variational formula for functions in $\Sdir$.

\medskip

For any $f \in L^2 (D^d)$, $d\in\N_0$, the $d$-th order Wiener-Itô integral $W_d:L^2(D^d)\to L^2(\mu)$ is defined precisely in \cite[Section 1.1.2]{N06} but we just need some of its properties and we omit its definition. It satisfies $W_d (\Pi f)=W_d(f)$, where $\Pi f$ denotes the symmetrisation of $f$, i.e. the real valued function on $D^d$ defined by
\begin{equation}
\forall u=(u_1,\ldots, u_d)  \in D^d, \quad [\Pi f] (u) = \frac{1}{d!} \sum_{\sigma \in \mathfrak S_d} f (u_{\sigma(1)}, \ldots, u_{\sigma(d)}) \ ,      
\end{equation}
where $\mathfrak S_d$ is the symmetric group on $\{1, \ldots, d\}$. Let $L_s^2 (D^d) = \{f \in L^2 (D^d) \; ; \; \Pi f =f\}$ be the space of symmetric square integrable functions on $L^2 (D^d)$. Then, for any $F\in L^2(\mu)$ there exists a unique (deterministic) sequence $(f_d)_{d\in\N_0} \in \prod_{d \in\N_0} L^2_s(D^d)$, such that the chaos-expansion of $F$ is given by the converging series in $L^2(\mu)$:
\begin{align*}
    F = \sum_{d=0}^{\infty} W_d(f_d) \ .
\end{align*}
Furthermore, it holds that $\frac{1}{\sqrt{d!}} W_d$ is an isometry between $L^2_s (D^d)$ and the $d$-th homogeneous chaos $\Big\{W_d(f_d)\; ; \;  f_d\in L^2_{s}(D^d)\Big\}$, which is also equal to the closure of the span of all random variables $\mathcal Y \mapsto H_d (\mathcal Y (\varphi)) \in \mathcal F$, where $\varphi\in \Sdir$ such that $\|\varphi\|_{L^2 (D)} =1$, and $H_d$ is the $d$-th Hermite polynomial{\footnote{$H_d(u):= (-1)^d e^{\frac{u^2}{2}} \partial_u^d [ e^{-\frac{u^2}{2}}]$, $u \in \mathbb R$. }.} In fact it holds that $H_d(\mathcal Y(\varphi)) = W_d(\varphi^{\otimes d})$.
\medskip

In order to approximate elements of $L^2_s(D^d)$ we define a version of $\Sdir$ with domain $D^d$:
\begin{equation}
\label{eq:sdird}
\begin{split}
    \Sdir(D^d) := \Big\{ g\in C^{\infty}(D^d) \; ; \;  & \forall u\in\partial D^d, \ \alpha = (2k_1,...,2k_d)\text{ for } k\in\mathbb{N}_0^d :\partial_{\alpha} g(u) = 0 , \\
    & \forall \ell\in\N :\|g\|_{\infty, \ell} = \sup_{i\in\N,j\in \N_0^d,\ i + |j| \leq \ell}\ \sup_{u\in D ^d}(1+|u|^i)\partial^j_u \, g(u) <\infty \Big\}.
\end{split}
\end{equation}
Now we can express the action of $L_0$ in terms of the Laplacian for elements of a homogeneous chaos. 
\begin{lemma} 
\label{lem:laplaceinvariance}
For all symmetric $f_d \in \Sdir (D^d)$ it holds $L_0 W_d(f_d) = W_d(\Delta f_d)$ in $L^2(\mu)$.
\end{lemma}

\begin{proof} 
The proof is the same as in \cite[Lemma 3.7]{GP18} where $\mathbb R$ there is replaced by $D$, and  the integration by parts also hold  here because functions in $\Sdir (D^d)$ vanish at $0$.
\end{proof}

We have seen that applying $L_0$ to elements of $L^2(\mu)$ is related to the regularity of elements of $L^2(D^d)$, thus we can rewrite the the norms $\|\cdot\|_{\pm 1,\mu}$ in terms of elements in $\Sdir(D^d)$.

\begin{lemma} \label{lem:sobolevchaos}
For symmetric $f\in\Sdir(D^d)$ (i.e. with $\Pi f = f$) we have
\begin{align*}
    \|W_d(f)\|_{1,\mu}^2 = 2 d!\|\nabla f\|_{L^2}^2 \text{ and } \|W_d(f)\|_{-1,\mu} = \sup_{g\in \Sdir(D^d), \Pi g=g}\frac{\left\langle f,g\right\rangle_{L^2 } }{2d!\|\nabla g\|_{L^2 }}.
\end{align*}
\end{lemma}
\begin{proof}
    Combining definition of $\|\cdot\|_{1,\mu}$, Lemma \ref{lem:laplaceinvariance}, the isometry of $\frac{1}{\sqrt{d !}}W_d$, as well as integration by parts, we obtain 
    \begin{align*}
        \|W_d(f)\|_{1,0}^2 & = -2\mathbb E \left[W_d(f)L_0W_d(f) \right]= -2\mathbb E\left[W_d(f)W_d(\Delta f)\right]\\
        & = -2d!\langle f,\Delta f\rangle_{L^2 }= 2 d!\|\nabla f\|_{L^2}^2.
    \end{align*}
For the dual norm we manipulate the norm in the same way and use the density of symmetric functions of $\Sdir(D^d)$ in $L^2_s(D^d)$, the orthogonality of the Wiener chaos expansion and the previously derived equality to obtain
    \begin{align*} 
         \left\|W_d\left( f\right)) \right\|_{-1,0} &=  \sup_{G\in \mathcal F}\frac{\mathbb E\left[ W_d\left(f\right)G \right]}{\|G\|_{1,0}}  =  \sup_{g\in L^2_s(D^d)}\frac{\mathbb E [W_d\left(f\right) W_d(g)]}{\|W_d(g)\|_{1,0}}\\
        & =  \sup_{g\in \Sdir(D^d), \Pi g=g}\frac{\mathbb E\left[ W_d\left(f\right) W_d(g)\right]}{\|W_d(g)\|_{1,0}} =  \sup_{g\in \Sdir(D^d), \Pi g=g}\frac{\left\langle f,g\right\rangle_{L^2} }{2d!\|\nabla g\|_{L^2}} .
    \end{align*}

\end{proof}

\section{Replacement Lemmas} \label{sec:replacementlemmas}

We prove two classical replacement lemmas for the Markov process defined in Section \ref{sec:modelandmainresults}.  The results resemble those for similar particle systems e.g. in \cite[Section 6]{BGJS22}. The first relies on the heat bath dynamics, while the second is based on the exchange noise.

\begin{lemma}\label{lem:boundaryreplacement}
For fixed $T>0$ and $n\in\N$ the following inequality holds
\begin{equation*}
    \ee{ \sup_{0\leq t\leq T} \left(\int_0^t \sqrt{n} \bar\xi_{\floor{c_n t}}(n^2s) ds \right)^2 } \lesssim n^{\delta -1}.
\end{equation*}
\end{lemma}
\begin{proof}
    We start by applying the general variance estimate \eqref{eq:inhomvariance} to obtain
    \begin{equation*}
    \begin{split}
    &\ee{\sup_{0\leq t\leq T} \parent{\int_0^t \sqrt{n}{\bar\xi_{\floor{c_nt}}(n^2 s)} \ ds}} \lesssim nT \int_0^T \| \bar{\xi}_{\floor{c_nt}} (n^2t)\|_{-1,\nu_{\beta,\lambda, t}} dt \\
    &\lesssim nT \int_0^T\sup_{g\in L^2(\nu_{\beta,\lambda})} 2\int g(\xi)\bar{\xi}_{\floor{c_nt}} \nu_{\beta,\lambda}(d\xi) + n^{2-\delta} \int g(\xi)  B_{n,t} g(\xi) \nu_{\beta,\lambda}(d\xi) \ dt \\
    \end{split}
    \end{equation*}
    where in the second step we dropped the negative semidefinite part of the exchange dynamics $S$.
    Then for $t\in[0,T]$ integrating by parts under  the measure  $\nu_{\beta,\lambda}$ we get
    \begin{align*}
        & \int g(\xi)\bar{\xi}_{\floor{c_nt}} \nu_{\beta,\lambda}(d\xi)+ n^{2-\delta} \int g(\xi)  B_{n,t} g(\xi) \nu_{\beta,\lambda}(d\xi) \\ 
        & = \int g(\xi)\partial_{\xi_{\floor{c_nt}}}\int_0^
        {\xi_{\floor{c_nt}}}(x-\rho) dx \ \nu_{\beta,\lambda}(d\xi)  - n^{2-\delta} \int  (\partial_{\xi_{\floor{c_nt}}} g(\xi))^2 \nu_{\beta,\lambda}(d\xi)  \\
        & = \int A(\xi_{\floor{c_nt}}) \partial_{\xi_{\floor{c_nt}}}g(\xi) \nu_{\beta,\lambda}(d\xi) - n^{2-\delta} \int  (\partial_{\xi_{\floor{c_nt}}} g(\xi))^2 \nu_{\beta,\lambda}(d\xi)\\
        &\leq \frac{1}{K}\int A(\xi_{\floor{c_nt}}) \nu_{\beta,\lambda}(d\xi) + (K-n^{\delta-2})\int (\partial_{\xi_{\floor{c_nt}}} g(\xi))^2 d\nu_{\beta,\lambda}(d\xi)
    \end{align*}
    where in the last step we used Young's inequality for some $K>0$ and to make the integration by parts applicable,  we choose $A:(0,\infty)\to \R$ such that
    \begin{equation*}
        A(u):= -e^{W_{\beta,\lambda}(u)} \int_0^{u} (v-\rho)e^{-W_{\beta,\lambda}(v)} dv.
    \end{equation*}
    From tail-bounds of the $\Gamma$-distribution it is not difficult to show that $\|A\|_{L^2(\nu_{\beta,\lambda}^{(1)})}< \infty$. Thus choosing $K = n^{2-\delta}$ it follows that for all $t\in[0,T]$:
    \begin{align*}
        \|\bar\xi_{\floor{c_nt}}(n^2t)\|_{-1,\nu_{\beta,\lambda},t} \lesssim n^{\delta-2}, 
    \end{align*}
    which finishes our proof.
\end{proof}

\begin{lemma} \label{lem:boxreplacement}
For fixed $T>0$, any $\ell,n\in\N$ and a measurable $z:[0,T]\to \Z$ (e.g. $z(t) = \floor{c_nt}$) it holds that 
\begin{equation} \label{eq:boxreplacement}
    \ee{\sup_{0\leq t\leq T} \left( \frac{1}{\ell}\int_0^t \sum_{x=z(t)+1 }^{z(t)+\ell}(\xi_{x}(n^2 s) -\xi_{z(t)}(n^2 s))\  ds \right)^2}  \lesssim \frac{\ell}{n^2}.
\end{equation}
\end{lemma}

\begin{proof} 
As a first step we rewrite the left hand side of \eqref{eq:boxreplacement} in terms of a telescopic sum, which is then equal to
\begin{equation*}
     \ee{\sup_{0\leq t\leq T} \parent{\frac{1}{\ell}\int_0^t \sum_{x=z(t)+1 }^{z(t)+\ell} \sum_{y = z(t)}^{x-1} (\xi_{y+1}(n^2 s) -\xi_y(n^2 s))\  ds }^2}.
\end{equation*}
Then from the variance estimate \eqref{eq:inhomvariance} we bound last display by 
\begin{equation} \label{eq:boxvariancebound}
T\int_0^T\sup_{g\in L^2(\nu_{\beta,\lambda})} \int \frac{2}{\ell} \sum_{x=z(t)+1 }^{z(t)+\ell} \sum_{y = z(t)}^{x-1} (\xi_{y+1} -\xi_y)g(\xi) \nu_{\beta,\lambda}(d\xi) + n^2  \langle g,(\gamma S+B_{n,t})g\rangle_{\nu_{\beta,\lambda}} \ dt.
\end{equation} 
We will now bound the integrand in $t \in [0,T]$ by a constant depending only on $n$ and $\ell$. Since we estimate the integrand for each $t\in[0,T]$, let us write to simplify $z(t) =z\in\Z$.  By making a change of variables in the first term inside the supremum, we can write the first integral as
\begin{equation*}
\begin{split}
    \int \frac{2}{\ell} \sum_{x=z+1 }^{z+\ell} &\sum_{y = z}^{x-1} (\xi_{y+1} -\xi_y) g(\xi) \nu_{\beta,\lambda}(d\xi) = - \int \frac{2}{\ell} \sum_{x=z+1 }^{z+\ell} \sum_{y = z}^{x-1} (\xi_{y+1} -\xi_y) g(\xi^{y,y+1}) \nu_{\beta,\lambda}(d\xi) \\
    & = \int \frac{1}{\ell} \sum_{x=z+1 }^{z+\ell} \sum_{y = z}^{x-1} (\xi_{y+1} -\xi_y) (g(\xi)- g(\xi^{y,y+1})) \nu_{\beta,\lambda}(d\xi).
\end{split}
\end{equation*}
The last display can be bounded with Young's inequality by
\begin{equation*}
    \int \frac{1}{2K\ell} \sum_{x=z+1 }^{z+\ell} \sum_{y = z}^{x-1} (\xi_{y+1} -\xi_y)^2 \ \nu_{\beta,\lambda}(d\xi) \ +\  \int \frac{K}{2\ell} \sum_{x=z+1 }^{z+\ell} \sum_{y = z}^{x-1} (g(\xi)- g(\xi^{y,y+1}))^2 \ \nu_{\beta,\lambda}(d\xi)
\end{equation*}
for any $K>0$. Then by choosing $K =2\gamma n^2 $ and dropping the (negative definite) operator $B_{n,t}$ we can bound the contribution of the symmetric generators
\begin{equation*}
\begin{split}
    &\int \frac{K}{2\ell} \sum_{x=z+1 }^{z+\ell} \sum_{y = z}^{x-1} (g(\xi)- g(\xi^{y,y+1}))^2 \nu_{\beta,\lambda}(d\xi) + n^2  \langle g,(\gamma S +\mathcal{B})g \rangle_{\nu_{\beta,\lambda}}  \\ 
    \leq  &\int \frac{K}{2\ell} \sum_{x=z+1 }^{z+\ell} \sum_{y = -{\infty}}^{\infty} (g(\xi)- g(\xi^{y,y+1}))^2 \nu_{\beta,\lambda}(d\xi) + \gamma n^2  \langle g, S g \rangle_{\nu_{\beta,\lambda}}\\
    \leq & \frac{K}{2}\langle g, (-S) g \rangle_{\nu_{\beta,\lambda}} + \gamma n^2  \langle g, S g \rangle_{\nu_{\beta,\lambda}} = 0.
\end{split}
\end{equation*} 
Thus coming back to the bound of \eqref{eq:boxvariancebound}, we can then compute the variance directly to obtain
\begin{equation*}
\begin{split}
    &\ee{\sup_{0\leq t\leq T} \parent{\frac{1}{\ell}\int_0^t \sum_{x=z(t)+1 }^{z(t)+\ell}(\xi_{x}(n^2 s) -\xi_\ell(n^2 s))\  ds }^2}\\
    \lesssim&   T^2\int \frac{1}{K\ell} \sum_{x=z(t)+1 }^{z(t)+\ell} \sum_{y = \ell}^{x-1} (\xi_{y+1} -\xi_y)^2 \ \nu_{\beta,\lambda}(d\xi)    \\
    \lesssim &  T^2\frac{\ell^2}{ n^2 \ell} \ee{\xi_{2}^2 + \xi_{1}^2 - 2\xi_{2}\xi_{1}} \lesssim \frac{\ell}{n^2}.
\end{split}
\end{equation*}
This finishes our proof.
\end{proof}

\section*{Acknowledgments}

The authors thank N. Perkowski for discussions related to the problem. The authors thank the warm hospitality of the University of Minho and Instituto Superior Técnico in Portugal, where parts of this work were developed.\\
ADj gratefully acknowledges  funding by the Deutsche Forschungsgemeinschaft (DFG, German Research Foundation) CRC/TRR 388 "Rough Analysis, Stochastic Dynamics and Related Fields“ – Project ID 516748464 and by Daimler and Benz Foundation as part of the scholarship program for junior professors and postdoctoral researchers.\\
LS is grateful for funding by DFG through the Berlin-Oxford IRTG 2544 “Stochastic Analysis in Interaction”. \\
The research of PG is partially funded by Fundação para a Ciência e Tecnologia (FCT), Portugal, through grant No. UID/4459/2025, as well as the ERC/FCT SAUL project.

\end{document}